\documentclass[11pt,letterpaper]{amsart}
\usepackage{mathrsfs}
\usepackage{amsmath,amssymb,amsfonts,amsthm}
\usepackage[utf8]{inputenc}
\usepackage[OT2,T1]{fontenc}
\DeclareSymbolFont{cyrletters}{OT2}{wncyr}{m}{n}
\DeclareMathSymbol{\Sha}{\mathalpha}{cyrletters}{"58}
\usepackage[centering,vmargin=1.3in,hmargin=1.2in]{geometry}

\usepackage{lmodern}
\usepackage{graphics}
\usepackage{mathptmx}
\usepackage{dsfont}
\usepackage{xcolor}
\usepackage{enumitem}

\usepackage{marginnote}
\usepackage{hyperref}

\usepackage{graphicx}
\usepackage[all,cmtip]{xy}

{\setbox0\hbox{$ $}}\fontdimen16\textfont2=\fontdimen17\textfont2
\entrymodifiers={+!!<0pt,\the\fontdimen22\textfont2>}
\SelectTips{cm}{11}

\mathcode`A="7041 \mathcode`B="7042 \mathcode`C="7043 \mathcode`D="7044
\mathcode`E="7045 \mathcode`F="7046 \mathcode`G="7047 \mathcode`H="7048
\mathcode`I="7049 \mathcode`J="704A \mathcode`K="704B \mathcode`L="704C
\mathcode`M="704D \mathcode`N="704E \mathcode`O="704F \mathcode`P="7050
\mathcode`Q="7051 \mathcode`R="7052 \mathcode`S="7053 \mathcode`T="7054
\mathcode`U="7055 \mathcode`V="7056 \mathcode`W="7057 \mathcode`X="7058
\mathcode`Y="7059 \mathcode`Z="705A

\newcommand{\BA}{{\mathbb {A}}}

\newcommand{\BC}{{\mathbb {C}}}

\newcommand{\BF}{{\mathbb {F}}}
\newcommand{\BG}{{\mathbb {G}}}

\newcommand{\BN}{{\mathbb {N}}}

\newcommand{\BP}{{\mathbb {P}}}
\newcommand{\BQ}{{\mathbb {Q}}}
\newcommand{\BR}{{\mathbb {R}}}

\newcommand{\BX}{{\mathbb {X}}}
\newcommand{\BY}{{\mathbb {Y}}}
\newcommand{\BZ}{{\mathbb {Z}}}
\newcommand{\CA}{{\mathcal {A}}}
\newcommand{\CB}{{\mathcal {B}}}
\newcommand{\CC}{{\mathcal {C}}}
\newcommand{\CD}{{\mathcal {D}}}
\newcommand{\CE}{{\mathcal {E}}}
\newcommand{\CF}{{\mathcal {F}}}
\newcommand{\CG}{{\mathcal {G}}}
\newcommand{\CH}{{\mathcal {H}}}
\newcommand{\CI}{{\mathcal {I}}}
\newcommand{\CJ}{{\mathcal {J}}}
\newcommand{\CK}{{\mathcal {K}}}
\newcommand{\CL}{{\mathcal {L}}}
\newcommand{\CM}{{\mathcal {M}}}
\newcommand{\CN}{{\mathcal {N}}}
\newcommand{\CO}{{\mathcal {O}}}
\newcommand{\CP}{{\mathcal {P}}}
\newcommand{\CQ}{{\mathcal {Q}}}
\newcommand{\CR}{{\mathcal {R}}}

\newcommand{\CT}{{\mathcal {T}}}

\newcommand{\CV}{{\mathcal {V}}}
\newcommand{\CW}{{\mathcal {W}}}
\newcommand{\CX}{{\mathcal {X}}}
\newcommand{\CY}{{\mathcal {Y}}}
\newcommand{\CZ}{{\mathcal {Z}}}
\newcommand{\RA}{{\mathbf {A}}}

\newcommand{\FB}{\mathfrak{B}}

\newcommand{\Hom}{\operatorname{Hom}}

\newcommand{\Pic}{\operatorname{Pic}}

\newcommand{\XX}{\underline{\mathbf{X}}}
\newcommand{\xx}{\underline{\mathbf{x}}}
\newcommand{\YY}{\underline{\mathbf{Y}}}
\newcommand{\yy}{\underline{\mathbf{y}}}
\newcommand{\ZZ}{\underline{\mathbf{Z}}}
\newcommand{\zz}{\underline{\mathbf{z}}}

\renewcommand{\aa}{\underline{\mathbf{a}}}
\newcommand{\bb}{\underline{\mathbf{b}}}

\newcommand{\dd}{\underline{\mathbf{d}}}
\newcommand{\ee}{\underline{\mathbf{e}}}
\newcommand{\ff}{\underline{\mathbf{f}}}
\newcommand{\uu}{\underline{\mathbf{u}}}
\newcommand{\one}{\underline{\mathbf{1}}}

\newcommand{\CTUT}{\mathcal{T}}
\newcommand{\FTUT}{\mathfrak{T}}
\newcommand{\CTNS}{\mathbf{T}_{\operatorname{NS}}}
\newcommand{\FTNS}{\mathfrak{T}_{\operatorname{NS}}}

\DeclareMathOperator*{\dprime}{\prime \prime}
\DeclareMathOperator{\bxi}{\mathbf{\xi}}

\input cyracc.def
\DeclareFontFamily{U}{russian}{}
\DeclareFontShape{U}{russian}{m}{n}
{ <5><6> wncyr5
	<7><8><9> wncyr7
	<10><10.95><12><14.4><17.28><20.74><24.88> wncyr10 }{}
\DeclareSymbolFont{Russian}{U}{russian}{m}{n}
\DeclareSymbolFontAlphabet{\mathcyr}{Russian}
\makeatletter
\let\@math@cyr\mathcyr
\renewcommand{\mathcyr}[1]{\@math@cyr{\cyracc #1}}
\makeatother

\theoremstyle{plain}

\numberwithin{equation}{section}

\newtheorem{theorem}{Theorem}[section] 
\newtheorem*{mainthm}{Main Theorem}
\newtheorem{lemma}[theorem]{Lemma}

\newtheorem{proposition}[theorem]{Proposition}

\newtheorem{corollary}[theorem]{Corollary}

\newtheorem*{condGS}{Condition (GS)}
\newtheorem*{condEE}{Condition (EE)}
\newtheorem*{condEEUT}{Condition (EEUT)}

\theoremstyle{definition}

\newtheorem{definition}[theorem]{Definition}
\newtheorem{principle}[theorem]{Principle}
\newtheorem{hypothesis}[theorem]{Setting}
\newtheorem{hypothesis1}[theorem]{Hypothesis}

\theoremstyle{remark}
\newtheorem{remark}[theorem]{Remark}
\newtheorem{remarks}[theorem]{Remarks}

\makeatletter\let\@wraptoccontribs\wraptoccontribs\makeatother

\makeatletter\def\@@and{et}\makeatother

\title[Equidistribution and geometric sieve for toric varieties]
{Equidistribution of rational points and the geometric sieve for toric varieties}
\author{Zhizhong Huang}
\address{State Key Laboratory of Mathematical Sciences, Academy of Mathematics and Systems Science, Chinese Academy of Sciences, Beijing 100190, China}
\email{zhizhong.huang@yahoo.com}


\begin{document}
			\begin{abstract}
			We prove the Manin--Peyre principle about the equidistribution of rational points on smooth proper split toric varieties. That is, rational points of bounded anticanonical height outside of the boundary divisors are equidistributed with respect to the Tamagawa measure in the adelic space. Based on a refinement of the universal torsor method due to Salberger, we provide asymptotic formulas with effective error terms for the number of rational points lying in an arbitrary adelic neighbourhood, and we also prove  an effective version of the Ekedahl-type geometric sieve. Several applications to sieving rational points  are derived. 
		\end{abstract}
						\maketitle	
								\begin{flushright}
							\textit{À Emmanuel, pour ton exigence}
						\end{flushright}					
	\tableofcontents

\section{Introduction}
\subsection{Empiricism and main theorem}
At the end of the 1980s, Manin and his collaborators (see \cite{Bat-Manin,F-M-T}) put forward conjectures about counting rational points of bounded height on projective algebraic varieties, together with geometric interpretations of the order of magnitude.  In \cite{Peyre,PeyreLille}, Peyre refines and generalises their conjectures to the following equidistribution principle of rational points in the adelic space. 
Let $V$ be an almost-Fano variety over a number field $k$ (see Setting \ref{def:almostFano}). Let $H:V(k)\to\BR_{>0}$ be the height function induced by a fixed choice of adelic norm associated to the anticanonical line bundle $K_V^{-1}$. 
Let $\omega^V$ be the Tamagawa measure on the adelic space $V(\RA_k)$ induced by the fixed adelic norm on $K_V^{-1}$ (see \eqref{eq:Tamagawameas}). For any fixed infinite subset $\CA\subset V(k)$, we define the normalised sequence of counting measures	\begin{equation}\label{eq:seqVAk}
	\left(\delta_{\CA_{\leqslant B}}:=\frac{1}{\alpha(V)\beta(V)B (\log B)^{r-1}}\sum_{P\in \CA:H(P)\leqslant B} \delta_P\right)_B
\end{equation} on the Brauer--Manin set $V(\RA_k)^{\operatorname{Br}}\subset V(\RA_k)$, where $\delta_P$ stands for the Dirac measure at $P$, $r:=\operatorname{rank}(\Pic(V))$, $\alpha(V)$ is related to the cone of pseudo-effective divisors $\overline{\operatorname{Eff}}(V)$, and $\beta(V)$ is the order of the first Galois cohomology group of the geometric Picard group $\Pic(\overline{V})$ (see \S\ref{se:heightsTamagawa}). We can now state  the equidistribution principle in the following form.
\begin{principle}[The Manin--Peyre equidistribution principle]\label{prin:Manin-Peyre}
	There exists a thin subset $M\subset V(k)$ (in the sense of Serre \cite[Definition 3.1.1]{Serre}) such that, as $B\to \infty$, the sequence of counting measures $(\delta_{(V(k)\setminus M)_{\leqslant B}})_B$ on $V(\RA_k)^{\operatorname{Br}}$
	converges weakly to the measure $\omega^V$. (In other words, for every continuous function $f$ on $V(\RA_k)^{\operatorname{Br}}$, we have $\int f \operatorname{d}\delta_{(V(k)\setminus M)_{\leqslant B}}\to \int f \operatorname{d}V(\RA_k)^{\operatorname{Br}}$ as $B\to \infty$.)
\end{principle} 
For an algebraic variety defined over a number field, the \emph{strong approximation} property (see \S\ref{se:adelicspaces}) measures the density of rational points inside of the adelic space. Principle \ref{prin:Manin-Peyre} is therefore a quantitative approach to strong approximation, as it addresses counting rational points satisfying given local conditions ordered by anticanonical height.
\begin{remarks}
	\hfill
	\begin{enumerate}
		\item In view of Principle \ref{prin:Manin-Peyre}, Manin's conjecture is the particular case where one applies the counting measures \eqref{eq:seqVAk} to the whole space $V(\RA_k)^{\operatorname{Br}}$.
			\item We would like to stress that it is usually necessary to restrict to $V(k)\setminus M$, the complement of such a thin subset $M$, otherwise the number of rational points in $M$ could dominate or have the same order of growth as those on the whole variety.
		\item It is pointed out by Peyre \cite[\S3]{Peyre} that the validity of Principle \ref{prin:Manin-Peyre} for \emph{at least one} choice of adelic norm is equivalent to its validity for \emph{any choice} of adelic norm on $K_V^{-1}$.
	\end{enumerate}
\end{remarks}

Since the removal of a codimension two closed subset does not introduce new cohomological obstruction to strong approximation (as first observed by Min\v{c}hev \cite{Minchev}), inspired by a question of Wittenberg \cite[Question 2.11]{Wittenberg} on \emph{arithmetic purity of strong approximation} (APSA)  and its recent progress on linear algebraic groups and their homogeneous spaces (see \cite{Wei,Cao-Liang-Xu,Cao-Huang1}), Principle \ref{prin:Manin-Peyre} admits the following natural extension to certain open subvarieties of any fixed almost-Fano variety, which is therefore a quantitative version of (APSA). We keep using the same notation as above.
\begin{principle}[Purity of equidistribution]\label{prin:purity}
	Suppose that $V$ satisfies Principle \ref{prin:Manin-Peyre}, upon removing a thin subset $M\subset V(k)$. Then the equidistribution principle (with the removal of $M$) also holds for any Zariski open dense subset of $V$ whose complement has codimension at least two. 
	That is, let $W\subset V$ be such an open subset. Then as $B\to\infty$, the normalised sequence of counting measures
	\begin{equation}\label{eq:seqW}
		\left(\delta_{(W(k)\setminus M)_{\leqslant B}}\right)_B
	\end{equation} on $W(\RA_k)^{\operatorname{Br}}$ converges weakly to the measure $\omega^V|_W$, where $\omega^V|_W$ is the restriction of the  Tamagawa measure $\omega^V$ to $W(\RA_k)^{\operatorname{Br}}$. 
\end{principle}

Let us now compare Principles \ref{prin:Manin-Peyre} and \ref{prin:purity} with \emph{the Hardy--Littlewood property} in setting of affine varieties first introduced in  the work of Borovoi and Rudnick \cite{Borovoi-Rudnick}, and its arithmetic purity generalisation ((APHL) for short) initiated in \cite[Question 1.1]{Cao-Huang2}. As seen from the discreteness of the embedding $k\hookrightarrow \RA_k$, in most cases, rational points on an affine variety $Y$ over $k$ can be dense only in $Y(\RA_k^S)\subset \prod_{\nu\notin S} Y(k_\nu)$, the adelic space \emph{off} a finite set of places $S$ (see \S\ref{se:adelicspaces}). Therefore, we expect that (see \cite[Definition 1.2]{Cao-Huang2}), in the case where $k=\BQ$, rational points in each connected component of $Y(\BR)$ should be equidistributed in $Y(\RA_\BQ^\infty)$ with respect to the finite part of the Tamagawa measure (together with a certain density function measuring the failure of strong approximation). So the biggest difference between (APHL) and Principle \ref{prin:purity} is that, in the latter we consider the (Brauer--Manin set of the) ``full'' adelic space $W(\RA_k)$, and we expect that rational points should be equidistributed with respect to the ``full'' Tamagawa measure.

The main goal of this article is to prove:
\begin{mainthm}
	Principles \ref{prin:Manin-Peyre} and \ref{prin:purity} hold  for any smooth proper split toric variety over $\BQ$ whose anticanonical line bundle is globally generated.
\end{mainthm}

Some further applications will be exhibited in \S\ref{se:applifibration} and \S\ref{se:primevalue}.
\subsection{Results on effective equidistribution and the geometric sieve}\label{se:results}
Let $X$ be a smooth proper split toric variety over $\BQ$ such that $K_X^{-1}$ is globally generated. We can associate a canonical toric adelic norm $(\|\cdot\|_{\operatorname{tor},\nu})_{\nu\in\mathfrak{M}(\BQ)}$ on $X$ (see \S\ref{se:toricparaheight}), which induces a toric height function $H_{\operatorname{tor}}:X(\BQ)\to\BR_{>0}$ and a toric Tamagawa measure $\omega_{\operatorname{tor}}^X=\omega_{\operatorname{tor},\infty}^X \otimes \omega_{\operatorname{tor},f}^X$ on $X(\RA_\BQ)=X(\BR)\times X(\RA_\BQ^\infty)$. Let $\CT_{O}\subset X$ be the open orbit, that is the complement of all boundary divisors.  

We now state our result on  \emph{effective equidistribution}  in terms of every ``standard'' adelic set $\CF=\CF_\infty\times\CF_{f}\subset X(\RA_\BQ)$, where $\CF_\infty\subset \mathfrak{F}$, $\mathfrak{F}$ being a specific family of $\omega_{\operatorname{tor},\infty}^X$-continuous real sets (see \S\ref{se:convergence}) which form a topological basis of $\CT_{O}(\BR)$ (see \eqref{eq:familyF}), and $\CF_{f}\subset X(\RA_\BQ^\infty)$ is non-empty open-closed\footnote{Here and after ``open-closed'' means both open and closed.}.

\begin{theorem}[Effective equidistribution]\label{thm:mainequidist}
	With the notation above, as $B\to\infty$, we have 
	\begin{multline*}
		\#\{P\in \CF\cap \CT_{O}(\BQ):H_{\operatorname{tor}}(P)\leqslant B\}\\=\alpha(X)\omega_{\operatorname{tor}}^X (\CF)B(\log B)^{r-1}+O_{\CF_\infty,\varepsilon}\left(\CL(\CF_f)^{r+2\dim X+\varepsilon}B(\log B)^{r-\frac{3}{2}+\varepsilon}\right), 
	\end{multline*}
	 for every $\varepsilon>0$, where the constant $\alpha(X)$ is defined by \eqref{eq:alphaV}, and $\CL(\CF_f)$ is the covering exponent of $\CF_{f}$ introduced in Definition \ref{def:LCEf}.
\end{theorem}
Theorem \ref{thm:mainequidist} confirms Principle \ref{prin:Manin-Peyre} for such toric varieties, since the family of all such standard adelic sets form a topological basis of $X(\RA_\BQ)\cap \CT_{O}$ (cf. \cite[Proposition 3.3]{Peyre}). As another key feature, it also provides an explicit uniform polynomial dependency on $\CL(\CF_{f})$ for the adelic neighbourhood $\CF$ in the error term.  
(See Remark \ref{rmk:effectiveCFinfty} for discussions about the explicit dependency on $\CF_\infty$.)

Our next result is concerned with the \emph{geometric sieve}, which shows that the number of rational points which specialise into a Zariski closed subset of codimension at least two modulo a certain large prime are negligible. Let $\CX$ be the canonical smooth proper integral model  of $X$ over $\BZ$, that is, $\CX$ is the toric scheme over $\BZ$ of the structural fan. We write $\widetilde{P}\in\CX(\BZ)$ the unique lift for every $P\in X(\BQ)$.
\begin{theorem}[Geometric sieve]\label{thm:maingeomsieve}
	Let $Z\subset X$ be Zariski closed of codimension at least two.
	Let $\CZ\subset \CX$ be the Zariski closure of $Z$ in $\CX$. Then uniformly for every $N>1$,
	\begin{multline*}
		\#\{P\in \CT_{O}(\BQ):H_{\operatorname{tor}}(P)\leqslant B, \text{there exists } p\geqslant N,\widetilde{P}~\operatorname{mod}~ p \in\CZ(\BF_p)\}\\\ll_Z \frac{B(\log B)^{r-1}}{N\log N}+B(\log B)^{r-2}\log\log B.
	\end{multline*}
\end{theorem}
Theorem \ref{thm:maingeomsieve} and Principle \ref{prin:Manin-Peyre} together imply Principle \ref{prin:purity}, as we will show in \S\ref{se:appgeomsieve}. 

Result of such sort goes back to Ekedahl \cite{Ekedahl}, who essentially established the geometric sieve for projective spaces. This was later vastly generalised and applied to a variety of counting problems, notably in works \cite[\S9]{Poonen-Stoll}, \cite[\S2.6]{Bhargava-Shankar}, \cite{Poonen}, \cite[\S3]{BBL}, \cite[\S3]{Cao-Huang2}, \cite[\S7]{Browning-HB}. To the author's knowledge, such effective results on the geometric sieve for higher dimensional varieties of degree $>1$ have so far only been obtained for quadratic hypersurfaces (see \cite[Theorem 1.6]{Cao-Huang2} and \cite[Theorems 1.1--1.3]{Browning-HB}). The (APSA) property (with Brauer--Manin obstruction) for general smooth proper toric varieties over general number fields is proven by Chen \cite[Theorem 1.7]{ChenS} (see also \cite[Theorem 1.1]{Wei} and \cite[Theorem 1.3]{Cao-Liang-Xu} where it is required to ignore at least one place). Our result provides a quantitative version of (APSA) in the split case.

The distribution of rational points on toric varieties have long remained central objects of study in arithmetic geometry. Manin's conjecture for toric varieties is established by Batyrev and Tschinkel in \cite{Bat-Tsch1,Bat-Tsch2} using harmonic analysis, and in the split case it is also independently proven by Salberger \cite{Salberger} using the \emph{universal torsor method}. In de la Bretèche's work  \cite{Breteche,Breteche2}, Salberger's approach is combined with an analytic method for multivariate Dirichlet series attached to certain arithmetic functions. The works \cite{Bat-Tsch1,Bat-Tsch2,Breteche} all provide asymptotic formulas with power-saving secondary terms: $$\#\{P\in \CT_{O}(\BQ):H(P)\leqslant B\}=B\CQ(\log B)+O(B^{1-\iota}),$$ where $H$ is a suitable anticanonical height function, $\CQ$ is a polynomial with real coefficients of degree $r-1$, and $0<\iota<1$.

Although Manin's conjecture has been proven for a number of varieties, it is worth emphasising that Principle \ref{prin:Manin-Peyre} has so far only been established in very rare cases (among which are certain complete intersections with many variables and generalised flag varieties, cf. \cite{Birch}, \cite[\S3.2 p. 227]{PeyreBeyond}, \cite[\S5--\S6]{Peyre}). It was announced in the unpublished preprint \cite[Theorem 3.10.3]{Chambert-Loir-Tsch} that Principle 1.1 holds (and more generally integral points are also equidistributed) for smooth proper toric varieties over general number fields. However, an counter-example to this preprint concerning integral points is found recently by Wilsch \cite{Wilsch}. Since the preprint of our article has appeared, Santens \cite[p.7~2nd paragraph]{Santens} establishes a corrected version.
Their method, apart from being capable of handling non-split toric varietes (whereas ours depends crucially on the splitness condition), can also deal with more general line bundles. However it does not seem evident to extract an effective error term with an explicit dependence on the given adelic neighbourhood following their method.  On reworking the approach of \cite{Breteche,Breteche2}, though not straightforward, it might be possible to deduce an asymptotic formula for any given \emph{finite} adelic neighbourhood with an effective error term, while dealing with an arbitrary \emph{real} neighbourhood still appears challenging. Besides, the deduction of Theorem \ref{thm:maingeomsieve} (and Principle \ref{prin:purity}) seems to go beyond both of these approaches.
\subsection{Applications to counting rational points arising from fibrations}\label{se:applifibration}
Motivated by a question of Serre \cite{Serre-Br}, and a series of results due to Loughran \emph{et al} \cite{Loughran,BBL,Loughran-Smeets,Browning-Loughran}, we address further applications of Theorems \ref{thm:mainequidist} and \ref{thm:maingeomsieve} concerning estimating effectively rational points of bounded height lying inside the adelic image of a fibration.

Recall that  a scheme $B$ of finite type over a perfect field $k$ is \emph{pseudo-split} \cite[Definition 1.3]{Browning-Loughran} if every element of the absolute Galois group $\operatorname{Gal}(\overline{k}/k)$ fixes an irreducible component of $B_{\overline{k}}$ of multiplicity one. This property is weaker than being \emph{split} \cite[Definition 0.1]{Skorobogatov}, as the latter means that the scheme $B$ contains an open subscheme over $k$ which is geometrically integral.

As before, $X$ denotes a smooth proper split toric variety over $\BQ$ such that $K_X^{-1}$ is globally generated. Let $f:Y\to X$ be a dominant proper morphism, where $Y$ is proper, smooth and geometrically integral over $\BQ$. The general philosophy is that, whether the set $f(Y(\RA_\BQ))\subset X(\RA_\BQ)$ has Tamagawa measure zero or not, depends crucially on the splitting behaviour of the fibres over codimension one points of the fibration.
We define the counting function
\begin{equation}\label{eq:countingfibration}
	\CN_{\operatorname{loc}}(f;B):=\#\{P\in\CT_{O}(\BQ):H_{\operatorname{tor}}(P)\leqslant B,P\in f(Y(\RA_\BQ))\}.
\end{equation}
The following theorem establishes quantitative estimates for the proportion of fibres that are everywhere locally soluble.
\begin{theorem}\label{thm:ratptsfibration}
	\hfill
	\begin{enumerate}
		\item Assume that $f$ is generically finite of degree $>1$. Then $$\CN_{\operatorname{loc}}(f;B)=O(B(\log B)^{r-1-\iota_{f}}),$$ where $0<\iota_{f}<1$ is a certain numerical constant depending on $f$.
	\item Assume that $f$ has geometrically integral generic fibre.
	\begin{enumerate}
		\item Assume that there exists at least one non-pseudo-split fibre over the codimension one points of $X$.  Then $$\CN_{\operatorname{loc}}(f;B)=O\left(\frac{B(\log B)^{r-1}}{(\log\log B)^{\varDelta(f)}}\right),$$ where $\varDelta(f)>0$ depends on  $f$.
		\item Assume that the fibre of $f$ above any codimension one point of $X$ is pseudo-split, and $Y(\RA_\BQ)\neq \varnothing$. Then  $\omega_{\operatorname{tor}}^X\left(f(Y(\RA_\BQ))\right)>0$ and \begin{equation}\label{eq:everylocdensity}
			\CN_{\operatorname{loc}}(f;B)\sim \alpha(X)\omega_{\operatorname{tor}}^X\left(f(Y(\RA_\BQ))\right)B(\log B)^{r-1}.
		\end{equation} 
\end{enumerate}
	\end{enumerate}
\end{theorem}
In the situation of Theorem \ref{thm:ratptsfibration} (1), the subset $f(Y(\BQ))\subset X(\BQ)$ is a thin set  \emph{of type II} (cf. \cite[\S3.1]{Serre}).
Theorem \ref{thm:ratptsfibration} (1) gives therefore a quantitative version of the fact that thin sets have Tamagawa measure zero (cf. Theorem \ref{thm:locsolfibre} (1)), and
its  power-saving on $\log B$ is a form of the effective \emph{Hilbert Irreducibility Theorem} with coefficients in a toric variety (cf. \cite[\S3.4]{Serre}). For any thin set  \emph{of type I} (i.e. a subset contained in $Z(\BQ)$ where $Z\subset X$ is a proper Zariski closed subset), see Corollary \ref{cor:subvar} for a (slightly) better estimate. 

Theorem \ref{thm:ratptsfibration} (2) generalises various results such as \cite[Theorem 1.1]{Loughran-Smeets} and \cite[Theorem 1.3]{BBL}, where the base of the fibration is the toric variety $\BP^n$.  We refer to \cite[\S1.1]{Loughran}, \cite[p. 1450--1451]{Loughran-Smeets}, \cite[p. 5762]{Browning-Loughran} for the definition of the constant $\varDelta(f)$ in Theorem \ref{thm:ratptsfibration} (2a). Theorem \ref{thm:ratptsfibration} (2b) also confirms a conjecture of Loughran \cite[Conjecture 1.7]{Loughran}.

\begin{remark} In general the existence of a non-split fibre over codimension one points is not sufficient to guarantee that $\CN_{\operatorname{loc}}(f;B)=o(B(\log B)^{r-1})$. See \cite[Example 5.9]{Loughran-Smeets} for an example studied by Colliot-Thélène, in which all fibres are everywhere locally soluble and pseudo-split, but two of them are non-split. This example fits into the situation of Theorem \ref{thm:ratptsfibration} (2b). 
\end{remark}

\subsection{Application to counting integral points in codimension one subsets}\label{se:primevalue}
We now state our second application of the effective equidistribution, inspired by the work of Nevo--Sarnak \cite[Theorems 1.1, 1.4, 1.6]{Nevo-Sarnak}. 

Let $X$ be a smooth proper split toric variety over $\BQ$ with globally generated $K_X^{-1}$ and let $Z\subset X$ be a reduced effective divisor. Let $\mathbf{X}$ be an integral model of $X$ over $\BZ$. We write $\widehat{P}\in\mathbf{X}(\BZ)$ for the lift of $P\in X(\BQ)$. Let $\mathbf{Z}$ be the Zariski closure of $Z$ in $\mathbf{X}$ and $\mathbf{U}:=\mathbf{X}\setminus\mathbf{Z}$. For $S$ a finite set of places of $\BQ$ containing the real place, let us consider $$\CN_{(1)}(Z;B):=\#\{P\in\CT_{O}(\BQ): H_{\operatorname{tor}}(P)\leqslant B,\widehat{P}\in \mathbf{U}(\BZ_S)\},$$ which counts the number of $S$-integral points of $\mathbf{U}$. If $\mathbf{Z}$ is the zero-locus of an integral global section $s$ of $\CO_{\mathbf{X}}(\mathbf{Z})$, we consider $$\CN_{(2)}(Z;B):=\#\{P\in\CT_{O}(\BQ): H_{\operatorname{tor}}(P)\leqslant B,|s(\widehat{P})| \text{ is a prime number}\}.$$
\begin{theorem}\label{thm:primevalue}
There exists $\theta>0$ such that  $$\CN_{(i)}(Z;B)=O\left(\frac{B(\log B)^{r-1}}{(\log\log B)^{
	\theta}}\right), \quad\text{for all }i=1,2.$$ where the implied constant depends on $Z$ and $\mathbf{X}$.\end{theorem}

An elementary example to which $\CN_{(2)}(Z;B)$ applies is the following. Suppose that $X$ embeds into a certain $\BP^m$ and $Z$ is the divisor  of a hyperplane section in $\CO_{\BP^m}(1)$, for example one of the coordinate functions of $\BP^m$. So Theorem \ref{thm:primevalue} shows in particular that the density of rational points, one of whose (integral) projective coordinates takes prime values, is negligible.
\begin{remarks}
	\hfill
	\begin{enumerate}
		\item More generally, an appropriate adaptation of our method can also deal with prime values of any non-zero rational function on $X$.
		\item Establishing non-trivial lower bounds for $\CN_{(2)}(Z;B)$ is directly related to the widely open Schinzel hypothesis and thus remains a very challenging problem. See e.g. \cite[p. 362]{Nevo-Sarnak} for conjectural formulas.
	\end{enumerate}
\end{remarks}
\subsection{Methods, strategy and other results} 
\subsubsection{Universal torsor method à la Salberger} 
The main ingredient of our approach to proving Theorems \ref{thm:mainequidist} and \ref{thm:maingeomsieve} is based on \cite{Salberger}. This method has the advantage that the parametrisation of rational points is made in a totally explicit way that strongly ties to the combinatorial data of the structural fan defining the toric variety. In \cite{Huang}, it has also been used by the author to study Diophantine approximation of rational points on split toric varieties. 

 We now sketch the basic strategy to prove Theorem \ref{thm:mainequidist}. For any smooth proper split toric variety $X$, there is a unique universal torsor $X_0$ over $X$ (under the Néron--Severi torus $\CTNS$) up to isomorphism, which is itself a quasi-affine toric variety and admits a canonical integral model $\CX_0$ over $\BZ$. The toric height function on $X$ can be lifted into $X_0$, so that counting rational points in $X(\BQ)$ is equivalent to counting integral points in $\CX_0(\BZ)$.  The  procedure of point counting in $\CX_0(\BZ)$ of bounded toric height with appropriate real conditions and congruence conditions is to compared it with integrals against the real toric adelic measure on $X_0$. Then the computation boils down to certain integrals on $\CTNS$ against the global $\CTNS$-invariant differential form.

\subsubsection{Uniform effective equidistribution conditions.}
The error terms of the weak convergence of the sequence of measures in Principles \ref{prin:Manin-Peyre} and \ref{prin:purity} may depend on each test function defined over the adelic space. We prove in Proposition \ref{prop:EEimpliesgeneral} that the error terms can be made uniform, provided that this is true  for a certain family of adelic subsets called \emph{congruence neighbourhoods}, introduced in \S\ref{se:congneigh}. This motivates us to formulate, based on all known examples, an effective equidistribution condition \textbf{(EE)} in \S\ref{se:condEE}.  Similarly, when lifting into universal torsors, we also formulate a condition \textbf{(EEUT)} in \S\ref{se:EEUT}, and we prove in Proposition \ref{prop:univtorEEGS} that \textbf{(EEUT)} implies \textbf{(EE)} with a different choice of exponents. This is the way we achieve the uniform dependency in Theorem \ref{thm:mainequidist}.

\subsubsection{Equidistribution and sieving}
Principle \ref{prin:purity}, Theorems \ref{thm:ratptsfibration} and \ref{thm:primevalue} are all  concerned with obtaining asymptotic formulas for rational points satisfying ``infinitely many'' local conditions, whose corresponding sets in the adelic space are measurable but not open. The common difficulty (e.g. to operate harmonic analysis or to use the height zeta function) is the lack of group actions on such sets. Our strategy is an abstract ``local-to-global passage''-- Theorems \ref{thm:Tamagawazero} and \ref{thm:keyOmega}. Suppose that an almost-Fano variety $V$ satisfies Principle \ref{prin:Manin-Peyre}.  Roughly speaking,  we ``approximate'' these non-open sets by open sets of $V(\RA_k^{\infty_k})$, defined by a finite collection of local conditions ``truncated'' up to a reasonably large parameter, which goes to $\infty$ as $B$ grows. A condition \textbf{(GS)} is formulated which uniformly controls the number of rational points violating the local condition associated at a certain large prime. A significant feature is that both theorems do not require effective error term estimates, either in Principle \ref{prin:Manin-Peyre} or in condition \textbf{(GS)}. 

\subsubsection{A Selberg's sieve for rational points.} To obtain the upper bounds in Theorem \ref{thm:ratptsfibration} (1) \& (2a) and Theorem \ref{thm:primevalue}, as an effective version of Theorem \ref{thm:Tamagawazero}, we develop a Selberg sieve -- Theorem \ref{thm:Selbergsieve} -- capturing the  finitely many  truncated local conditions as above, and each of them can be ``detected'' modulo a uniform power of primes (indeed modulo $p$ for (1) and modulo $p^2$ for (2a)). A generalised Lang--Weil estimate (Proposition \ref{co:Lang-Weil}) is established in course of proving Theorem \ref{thm:primevalue}.

\subsubsection{Toric van der Corput method}
One key technical innovation of our article -- Theorem \ref{thm:CABCAAB} -- is showing that the contribution from integral points of bounded toric height lying inside certain lopsided boxes is negligible. We tentatively name this treatment \emph{toric van der Corput method}, since it is partly in spirit inspired by the classical van der Corput method, which estimates certain exponential sums over the integers by comparing them with integrals, and controls the error via appropriate averaging processes. In our toric setting, we achieve this by comparing with the aforementioned integrals on $\CTNS$, and evaluate them on making use of toric geometry.
The passage to the complement of such integral points is a crucial step for more demanding lattice point counting problems in spike-shaped domains such as the geometric sieve. This method and its applications mentioned above complement those in \cite[\S11]{Salberger}.

\subsection{Structure of the article}
In \S\ref{se:preliminary} we recall preliminaries for the construction of Tamagawa measures on almost-Fano varieties. 

In \S\ref{se:equidistgeomsieve} we establish various sieving results. In  \S\ref{se:local-to-global} based on Principle \ref{prin:Manin-Peyre} we establish Theorems \ref{thm:Tamagawazero} and \ref{thm:keyOmega}. In \S\ref{se:Selbersieve} based on condition \textbf{(EE)} introduced in \S\ref{se:condEE} we establish Theorem \ref{thm:Selbergsieve}.

In \S\ref{se:univtor}, we first recall the construction of Tamagawa measures on universal torsors, and prove Theorem \ref{thm:equidistunivtor}, which states that a form of equidistribution of integral points on universal torsors implies Principle \ref{prin:Manin-Peyre}. We then prove Proposition \ref{prop:univtorEEGS} as an effective version of this upon introducing condition \textbf{(EEUT)}.

In \S\ref{se:toricparaheight} we recall basic  toric geometry and the construction of toric adelic norms and toric Tamagawa measures.

In \S\ref{se:purityproof} we prove Theorem \ref{thm:effectiveEEUT}, which gives asymptotic formulas for lattice points counting on universal torsors with bounded toric height and appropriate local conditions. Based on this we deduce Theorem \ref{thm:mainequidist}.

In \S\ref{se:toricvandercorput} we prove Theorem \ref{thm:CABCAAB}, which is the central technical result about the toric van der Corput method. 

In \S\ref{se:countingsubvar} we prove upper bounds for points lying in a closed subvariety.

In \S\ref{se:geomsieve} we prove Theorem \ref{thm:maingeomsieve}, by means of establishing Theorem \ref{thm:geomsieve1}, which is the geometric sieve for universal torsors. 

In \S\ref{se:application} we gather all ingredients to prove Theorems \ref{thm:ratptsfibration} \& \ref{thm:primevalue} and confirm Principle \ref{prin:purity}.

\subsection*{Acknowledgements}
We thank Tim Browning for valuable suggestions and encouragements, and we thank Olivier Wittenberg for his constant interest. We are grateful to Yang Cao for his generosity of sharing numerous ideas since our collaborations \cite{Cao-Huang1,Cao-Huang2}. We have benefited from enlightening discussions with many colleagues, especially Régis de la Bretèche,  Daniel Loughran, Emmanuel Peyre, Efthymios Sofos and Florian Wilsch. Part of this work was reported at the `Skoro60bis' conference at Institut Henri Poincaré, and we thank the organisers for their kind invitation.   We heartfully thank the anonymous referees for their efforts that help to ameliorate substantially the quality of this article.

\subsection*{Funding} The author was supported by National Key R\&D Program of China No. 2025YFA1017300 and 2025YFA1017303. Furthermore, the financial supports from SFB 1085 ``Higher Invariants'', from Max-Planck-Institut für Mathematik in Bonn, from Rachel Newton's grant at Kings' College London, from Université Grenoble Alpes, and from the FWF grant No. P32428-N35 at Institute of Science and Technology Austria are also greatly appreciated.

\section{Glossary of terminology}\label{se:preliminary}
\subsection{Notation and conventions}
In this article, the symbol $k$ denotes a number field, with discriminant $d_k$. The ring of integers is denoted by $\CO_k$ or simply $\CO$ if $k$ is considered fixed. We write $\infty_k$ for the (finite) set of archimedean places. Let $\mathfrak{c}_k$ be the class group of $k$. For $S\subset\mathfrak{M}(k)$ any finite set containing $\infty_k$, we denote the ring of $S$-integers by $\CO_{k,S}$ or simply $\CO_{S}$.  For $\nu\in\mathfrak{M}(k)$ non-archimedean, we write $\mathfrak{p}_\nu$ for the prime ideal of $\CO_k$, $k_{\nu}$ for the local field, $\CO_{\nu}$ for the ring of $\nu$-adic integers, $\mathfrak{m}_{\nu}$ for the unique maximal ideal of $\CO_{\nu}$, $\BF_\nu$ for the residue field, and $\operatorname{ord}_{\nu}(\mathfrak{l})$ for the order of an ideal $\mathfrak{l}$ with respect to the $\nu$-adic valuation. We shall use the ordinary absolute values defined as follows. If $\nu\in\mathfrak{M}(k)$ restricts to $\nu'\in\mathfrak{M}(\BQ)$, then for $x\in k$, $$|x|_\nu:=\left|N_{k_{\nu}/\BQ_{\nu'}(x)}\right|_{\nu'}$$ where $N_{k_{\nu}/\BQ_{\nu'}}:k_{\nu}\to\BQ_{\nu'}$ is the norm map, and $|\cdot|_{\nu'}$ is the usual real or $p$-adic absolute value. 
The normalised $\nu$-adic measures $\operatorname{d}x_\nu$ on $k_\nu$ are given by
\begin{equation}\label{eq:nuadicmeasurenormal}
	\begin{cases}
		\int_{\CO_{\nu}} \operatorname{d}x_\nu =1 &\text{ if } \nu \text{  is non-archimedean};\\
		\text{the usual Lebesgue measure} &\text{ if } k_\nu=\BR;\\
		2\operatorname{d}x\operatorname{d}y &\text{ if } k_\nu=\BC \text{ with coordinates }x+iy.
	\end{cases}
\end{equation}

A \emph{nice} $k$-variety is a separated, smooth, geometrically integral scheme of finite type over $k$ of dimension $\geqslant 1$. Fix an algebraic closure $\overline{k}$ of $k$. For every nice $k$-variety $V$, let $\overline{V}:=V\times_k\overline{k}$. 
We write $V^{(1)}$ for the set of codimension one points of $V$.
Let $d=\dim V$. We write $\operatorname{d}x_1\wedge\cdots\wedge\operatorname{d}x_d$ for differential forms as local sections of $K_V=\wedge^d\operatorname{Cot} (V)$ (the determinant line bundle of the cotangent bundle) of a nice $d$-dimensional variety $V$, and we write $\frac{\partial}{\partial x_1}\wedge\cdots\wedge \frac{\partial}{\partial x_d}$ for local sections of $K_V^{-1}=\wedge^d\operatorname{Tan} (V)$ (the determinant line bundle of the tangent bundle). An \emph{integral model} of $V$ is a separated, finite type, flat, $\CO_S$-scheme $\CV$, $S\subset\mathfrak{M}(k)$ finite containing $\infty_k$, such that its generic fibre is isomorphic to $V$.  Such integral models always exist, see e.g. \cite[Proposition 4.2, Lemma 4.3]{Salberger}.

In this article we use interchangeably Vinogradov's symbol and Landau's symbol. For real-valued functions $f$ and $g$ defined over the real numbers with $g$ non-negative,  $f \ll g$ and  $f=O(g)$ both mean that there exists $C>0$ such that $|f|\leqslant Cg$. The dependency on the implied constant $C$ will be specified explicitly. The notation $f\asymp g$ means that $f\ll g$ and $g\ll f$ both hold. If $g$ is nowhere zero, $f(x)=o(g(x))$ means that $\lim_{x\to \infty}\frac{f(x)}{g(x)}=0$, and $f\sim g$ means that $\lim_{x\to \infty}\frac{f(x)}{g(x)}=1$, or equivalently, $f-g=o(g)$.

\subsection{Convergence of measures and the Portmanteau Theorem}\label{se:convergence} (cf. \cite{Billingsley})
Let $\BX$ be a metric space, with the natural Borel $\sigma$-algebra.
Given $\omega$ any measure on $\BX$, a Borel measurable subset $\CB$ is called $\omega$-\emph{continuous} if $\omega(\partial \CB)=0$. 
We say that a sequence of probability measures $(\omega_B)_B$ \emph{converges  weakly} to a probability measure $\omega$  on $\BX$ if for every bounded continuous function $f: \BX\to\BR$, we have $\int_{\BX}f\operatorname{d} \omega_B\to \int_{\BX}f\operatorname{d}\omega$ as $B\to\infty$.
We shall frequently make use of the following criterion for the weak convergence. 
\begin{theorem}[Portmanteau Theorem, cf. \cite{Billingsley} Theorem 2.1]\label{thm:portmanteau}
	Let $(\omega_B)_B$ and $\omega$ be probability measures on $\BX$. Then  $(\omega_B)_B$ converges  weakly to $\omega$ if and only if for every Borel measurable  and $\omega$-continuous set $\CB\subset\BX$, we have
	$$\lim_{B\to\infty}\omega_B(\CB)=\omega(\CB).$$
\end{theorem}
If $g:\BX\to\BY$ is a measurable function between measure spaces, we write $f_*\omega$ for the \emph{pushforward} measure  of $\omega$ to $\BY$, defined as $$f_*\omega(\CC):=\omega(f^{-1}(\CC))$$ for every measurable set $\CC\subset \BY$.
\subsection{Adelic norms, adelic measures and heights}\label{se:adelicmetric}
(See \cite[\S4]{Salberger}, \cite[\S2, \S3]{PeyreBeyond}.)

Let $Y$ be a nice $k$-variety, and let $L\to Y$ be a line bundle on $Y$. 

Recall that, for every $\nu\in\mathfrak{M}(k)$, a \emph{$\nu$-adic metric} on $L$ is a map which associates to each point $P_\nu:\operatorname{Spec}(k_\nu)\to Y$ a $\nu$-adic norm $\|\cdot\|_\nu$ on the $k_\nu$-vector space $L_{P_\nu}:=P_\nu^*L$, such that for every section $s_\nu$ of $L$ defined on a $\nu$-adic open subset $U_\nu\subset Y(k_\nu)$, the map $U_\nu\to\BR_{\geqslant 0}$ given by $x\mapsto \|s_\nu(x)\|_\nu$ is continuous.

An \emph{adelic norm} on $L$ is a family $(\|\cdot\|_\nu)_{\nu\in\mathfrak{M}(k)}$ of metric satisfying the following property. There exist a finite set of places $S$ containing $\infty_k$, an integral model $\CL\to \CY$ of $L\to Y$ over $\CO_S$, such that  for every $\nu\in\mathfrak{M}(k)\setminus S$ and for every $\CO_\nu$-point $\mathbf{P}_\nu: \operatorname{Spec}(\CO_\nu)\to \CY$ with generic fibre $P_\nu\in Y(k_\nu)$,  $$\mathbf{P}_\nu^*\CL=\{x\in L_{P_\nu}:\|x\|_{\nu}\leqslant 1\}.$$
We sometimes call $(\|\cdot\|_\nu)_{\nu\notin S}$ the \emph{model norm} attached to the $\CO_S$-model $\CL\to \CY$.

Associated to a fixed adelic norm $(\|\cdot\|_\nu)_{\nu\in\mathfrak{M}(k)}$ on $L$, the \emph{(Arakelov) height function} $H_L: Y(k)\to\BR_{>0}$ is defined for $P\in Y(k)$ by \begin{equation}\label{eq:heightmetric}
	H_{L}(P):=\prod_{\nu\in\mathfrak{M}(k)}\|s(P)\|^{-1}_\nu,
\end{equation} for any fixed local section $s$ of $L$ defined in a neighbourhood of $P$ with $s(P)\neq 0$. By the product formula, this definition is independent of the choice of $s$.

We equip the anticanonical line bundle $K_Y^{-1}$ with an adelic norm $(\|\cdot\|_\nu)_{\nu\in\mathfrak{M}(k)}$ and write $$d:=\dim Y.$$ For every $\nu\in\mathfrak{M}(k)$, let $\phi_\nu:U_\nu\to k_\nu^{d}$  be a local chart which trivialises $K_Y^{-1}$ with local coordinates $(x_1,\cdots,x_d)$ and a local section $\frac{\partial}{\partial x_1}\wedge\cdots\wedge \frac{\partial}{\partial x_{d}}$, where $U_\nu\subset Y(k_\nu)$ is a $\nu$-adic open subset. Let $\operatorname{d}x_{1}\cdots\operatorname{d}x_{d}$  be the standard $\nu$-adic measure on $k_\nu^{d}$ (normalised as in \eqref{eq:nuadicmeasurenormal}). Let $\omega^Y_\nu$ be the measure on $Y(k_\nu)$ determined by the positive linear functional locally defined by $$f\mapsto\int_{U_\nu} f\circ\phi_\nu^{-1}(x_1,\cdots,x_d)\left\|\frac{\partial}{\partial x_1}\wedge\cdots\wedge \frac{\partial}{\partial x_d}\right\|_\nu\operatorname{d}x_{1}\cdots\operatorname{d}x_{d},$$
for every $f$ compactly supported in $U_\nu$, according to the Riesz representation theorem. 
We call $(\omega_\nu^Y)_{\nu\in\mathfrak{M}(k)}$ the family of \emph{adelic measures} associated to the adelic norm $(\|\cdot\|_\nu)_{\nu\in\mathfrak{M}(k)}$ on $K_Y^{-1}$ (cf. \cite[\S3.2 Construction 3.6]{PeyreBeyond}).

Let $\CY$ be an $\CO_k$-model of $Y$. For every $\nu\notin \infty_k$ and for every $m\geqslant 1$, we write 
\begin{equation}\label{eq:redmodpk}
	\operatorname{Mod}_{\nu,m}:\CY(\CO_\nu)\longrightarrow \CY(\CO_\nu/\mathfrak{m}_{\nu}^m) 
\end{equation}
for the reduction modulo $\mathfrak{m}_{\nu}^m$ map. 
For every non-empty subset $\CI\subset\CY(\CO_\nu/\mathfrak{m}_{\nu}^m) $, $\operatorname{Mod}_{\nu,m}^{-1}(\CI)$ is a non-empty open-closed compact subset of $\CY(\CO_\nu)\subset Y(k_\nu)$.  
The following simple observation will be useful.
\begin{lemma}\label{le:preparation}
	Let $Z\subset Y$  be a proper  Zariski closed subset, and let $W$ be the open subset $Y\setminus Z$. Let $\CZ$ be the Zariski closure of $Z$  in $\CY$, and $\CW:=\CY\setminus\CZ$. Then for every $\nu\notin \infty_k$ with $\CW(\CO_\nu)\neq\varnothing$ and for every $m\geqslant 1$, we have $\operatorname{Mod}_{\nu,m}^{-1}(\CW(\CO_\nu/\mathfrak{m}_\nu^m))=\CW(\CO_\nu)$.
\end{lemma}
\begin{proof}
	First of all it is clear that $\CW(\CO_\nu)\subset \operatorname{Mod}_{\nu,m}^{-1}(\CW(\CO_\nu/\mathfrak{m}_\nu^m))$. For every point $\overline{P}_m\in \CW(\CO_\nu/\mathfrak{m}_\nu^m)$, we write $\overline{P}$ for its image in $\CW(\BF_\nu)$. Assuming $\operatorname{Mod}_{\nu,m}^{-1}(\overline{P}_m)\neq\varnothing$, for any $\widetilde{P}\in \operatorname{Mod}_{\nu,m}^{-1}(\overline{P}_m)\subset \CY(\CO_\nu)$, since the open subset $\CW\cap \widetilde{P}$ of $\widetilde{P}$ contains the closed point $\overline{P}$, it is non-empty and hence must also contain the generic point of $\widetilde{P}$. This shows that $\widetilde{P}\in \CW(\CO_\nu)$, as desired.
\end{proof}

If the family $(\|\cdot\|_\nu)_{\nu\notin S}$ is the model norm attached to an $\CO_S$-model $K_\CY^{-1}\to \CY$ for a finite set of places $S$ containing $\infty_k$, we sometimes call $(\omega^Y_\nu)_{\nu\not\in S}$ the \emph{model measures} attached to $K_\CY^{-1}\to \CY$. 
It can be computed as follows.
\begin{lemma}[cf. e.g. \cite{Oesterle}  \S I.2.2, \cite{Salberger} Theorems 2.13 \& 2.14, Corollary 2.15]\label{le:modelmeasurecomp}
	There exists an integer $m_0\in\BN$ such that for any $m\geqslant m_0$, any $\nu\notin S$, and any $\CI\subset \CY(\CO_\nu/\mathfrak{m}_{\nu}^m)$, we have \begin{equation*}
		\omega^Y_\nu(\operatorname{Mod}_{\nu,m}^{-1}(\CI))=\frac{\#\CI}{\left(\#(\CO_\nu/\mathfrak{m}_{\nu}^m)\right)^d}.
	\end{equation*}
	If $\CY$ is smooth over $\CO_\nu$, then the equality above holds for any $m\geqslant 1$.
\end{lemma}
\subsection{Adelic spaces}\label{se:adelicspaces}
(See \cite[\S I.3]{Oesterle}, \cite[\S2.7]{Wittenberg}.)

Let $Y$ be a nice $k$-variety, and let $\CY$ be an integral model of $Y$ over $\CO_k$. For  $S\subset \mathfrak{M}(k)$ any finite set of places containing $\infty_k$,
let	$$\CY(\widehat{\CO}_S):=\prod_{\nu\not\in S}\CY(\CO_\nu),\quad Y(k_S):=\prod_{\nu\in S}Y(k_\nu).$$ When $S=\infty_k$ we write $\CY(\widehat{\CO}_k)$  instead of $\CY(\widehat{\CO}_{\infty_k})$. Equipped with the product topology, $\CY(\widehat{\CO}_S)$ is compact.	The \emph{adelic space} (after Weil and Grothendieck) of $Y$ is
		\begin{equation}\label{eq:adelization}
			Y(\RA_k):=\varinjlim_{S\subset \mathfrak{M}(k)}Y(k_S)\times\CY(\widehat{\CO}_S),
		\end{equation} where the direct limit is taken for all such finite subsets $S$ containing $\infty_k$. We equip $Y(\RA_k)$ with the topology induced by the direct limit. For any finite set $S'\subset\mathfrak{M}(k)$, let $Y(\RA_k^{S'})$ denote the \emph{adelic space off} $S'$,  which is the image of $Y(\RA_k)$ under the projection without the factors in $S'$.
	These are locally compact Haudorff topological spaces with countable base.
			If $Y$ is proper, then $$Y(\RA_k)=\prod_{\nu\in\mathfrak{M}(k)} Y(k_\nu),$$ and it is compact.
		
		We say that $Y$ satisfies \emph{strong approximation} if the diagonal image of $Y(k)$ is dense in $Y(\RA_k)$. We say that $Y$ satisfies \emph{(arithmetic) purity of strong approximation} (see \cite[p. 336]{Cao-Liang-Xu}, \cite[Definition 1.3 (ii)]{Cao-Huang1}) if for every Zariski closed subset $Z\subset Y$ of codimension at least two, the open subset $Y\setminus Z$ satisfies strong approximation.

The \emph{Brauer--Manin set} $Y(\RA_k)^{\operatorname{Br}}$ is the closed subset of $Y(\RA_k)$ consisting of adelic points that are orthogonal to all the elements  of the Brauer group $\operatorname{Br}(Y)$. Class field theory implies that the closure of $Y(k)$ is contained in $Y(\RA_k)^{\operatorname{Br}}$. (See e.g. \cite[\S3.3]{PeyreBeyond} for more details.) We can similarly define \emph{strong approximation with Brauer--Manin obstruction} and its purity analogue as above upon replacing $Y(\RA_k)$ by $Y(\RA_k)^{\operatorname{Br}}$.

\subsection{Congruence neighbourhoods}\label{se:congneigh}
We now introduce a special kind of adelic sets. Let $Y$ and $\CY$ be as before. Enlarging $S$ if necessary, we may assume that $\CY(\widehat{\CO}_S)\neq\varnothing$ by the Lang--Weil estimate \cite{Lang-Weil} and Hensel's lemma (cf. e.g. \cite[\S1.7]{Cao-Huang2}). For $\nu\notin S,m\in\BZ_{>0}$ and  $\xi_\nu\in \CY(\CO_\nu)$, we define 
$$\CE_\nu^{\CY}(m;\xi_\nu):=\left\{x_\nu\in\CY(\CO_\nu):\operatorname{Mod}_{\nu,m}(x_\nu)=\operatorname{Mod}_{\nu,m}(\xi_\nu)\right\}.$$ This is a non-empty $\nu$-adic open-closed subset of $\CY(\CO_\nu)$. 
Now for $\mathfrak{l}$ an ideal of $\CO_S$, $\bxi=(\xi_\nu)_{\nu\mid \mathfrak{l}}\in\prod_{\nu\mid \mathfrak{l}}\CY(\CO_\nu)$, we define the \emph{congruence neighbourhood} associated to $\CY,\bxi$ of \emph{level} $\mathfrak{l}$ to be
$$\CE^{\CY}(\mathfrak{l};\bxi):=\prod_{\nu\mid \mathfrak{l}}\CE^{\CY}_\nu(\operatorname{ord}_{\nu}(\mathfrak{l});\xi_\nu)\times \prod_{\nu\nmid \mathfrak{l},\nu\notin S}\CY(\CO_\nu).$$
This is a non-empty open-closed subset of $\CY(\widehat{\CO}_S)$.  The following properties are easily verified. 
\begin{itemize}
	\item For any two collections of points $\bxi^\prime=(\bxi_\nu^\prime),\bxi^{\dprime}=(\bxi^{\dprime}_\nu)\in \prod_{\nu\mid \mathfrak{l}}\CY(\CO_\nu)$,
	\begin{equation}\label{eq:disjoint}
		\CE^{\CY}(\mathfrak{l};\bxi^\prime)\cap \CE^{\CY}(\mathfrak{l};\bxi^{\dprime})\neq\varnothing\Longleftrightarrow \CE^{\CY}(\mathfrak{l};\bxi^\prime)= \CE^{\CY}(\mathfrak{l};\bxi^{\dprime});
	\end{equation}
\item For $\mathfrak{l}_1,\mathfrak{l}_2$ with $\mathfrak{l}_1\mid \mathfrak{l}_2$ and $\bxi_1=(\xi_{1,\nu})_{\nu\mid\mathfrak{l}_1},\bxi_2=(\xi_{2,\nu})_{\nu\mid\mathfrak{l}_2}$ with $\operatorname{Mod}_{\nu,\operatorname{ord}_{\nu}(\mathfrak{l}_1)}(\xi_{1,\nu})=\operatorname{Mod}_{\nu,\operatorname{ord}_{\nu}(\mathfrak{l}_1)}(\xi_{2,\nu})$ for all $\nu\mid \mathfrak{l}_1$, we have
$\CE^{\CY}(\mathfrak{l}_2;\bxi_2)\subset \CE^{\CY}(\mathfrak{l}_1;\bxi_1)$. 
\end{itemize}
Hence the family of such congruence neighbourhoods $(\CE^{\CY}(\mathfrak{l};\bxi))_{\mathfrak{l},\bxi}$ forms a topological basis for $\CY(\widehat{\CO}_S)$. 
\begin{definition}\label{def:LCEf}
	For every non-empty open-closed (compact) subset $\CE\subset \CY(\widehat{\CO}_S)$, we define the \emph{covering exponent} $\CL(\CE)$ (depending on the model $\CY$) to be the norm of the largest ideal $\mathfrak{l}$ such that there exist (finitely many)  congruence neighbourhoods $\{\CE^\CY(\mathfrak{l};\bxi_i)\}_i$ of level all equal to $\mathfrak{l}$ with $$\CE=\bigsqcup_{i} \CE^\CY(\mathfrak{l};\bxi_i).$$
\end{definition}
The following simple observation will be useful.
\begin{lemma}\label{le:CLCEfc}
	Let $\CE\subset \CY(\widehat{\CO}_S)$ be non-empty open-closed, and suppose $\CE\neq \CY(\widehat{\CO}_S)$. Let $\CE^c:=\CY(\widehat{\CO}_S)\setminus \CE$. Then we have $$\CL(\CE^c)=\CL(\CE).$$ 
\end{lemma}
\begin{proof}
		Write for simplicity $\CL=\CL(\CE)$. By the definition above we can cover the compact set  $\CE$ by a finite disjoint union of congruence neighbourhoods   $$\CE=\bigsqcup_{i\in I}\CE^\CY(\mathfrak{l};\bxi_i)$$ with $\#\CO_S/\mathfrak{l}=\CL$. 
		By \eqref{eq:disjoint}, we can extend this to a finite disjoint union covering for $\CY(\widehat{\CO}_S)$, still of level $\mathfrak{l}$:
		$$\CY(\widehat{\CO}_S)=\bigsqcup_{i\in I'} \CE^\CY(\mathfrak{l};\bxi_i),$$ where $I\subset I'$.
	Then $$\CE^c=\bigsqcup_{i\in I'\setminus I} \CE^\CY(\mathfrak{l};\bxi_i),$$ and hence $\CL(\CE^c)\leqslant \CL$. The reverse inequality can be deduced in the same way.
\end{proof}
\subsection{Almost-Fano varieties}\label{se:heightsTamagawa}
(See \cite[\S3.4.2]{PeyreBeyond}.)
\begin{hypothesis}[cf. \cite{PeyreBeyond} Hypotheses 3.27]\label{def:almostFano}
	We call a nice proper $k$-variety $V$ \emph{almost Fano} if it satisfies the following hypotheses:
	\begin{enumerate}
		\item The set of rational points $V(k)$ is Zariski dense.
		\item The cohomological groups $H^1(V,\CO_V)$, $H^2(V,\CO_V)$ are all zero.
		\item The group $\Pic(\overline{V})$ is torsion free, and the group $\operatorname{Br}(V)/\operatorname{Im}(\operatorname{Br}(k))$ is finite.
		\item The cone of pseudoeffective divisors $\overline{\operatorname{Eff}}(V)$ is finitely generated.
		\item  The anticanonical line bundle $K_V^{-1}$ is within the interior of $\overline{\operatorname{Eff}}(V)$. 
	\end{enumerate}
	\begin{remarks}
		\hfill
		\begin{enumerate}
			\item The finiteness condition (3) implies that the Brauer--Manin set $V(\RA_k)^{\operatorname{Br}}$ is open-closed in $V(\RA_k)$.
			\item 		Recent progress on birational geometry shows that Fano varieties (i.e. nice projective varieties with ample anticanonical line bundle) are Mori Dream Spaces \cite[Corollary 1.3.2]{BCHM}. In particular their cones of pseudoeffective divisors are rational polyhedral. So Fano varieties satisfy the properties (2)--(5) above.
		\end{enumerate}
	\end{remarks}
\end{hypothesis}
\subsubsection{The constants $\alpha(V)$ and $\beta(V)$}
 (See \cite[D\'efinition 2.4]{Peyre}, \cite[Définition 4.8]{PeyreLille})
 
Recall $r:=\operatorname{rank}\operatorname{Pic}(V)$. Let $\Pic(V)^\vee$ be the dual lattice of $\Pic(V)$ inside $\Pic(V)^\vee_\BR\simeq \BR^r$. Let $\operatorname{d}y$ be the Haar measure on $\Pic(V)^\vee_\BR$ normalised so that $\Pic(V)^\vee$ has covolume one. Let $\overline{\operatorname{Eff}}(V)^\vee\subset\Pic(V)^\vee_\BR$ be the dual cone of $\overline{\operatorname{Eff}}(V)$.
We then define
\begin{equation}\label{eq:alphaV}
	\alpha(V):=\frac{1}{(r-1)!}\int_{\overline{\operatorname{Eff}}(V)^\vee}\operatorname{e}^{-\langle K_V^{-1},y\rangle}\operatorname{d}y;
\end{equation} 
\begin{equation}\label{eq:beta}
	\beta(V):=\#H^1(k,\operatorname{Pic}(\overline{V})).
\end{equation}

\subsubsection{Tamagawa measures on almost-Fano varieties}
(See \cite[\S2.2.3]{Peyre}.)

Let us fix a family of adelic measures $(\omega_\nu^V)_{\nu\in\mathfrak{M}(k)}$ on an almost-Fano variety $V$. 
Let $\CV$ be a proper integral model of $V$ over $\CO_S$ for a finite set $S\subset\mathfrak{M}(k)$ containing $\infty_k$. We enlarge $S$ if necessary, so that for every $\nu\not\in S$, we have $\Pic(\CV_{\overline{\BF_\nu}})\simeq \Pic(\overline{V})$ compatible with the actions of Galois groups. The geometric Frobenius $\operatorname{Fr}_\nu$ acts on $H^2_{\text{\'et}}(\CV_{\overline{\BF_\nu}},\BQ_l(1))\simeq \Pic(\overline{V})$ for any prime $l$ different from the characteristic of $\BF_\nu$.
Now on $\Re(s)>0$, we define the Artin L-functions
\begin{equation}\label{eq:artinLnu}
	L_\nu(s,\Pic(\overline{V})):=\frac{1}{\det(1-(\#\BF_\nu)^{-s}\operatorname{Fr}_\nu|\Pic(\overline{V})_\BQ)},\quad \text{for every }\nu\not\in S,
\end{equation}
\begin{equation}\label{eq:artinLS}
	L_S(s,\Pic(\overline{V})):=\prod_{\nu\not\in S}L_\nu(s,\Pic(\overline{V})).
\end{equation}
Then $L$-function $L_S(s,\Pic(\overline{V}))$ has a pole at $s=1$ of order $r=\operatorname{rank}(\Pic(V))$.
We now define the set of convergence factors 
\begin{equation}\label{eq:convergencefact}
	\lambda_\nu:=\begin{cases}
		1 & \text{ if } \nu\in S;\\ L_\nu(1,\Pic(\overline{V})) &\text{ if }\nu\not\in S.
	\end{cases}
\end{equation}
By the hypothesis (2) in Setting \ref{def:almostFano}, using the Grothendieck-Lefschetz formula and Deligne's proof of Weil's conjecture, one deduces (cf. \cite[Proposition 2.2.2]{Peyre} which requires only the properness assumption)
$$\frac{\omega_\nu^V(V(k_\nu))}{L_\nu(1,\Pic(\overline{V}))}=1+O((\#\BF_\nu)^{-\frac{3}{2}}),$$ for all $\nu\notin S$.
Therefore, the product measure
\begin{equation}\label{eq:Tamagawameas}
	\omega^V:=\frac{\lim_{s\to 1}(s-1)^rL_S(s,\Pic(\overline{V}))}{d_k^{\dim V/2}}\bigotimes_{\nu\in\mathfrak{M}(k)}\lambda_\nu^{-1}\omega_\nu^V
\end{equation} applied to $V(\RA_k)$ is absolutely convergent, and is independent of the choice of the finite set $S$.
 Note that $\frac{\omega^V}{\omega^V(V(\RA_k))}$ is a regular probability measure on $V(\RA_k)$, and coincides with the infinite product measure $\bigotimes_{\nu\in\mathfrak{M}(k)} \frac{\omega_\nu^V}{\omega_\nu^V(V(k_\nu))}$ by the uniqueness in the Riesz representation theorem.

\subsubsection{Tamagawa measures on open subvarieties}
Let $W\subset V$ be an dense open subset whose complement $Z=V\setminus W$ is of codimension at least two. For every $\nu\in\mathfrak{M}(k)$, the $\nu$-adic measure $\omega_\nu^V$ naturally restricts to the (open) measurable subset $W(k_\nu)$, which we denote by $\omega_\nu^W$.
According to \cite[Proposition 2.1]{Cao-Huang2}, the family of adelic measures $(\omega_\nu^W)_{\nu\in\mathfrak{M}(k)}$ together with the set of convergence factors \eqref{eq:convergencefact} define in the same way as \eqref{eq:Tamagawameas} a measure on $W(\RA_k)$, denoted by $\omega^V|_W$ (which can also be normalised to a probability measure).
\begin{remark}
	We note that the adelic topology of $W(\RA_k)$ is in general not the one induced by $V(\RA_k)$, and hence these two measure spaces have different $\sigma$-algebras.
\end{remark}

\section{Equidistribution and sieving}\label{se:equidistgeomsieve}
\begin{hypothesis}\label{hyp:almost-Fano}
	Let $V/k$ be almost-Fano, equipped with a family of adelic measures $(\omega_\nu^V)_{\nu\in\mathfrak{M}(k)}$. Let $\omega^V$ be the associated Tamagawa measure. 
\end{hypothesis} 
Throughout this section, we shall be working with varieties in Setting \ref{hyp:almost-Fano} satisfying the following.
\begin{hypothesis1}\label{hyp:deltaB}
	Assume that there exists a family of counting measures $(\Delta(B))_B$
	which converges weakly to $\omega^V$ in $V(\RA_k)^{\operatorname{Br}}$ as $B\to\infty$.
\end{hypothesis1}
\begin{remark}
	For instance, given any $V$ satisfying Principle \ref{prin:Manin-Peyre}, we can take $$\Delta(B)=\delta_{(V(k)\setminus M)_{\leqslant B}},$$ $\delta_{(V(k)\setminus M)_{\leqslant B}}$ being defined by \eqref{eq:seqVAk} in terms of the associated anticanonical height function and $M\subset V(k)$ a fixed thin subset.  But there can exist different ways of ordering rational points which also lead to the weak convergence of the corresponding family of counting measures (cf. \cite[\S4]{Peyre}).
\end{remark}

For every measurable $\CF\subset V(\RA_k)$ and $B>1$, we define the counting function
\begin{equation}\label{eq:countingDelta}
	\CN_{V}(\CF;B):=\int_{\CF}\operatorname{d}\Delta(B).
\end{equation}

\subsection{Local-to-global passage and the geometric sieve}\label{se:local-to-global}
Let $V$ be as in Setting \ref{hyp:almost-Fano}.
Let $(\Omega_\nu\subset V(k_\nu))_{\nu\in\mathfrak{M}(k)}$ be a collection of $\nu$-adic measurable sets satisfying \begin{equation}\label{eq:posmeas}
	\omega_\nu^V(\Omega_\nu)>0\text{ and } \omega^V_\nu(\partial\Omega_\nu)=0 \text{ for every } \nu\in\mathfrak{M}(k).
\end{equation}
We consider the measurable set $$\CE[(\Omega_{\nu})]:=\prod_{\nu\in\mathfrak{M}(k)}\Omega_\nu\subset V(\RA_k).$$
In this subsection we prove two theorems concerning the counting function $\CN_V\left(\CE[(\Omega_{\nu})];B\right)$.
Contrary to every single $\Omega_{\nu}$, the set $\CE[(\Omega_{\nu})]$ is not $\omega^V$-continuous if $\Omega_{\nu}\neq V(k_{\nu})$ for infinitely many $\nu$. Being opposed to each other, Theorem \ref{thm:Tamagawazero} states that if $\CE[(\Omega_{\nu})]$ has measure zero, then $\CN_V\left(\CE[(\Omega_{\nu})];B\right)$ has negligible contribution, while Theorem \ref{thm:keyOmega} provides a sufficient condition (named \textbf{(GS)}) to guarantee that $\CN_V\left(\CE[(\Omega_{\nu})];B\right)$ converges to the expected measure.

For every $N\geqslant 1$, we shall frequently make use of the following ``truncated'' adelic set \begin{equation}\label{eq:truncate}
	\CE_{N}[(\Omega_{\nu})]:= \prod_{\substack{\nu\in\infty_k\text{ or}\\\nu\not\in\infty_k: \#\BF_\nu\leqslant N}}\Omega_\nu\times \prod_{\substack{\nu\notin\infty_k: \#\BF_\nu>N}} V(k_\nu) \subset V(\RA_k).
\end{equation}
All implied constants in the rest of this subsection are allowed to depend on the family $(\Omega_\nu)_{\nu\in\mathfrak{M}(k)}$.
\begin{theorem}\label{thm:Tamagawazero}
	Assume that Hypothesis \ref{hyp:deltaB} holds. If $\omega^V\left(\CE[(\Omega_{\nu})]\right)=0$, then as $B\to \infty$,
	$$\CN_V\left(\CE[(\Omega_{\nu})];B\right)=o(1).$$
\end{theorem}
\begin{remark}
	Besides the immediate applications to counting rational points
	\begin{itemize}
		\item  lying in a proper closed subvariety; or
		\item inside of any adelic set of an open subvariety whose complement is a divisor;
	\end{itemize} Theorem \ref{thm:Tamagawazero} also applies to the situation of Theorem  \ref{thm:ratptsfibration} (1) as well as (2a) by taking $(\Omega_\nu=f(Y(k_\nu)))$. See Corollary \ref{co:paucityfibra} (including a verification of \eqref{eq:posmeas}). This generalises \cite[Theorems 1.2 \& 1.4]{Browning-Loughran} concerning thin sets.
\end{remark}
\begin{proof}[Proof of Theorem \ref{thm:Tamagawazero}]
Let $\varepsilon>0$ be fixed. That the set $\CE[(\Omega_{\nu})]$ having measure zero implies that the infinite product
$\prod_{\nu\in\mathfrak{M}(k)}\frac{\omega_\nu^V(\Omega_\nu)}{\omega_\nu^V(V(k_\nu))}$ diverges to zero. So there exists $N(\varepsilon)\geqslant 1$ such that,  $$\prod_{\substack{\nu\in\infty_k\text{ or}\\\nu\not\in\infty_k:  \#\BF_\nu\leqslant N(\varepsilon)}}\frac{\omega_\nu^V(\Omega_\nu)}{\omega_\nu^V(V(k_\nu))}<\varepsilon.$$
Then $\CE[(\Omega_{\nu})]\subset \CE_{N(\varepsilon)}[(\Omega_{\nu})]$ and 
$$\omega^V\left(\CE_{N(\varepsilon)}[(\Omega_{\nu})]\right)=\omega^V(V(\RA_k))\prod_{\substack{\nu\in\infty_k\text{ or}\\\nu\not\in\infty_k:  \#\BF_\nu\leqslant N(\varepsilon)}}\frac{\omega_\nu^V(\Omega_\nu)}{\omega_\nu^V(V(k_\nu))}<\varepsilon \omega^V(V(\RA_k)).$$
The set $\CE_{N(\varepsilon)}[(\Omega_{\nu})]$ is $\omega^V$-continuous thanks to the assumption \eqref{eq:posmeas}. Since $V(\RA_k)^{\operatorname{Br}}$ is open-closed, it is also $\omega^V$-continuous. So Theorem \ref{thm:portmanteau}  implies that, as $B\to \infty$,
\begin{align*}
	\lim_{B\to\infty}\CN_V\left(\CE_{N(\varepsilon)}[(\Omega_{\nu})]\cap V(\RA_k)^{\operatorname{Br}};B\right)= \omega^V\left(\CE_{N(\varepsilon)}[(\Omega_{\nu})]\cap V(\RA_k)^{\operatorname{Br}}\right).
\end{align*}
It follows that 
\begin{align*}
	\limsup_{B\to\infty}\CN_V\left(\CE[(\Omega_{\nu})];B\right)&=\limsup_{B\to\infty}\CN_V\left(\left(\CE[(\Omega_{\nu})]\right)\cap V(\RA_k)^{\operatorname{Br}};B\right)\\ &\leqslant \limsup_{B\to\infty}\CN_V\left(\CE_{N(\varepsilon)}[(\Omega_{\nu})]\cap V(\RA_k)^{\operatorname{Br}};B\right)\\ &= \omega^V\left(\CE_{N(\varepsilon)}[(\Omega_{\nu})]\cap V(\RA_k)^{\operatorname{Br}}\right)\\ &\leqslant \omega^V\left(\CE_{N(\varepsilon)}[(\Omega_{\nu})]\right)<\varepsilon \omega^V(V(\RA_k)).
\end{align*}
This finishes the proof.
\end{proof}

Our next theorem on the abstract geometric sieve is inspired by the work of Poonen--Stoll \cite[Lemma 20]{Poonen-Stoll}. For every $\nu\in\mathfrak{M}(k)$, we write $\Omega_\nu^c:=V(k_\nu)\setminus \Omega_\nu$.  For $N>1$, we define
\begin{equation}\label{eq:CR}
	\CR((\Omega_\nu);N,B):=\CN_{V}\left(\bigcup_{\substack{\nu\notin\infty_k: \#\BF_\nu>N}}\Omega_{\nu}^c\times \prod_{\nu^\prime\neq \nu}V(k_{\nu^\prime});B\right), 
\end{equation} and we write
$$h(N):=\limsup_{B\to\infty} \CR((\Omega_\nu);N,B).$$
\begin{condGS}
	As $N\to\infty$, we have $$h(N)=o(1).$$ 
\end{condGS}
We recall \eqref{eq:truncate}.
\begin{theorem}\label{thm:keyOmega}
	Assume that Hypothesis \ref{hyp:deltaB} holds, that $\CE[(\Omega_{\nu})]\subset V(\RA_k)^{\operatorname{Br}}$, and that the family $(\Omega_\nu)$ satisfies \textbf{(GS)}. Then we have $\omega^V\left(\CE[(\Omega_{\nu})]\right)>0$, and \begin{equation}\label{eq:asympCS}
	\lim_{B\to\infty}\CN_V\left(\CE[(\Omega_{\nu})];B\right)=\omega^V\left(\CE[(\Omega_{\nu})]\right).
\end{equation}
\end{theorem}
\begin{remarks}\hfill
	\begin{enumerate}
		\item Theorem \ref{thm:keyOmega} applies to the situation of Principle \ref{prin:purity} by taking $\CE[(\Omega_{\nu})]$ to be any sufficiently small adelic set of $W(\RA_k)$, and also applies to Theorem \ref{thm:ratptsfibration} (2b) with $(\Omega_\nu=f(Y(k_\nu)))$, condition \textbf{(GS)} being deduced from Proposition \ref{prop:surjectivity}. 
		\item 	On identifying $\Omega_{\nu}$ with the adelic set $\Omega_{\nu}\times \prod_{\nu^\prime\neq \nu}V(k_{\nu^\prime})\subset V(\RA_k)$,
		we regard the family $(\Omega_{\nu})$ as independent events in the probability space $\left(V(\RA_k),\frac{\omega^V}{\omega^V(V(\RA_k))}\right)$.  Let us consider the measurable set $$\CG[(\Omega_{\nu})]:=\bigcup_{N=1}^\infty\bigcap_{\substack{\nu\notin\infty_k: \#\BF_\nu>N}}\left(\Omega_{\nu}\times \prod_{\nu^\prime\neq \nu}V(k_{\nu^\prime})\right),$$ which is a \emph{tail} event interpreted as ``all but finitely many of the events $(\Omega_{\nu})$ occur''.  Its complement  $$ \CG[(\Omega_{\nu})]^c:=\bigcap_{N=1}^\infty \bigcap_{\substack{\nu\notin\infty_k: \#\BF_\nu>N}}\left(\Omega_{\nu}^c\times \prod_{\nu^\prime\neq \nu}V(k_{\nu^\prime})\right)$$ is also a tail event, and can equally be interpreted as  ``infinitely many of the events $(\Omega_{\nu}^c)$ occur''.  Since the infinite product $\prod_{\nu\in\mathfrak{M}(k)} \frac{\omega_\nu^V(\Omega_{\nu})}{\omega_\nu^V(V(k_\nu))}$ converges to a non-zero limit if and only if the series $\sum_{\nu\in\mathfrak{M}(k)}  \frac{\omega_\nu^V(\Omega^c_{\nu})}{\omega_\nu^V(V(k_\nu))}$ converges, applying the Borel--Cantelli lemma (see \cite[10.10]{Folland}), we obtain
		\begin{enumerate}
			\item If $\omega^V\left(\CE[(\Omega_{\nu})]\right)=0$, then the event $\CG[(\Omega_{\nu})]^c$ occurs with probability one;
			\item If $\omega^V\left(\CE[(\Omega_{\nu})]\right)>0$, then  the event  $ \CG[(\Omega_{\nu})]$ occurs with probability one.
		\end{enumerate}
		From this probabilistic point of view, condition \textbf{(GS)} provides 
		an inverse of (b). By requiring that the tail event $\CG[(\Omega_{\nu})]^c$  has measure zero ``uniformly'' in terms of $B$, it not only gives a sufficient condition for the set $\CE[(\Omega_{\nu})]$ to have positive measure, but also guarantees the  convergence of the counting function \eqref{eq:countingDelta}.
	\end{enumerate}
\end{remarks}
\begin{proof}[Proof of Theorem \ref{thm:keyOmega}]
Since the adelic space has the product topology, $V(\RA_k)^{\operatorname{Br}}$ is open, and $\CE[(\Omega_{\nu})]\subset V(\RA_k)^{\operatorname{Br}}$ by assumption, we may fix $N_0>1$ such that $\CE_{N_0}[(\Omega_{\nu})]\subset V(\RA_k)^{\operatorname{Br}}$.
For every $N_1,N_2\in\BZ_{>0}$ with $ N_1<N_2$, let us define $$\CE_{[N_1,N_2]}[(\Omega_{\nu})]:=\CE_{N_0}[(\Omega_{\nu})]\bigcap\left(\prod_{\substack{\nu\in\infty_k \text{ or}\\\nu\notin\infty_k:N_1<\#\BF_\nu\leqslant N_2}}\Omega_\nu\times \prod_{\substack{\nu\notin\infty_k:\#\BF_\nu\leqslant N_1\text{ or }> N_2}}V(k_\nu)\right)\subset V(\RA_k)^{\operatorname{Br}}.$$ 
Clearly $\CE_{[N_1,N_2]}[(\Omega_{\nu})]$ is $\omega^V$-continuous, thanks to \eqref{eq:posmeas}. Theorem \ref{thm:portmanteau} now implies that for every such fixed $N_1,N_2$, \begin{equation}\label{eq:aN}
		\lim_{B\to\infty}\CN_V\left(\CE_{[N_1,N_2]}[(\Omega_{\nu})];B\right)=\omega^V\left(\CE_{[N_1,N_2]}[(\Omega_{\nu})]\right).
	\end{equation}

\textbf{Step I.} We first show that $\omega^V\left(\CE[(\Omega_{\nu})]\right)>0$. Since $$\CN_V(\CE_{N_0}[(\Omega_{\nu})];B)-\CN_V\left(\CE_{[N_1,N_2]}[(\Omega_{\nu})];B\right)\leqslant \CR((\Omega_\nu);N_1,B),$$  according to \eqref{eq:aN} and 
taking $\limsup$ on both sides, we obtain
\begin{align*}
	\omega^V(\CE_{N_0}[(\Omega_{\nu})])-\omega^V(\CE_{[N_1,N_2]}[(\Omega_{\nu})])\leqslant h(N_1).
\end{align*}
Moreover, whenever $N_0\leqslant N_1<N_2$,
\begin{align*}
	&\omega^V(\CE_{N_0}[(\Omega_{\nu})])\left|\prod_{\substack{\nu\notin\infty_k:N_1<\#\BF_\nu\leqslant N_2}}\lambda_\nu^{-1}\omega_\nu^V(\Omega_\nu)-1\right| \\ 
	=&\lim_{B\to\infty}\left|\left(\prod_{\substack{\nu\notin\infty_k:N_1<\#\BF_\nu\leqslant N_2}}\lambda_\nu^{-1}\omega_\nu^V(V(k_\nu))\right)\CN_V\left(\CE_{[N_1,N_2]}[(\Omega_{\nu})];B\right)-\CN_{V}(\CE_{N_0}[(\Omega_{\nu})];B)\right|\\ \leqslant &\left|\prod_{\substack{\nu\notin\infty_k:N_1<\#\BF_\nu\leqslant N_2}}\lambda_\nu^{-1}\omega_\nu^V(V(k_\nu))-1\right|\lim_{B\to\infty}\CN_V\left(\CE_{[N_1,N_2]}[(\Omega_{\nu})];B\right)+h(N_1)\\ \leqslant&\left|\prod_{\substack{\nu\notin\infty_k:N_1<\#\BF_\nu\leqslant N_2}}\lambda_\nu^{-1}\omega_\nu^V(V(k_\nu))-1\right|\omega^V(\CE_{N_0}[(\Omega_{\nu})])+h(N_1).
\end{align*}
By construction, the Tamagawa measure \eqref{eq:Tamagawameas} with the family of convergence factors $(\lambda_\nu)$ is absolutely convergent. So $\prod_{\substack{\nu\notin\infty_k:N_1<\#\BF_\nu\leqslant N_2}}\lambda_\nu^{-1}\omega_\nu^V(V(k_\nu))$ is arbitrarily close to $1$ provided that $N_1<N_2$ are large. So condition \textbf{(GS)} implies that $\prod_{\substack{\nu\notin\infty_k:N_1<\#\BF_\nu\leqslant N_2}}\lambda_\nu^{-1}\omega_\nu^V(\Omega_\nu)$ is also arbitrarily close to $1$ uniformly for arbitrarily large $N_1<N_2$. 
Therefore, by Cauchy's criterion, this confirms the convergence of $\prod_{\substack{\nu\in\mathfrak{M}(k)\setminus\infty_k}}\lambda_\nu^{-1}\omega_\nu^V(\Omega_\nu)$ whence $\omega^V\left(\CE[(\Omega_{\nu})]\right)>0$.

\textbf{Step II.} 
We write for simplicity
$$g(B):=\CN_V\left(\CE[(\Omega_{\nu})];B\right).$$
To complete the proof we now show $\lim_{B\to\infty} g(B)$ exists and is equal to $\omega^V\left(\CE[(\Omega_{\nu})]\right)$.

For every $N\geqslant N_0$, let $g_N(B):=\CN_V\left(\CE_{[1,N]}[(\Omega_{\nu})];B\right)$.
First, we observe that $g(B)\leqslant g_N(B)$. So
$$\limsup_{B\to\infty} g(B)\leqslant\limsup_{B\to\infty}g_N(B)=\omega^V\left(\CE_{[1,N]}[(\Omega_{\nu})]\right),$$ and therefore $$\limsup_{B\to\infty}g(B)\leqslant \lim_{N\to\infty}\omega^V\left(\CE_{[1,N]}[(\Omega_{\nu})]\right)=\omega^V\left(\CE[(\Omega_{\nu})]\right).$$
On the other hand, since $$g_N(B)-g(B)\leqslant \CR((\Omega_\nu);N,B),$$ we have $$\limsup_{B\to\infty}\left(g_N(B)-g(B)\right)\leqslant h(N).$$ Hence
$$\liminf_{B\to\infty} g(B)\geqslant \liminf_{B\to\infty} g_N(B)-\limsup_{B\to\infty}\left(g_N(B)-g(B)\right)\geqslant \omega^V\left(\CE_{[1,N]}[(\Omega_{\nu})]\right)-h(N),$$ and therefore
$$\liminf_{B\to\infty} g(B)\geqslant \lim_{N\to\infty} \left(\omega^V\left(\CE_{[1,N]}[(\Omega_{\nu})]\right)-h(N)\right)=\omega^V\left(\CE[(\Omega_{\nu})]\right).$$
This finishes the proof.
\end{proof}

\subsection{Effective equidistribution condition}\label{se:condEE}
We fix $V$ as in Setting \ref{hyp:almost-Fano}. Let $\CV$ be a proper integral model of $V$  over $\CO_S$, where $S\subset \mathfrak{M}(k)$ is a finite set of places containing $\infty_k$. We define the measure \begin{equation}\label{eq:omegaSV}
	\widetilde{\omega}_S^V:=\bigotimes_{\nu\in S}\omega_\nu^V\text{ on }V(k_S).
\end{equation} We write $\operatorname{pr}_S:V(\RA_k)\to V(k_S)$ for the natural projection and let $\mathfrak{F}$ be a family of $\widetilde{\omega}_S^V$-continuous sets contained in $\operatorname{pr}_S(V(\RA_k)^{\operatorname{Br}})$. 
The following condition that we now formulate is an effective version of Hypothesis \ref{hyp:deltaB} which depends on the integral model $\CV$ and the family $\mathfrak{F}$. We require certain uniformity of levels in the error term.
\begin{condEE}
 Assume that there exist $\gamma\geqslant 0$,  and a continuous function $h:\BR_{>0}\to\BR_{>0}$ with $h(x)=o(1),x\to\infty$, satisfying the following property:
		
		Let $\CF_0\in\mathfrak{F}$, and let $\CE^\CV(\mathfrak{l};\bxi)\subset \CV(\widehat{\CO}_S)$ be a congruence neighbourhood  associated to $\CV,\bxi\in\prod_{\nu\mid \mathfrak{l}} \CV(\CO_\nu)$ of level $\mathfrak{l}$. Assume  $\CF:=\CF_0\times \CE^\CV(\mathfrak{l};\bxi)\subset V(\RA_k)^{\operatorname{Br}}$. Then we have
		$$\CN_V(\CF;B)=\omega^V(\CF)+O_{\CF_0}\left((\#\CO_S/\mathfrak{l})^\gamma h(B)\right),$$  where the implied constant depends on $\CF_0$ but is uniform for every such $\mathfrak{l},\bxi$.
\end{condEE}

\begin{remark}
	The polynomial-type bound on the norm of $\mathfrak{l}$ is compatible with all known examples (see e.g. \cite{Birch} for projective hypersurfaces of low degree and \cite[Theorem 3.1]{Cao-Huang2} for affine homogeneous varieties).
\end{remark}
\begin{proposition}\label{prop:EEimpliesgeneral}
	Suppose that $V$ satisfies \textbf{(EE)} with respect to $(\CV,\mathfrak{F})$. Then 
	for every $\omega^V$-continuous adelic set of the form $\CF=\CF_0\times\CF_1\subset V(\RA_k)^{\operatorname{Br}}$ where  $\CF_0\in\mathfrak{F}$ and  $\CF_1\subset \CV(\widehat{\CO}_S)$ non-empty open-closed, we have
	$$\CN_V(\CF;B)=\omega^V(\CF)+O_{\CF_0,\varepsilon}(\CL(\CF_1)^{\gamma+\dim V+\varepsilon} h(B)),$$ where the covering exponent $\CL(\CF_1)$ is defined in Definition \ref{def:LCEf}, and the implied constant depends on $\CF_0$ but is uniform with respect to $\CF_1$. 
\end{proposition}
\begin{proof}
	To ease notation, we write $\CL$ for $\CL(\CF_1)$.
	We cover $\CF_1$ by a finite disjoint union of congruence neighbourhoods $\{\CE^\CV(\mathfrak{l};\bxi_i)\}_{i\in I}$ of level $\mathfrak{l}$ with $\CL=\#\CO_S/\mathfrak{l}$:
	$$\CF_1=\bigsqcup_{i\in I} \CE^\CV(\mathfrak{l};\bxi_i).$$
	Our goal is to estimate $\# I$.

	We now fix a finite set $S'\subset\mathfrak{M}(k)$ containing $S$ such that $\CV$ is smooth over $\CO_{S'}$. If $\nu\mid\mathfrak{l}$ and $\nu\not\in S'$, using the Lang--Weil estimate \cite{Lang-Weil} and Hensel's lemma  (see e.g. \cite[Lemma 2.1]{Browning-Loughran}), we obtain
	\begin{equation}\label{eq:LW1}
		\#\CV(\CO_\nu/\mathfrak{m}_\nu^{\operatorname{ord}_{\nu}(\mathfrak{l})})=(\#\BF_\nu)^{(\dim V)\operatorname{ord}_{\nu}(\mathfrak{l})}\left(1+O((\#\BF_\nu)^{-\frac{1}{2}})\right),
	\end{equation} where the implied constant depends only on $\CV$.
	By Lemma \ref{le:modelmeasurecomp}, there exists $n_0\in\BN$ depending only on $\CV$ such that for every $n\geqslant n_0$ and for every $\nu\in S'\setminus S$, 
	\begin{equation}\label{eq:LW2}
		\#\CV(\CO_\nu/\mathfrak{m}_\nu^{n})=(\#\BF_\nu)^{(n-n_0)\dim V}\#\CV(\CO_\nu/\mathfrak{m}_\nu^{n_0})\leqslant (\#\BF_\nu)^{n\dim V}\#\CV(\CO_\nu/\mathfrak{m}_\nu^{n_0}).
	\end{equation} Note that the inequality above trivially holds when $n<n_0$. We define 
	$$A_1:=\prod_{\substack{\nu\in S'\setminus S}}\#\CV(\CO_\nu/\mathfrak{m}_\nu^{n_0}),$$
 which is an absolute constant depending only on $\CV$. Since 
	$$\left(\prod_{\nu\mid\mathfrak{l}}\operatorname{Mod}_{\nu,\operatorname{ord}_\nu(\mathfrak{l})}\right)(\bxi_{i_1})\neq \left(\prod_{\nu\mid\mathfrak{l}}\operatorname{Mod}_{\nu,\operatorname{ord}_\nu(\mathfrak{l})}\right)(\bxi_{i_2}) \quad \text{in}\quad \prod_{\nu\mid\mathfrak{l}}\CV(\CO_\nu/\mathfrak{m}_\nu^{\operatorname{ord}_{\nu}(\mathfrak{l})}) \quad \text{ for any }i_1\neq i_2,$$ by \eqref{eq:LW1} \eqref{eq:LW2}, we obtain
	\begin{align*}
		\#I&\leqslant \prod_{\nu\mid\mathfrak{l}}\#\CV(\CO_\nu/\mathfrak{m}_\nu^{\operatorname{ord}_{\nu}(\mathfrak{l})}),\\&\leqslant
\prod_{\substack{\nu\mid \mathfrak{l}, \nu\not\in S'}}(\#\BF_\nu)^{(\dim V)\operatorname{ord}_{\nu}(\mathfrak{l})}\left(1+O((\#\BF_\nu)^{-\frac{1}{2}})\right)		\times \prod_{\substack{\nu\mid\mathfrak{l},\nu\in S'\setminus S}}(\#\BF_\nu)^{(\dim V)\operatorname{ord}_{\nu}(\mathfrak{l})}\#\CV(\CO_\nu/\mathfrak{m}_\nu^{n_0})\\ &\leqslant A_1 A_2^{\#\{\nu:\nu\mid \mathfrak{l}\}}\prod_{\nu\mid \mathfrak{l}}(\#\BF_\nu)^{(\dim V)\operatorname{ord}_{\nu}(\mathfrak{l})}\ll_\varepsilon \CL^{\dim V+\varepsilon},
	\end{align*} where $A_2>0$ is another absolute constant arising from  the implied constant in \eqref{eq:LW1}.
	
	Applying condition \textbf{(EE)} to each $\CF_0\times\CE^\CV(\mathfrak{l};\bxi_i)\subset V(\RA_k)^{\operatorname{Br}}$, we finally obtain \begin{align*}
		\CN_V(\CF;B)=&\sum_{i\in I}\CN_V(\CF_0\times\CE^\CV(\mathfrak{l};\bxi_i);B)\\ =&\left(\sum_{i\in I}\omega^V(\CF_0\times\CE^\CV(\mathfrak{l};\bxi_i))\right)+O_{\CF_0}\left((\# I)\CL^\gamma h(B)\right)\\ =&\omega^V(\CF)+O_{\CF_0,\varepsilon}(\CL^{\gamma+\dim V+\varepsilon} h(B)).
	\end{align*}
	The proof is thus completed.
\end{proof}

\subsection{The Selberg sieve for almost-Fano varieties}\label{se:Selbersieve}
In this subsection, based on condition \textbf{(EE)}, we seek to establish an effective sieving result for the number of rational points of bounded height on almost-Fano varieties satisfying local conditions ``modulo prime powers''. The main result is Theorem \ref{thm:Selbergsieve}  as an effective version of Theorem \ref{thm:Tamagawazero}, which generalises \cite[Theorem 1.7]{Browning-Loughran} where $V$ is a projective quadratic hypersurface. Our key analytic input is (a generalisation of) Selberg's sieve (Theorem \ref{thm:Selberg}), best suitable for our application. 

For $\CP$ any finite set of places of a number field $k$, we use the notation $\Pi(\CP):=\prod_{\nu\in\CP} \mathfrak{p}_\nu$. For any ideal $\mathfrak{b}$ we let $\CP(\mathfrak{b}):=\{\nu\notin\infty_k:\nu\mid \mathfrak{b}\}$.
Let $\mathbf{N}$ be the norm map of ideals, i.e. $\mathbf{N}(\mathfrak{b})=\#\left(\CO_k/\mathfrak{b}\right)$. Define the function $\tau_3$ by $\tau_3(\mathfrak{b}):=\sum_{(\mathfrak{b}_1,\mathfrak{b}_2,\mathfrak{b}_3):\mathfrak{b}=\mathfrak{b}_1\mathfrak{b}_2\mathfrak{b}_3}1$.
Let $\CA=(a_{\mathfrak{b}})_{\mathfrak{b}\in I}$ be a sequence of non-negative real numbers indexed by a finite set $I$ of ideals of $\CO_k$.
The following provides an upper bound estimate for the sifting function
$$S(\CA,\CP):=\sum_{\mathfrak{b}\in I:(\mathfrak{b},\Pi(\CP))=1}a_{\mathfrak{b}}.$$
\begin{theorem}[Selberg, cf. \cite{I-K} Theorem 6.4, \cite{Rieger} Satz 1]\label{thm:Selberg}
	Let $\mathfrak{X}>0$ and let $g$ be a multiplicative arithmetic function supported on square-free ideals satisfying $0<g(\mathfrak{p}_\nu)<1$ for every $\nu\in\CP$ and $g(\mathfrak{p}_\nu)=0$ for every $\nu\notin\CP$. Then for every $D>1$, we have
	$$S(\CA,\CP)\leqslant \frac{\mathfrak{X}}{J(\CP,D)}+\sum_{\substack{\mathfrak{d}\mid \Pi(\CP)\\\mathbf{N}(\mathfrak{d})<D}}\tau_3(\mathfrak{d})\left|\sum_{\mathfrak{b}\in I:\mathfrak{d}\mid \mathfrak{b}}a_{\mathfrak{b}}-g(\mathfrak{d})\mathfrak{X}\right|,$$ where $$J(\CP,D):=\sum_{\substack{\mathfrak{d}\mid \Pi(\CP)\\\mathbf{N}(\mathfrak{d})<\sqrt{D}}}\prod_{\nu\mid \mathfrak{d}}\frac{g(\mathfrak{p}_\nu)}{1-g(\mathfrak{p}_\nu)}.$$ 
\end{theorem}

Now we fix a variety $V$ as in Setting \ref{hyp:almost-Fano}, with a fixed proper $\CO_S$-integral model $\CV$, where $S$  is a finite set of places containing $\infty_k$. For every collection $(\Omega_{\nu})_{\nu\notin S}$ of non-empty open-closed sets $\Omega_\nu\subset \CV(\CO_\nu)$, and for every (not necessarily finite) set of places $S'\subset \mathfrak{M}(k)\setminus S$, we let 
$$\CE^{\CV}_{(S')}[(\Omega_{\nu})]:=\prod_{\nu\in S'}\Omega_\nu\times \prod_{\nu'\notin S'\cup S}\CV(\CO_{\nu'})\times V(k_S)\subset V(\RA_k).$$
  Now for every $N\geqslant 1$, we consider the finite set of places \begin{equation}\label{eq:CPN}
	\CP_N:=\{\nu\notin S:\#\BF_v<N,0<\omega_\nu^V(\Omega_\nu)<\omega_\nu^V( \CV(\CO_\nu))\}.
\end{equation} 
\begin{theorem}\label{thm:Selbergsieve}
	Assume that condition \textbf{(EE)} holds with exponent $\gamma$ and function $h$ with respect to $(\CV,\mathfrak{F})$. Let $(\Omega_\nu)_{\nu\notin S}$ be a collection of open-closed non-empty sets. Assume that there exists $n_0\in\BZ_{>0}$ such that for all $\nu\notin S$, \begin{equation}\label{eq:CLbdk}
		\CL\left(\Omega_{\nu}\times\prod_{\substack{\nu^\prime\notin S,\nu'\neq \nu}}\CV(\CO_{\nu'})\right)\leqslant (\#\BF_v)^{n_0}.
	\end{equation} (Recall Definition \ref{def:LCEf} for $\CL$.) Assume moreover $V(\RA_k)^{\operatorname{Br}}=V(\RA_k)$. Then for every $N\geqslant 1$,
	$$\CN_V(\CE^{\CV}_{(\CP_N)}[(\Omega_\nu)];B)\ll_\varepsilon G(N)^{-1}+N^{2n_0(\gamma+\dim V)+2+\varepsilon} h(B),$$ where the function $G$ is defined by (recall $\Omega_\nu^c:= V(k_\nu)\setminus \Omega_\nu$) \begin{equation}\label{eq:GN}
		G(x):=\sum_{\substack{\mathfrak{d}\mid \Pi(\CP_N)\\\mathbf{N}(\mathfrak{d})<x }}\prod_{\substack{\nu\mid \mathfrak{d}}}\frac{\omega_\nu^V(\Omega_\nu^c)}{\omega_\nu^V(\Omega_\nu)}.
	\end{equation}
 The implied constant above is uniform regarding $B,N$ and the family $(\Omega_\nu)_{\nu\notin S}$.
\end{theorem}
\begin{remarks}
	\hfill
	\begin{enumerate}
		\item One flexibility of executing the Selberg sieve is that no condition on the sieve dimension is required, compared to the Brun combinatorial sieve (cf. \cite[Fundamental Lemma 6.3, Condition (6.47)]{I-K}).
		\item 	Under the assumption \eqref{eq:CLbdk}, we have for every ideal $\mathfrak{b}$ with $\nu\mid\mathfrak{b}\Rightarrow\nu\notin S$, 
		\begin{equation}\label{eq:bdCL}
			\CL\left(\prod_{\nu\mid \mathfrak{b}}\Omega_\nu\times \prod_{\nu\notin S,\nu\nmid\mathfrak{b}}\CV(\CO_\nu)\right)\leqslant \prod_{\nu\mid\mathfrak{b}}(\#\BF_\nu)^k\leqslant \mathbf{N}(\mathfrak{b})^k,
		\end{equation} and Proposition \ref{prop:EEimpliesgeneral} gives
		\begin{equation}\label{eq:CLEE}
			\CN_V(\CE^{\CV}_{(\CP_N)}[(\Omega_\nu)];B)=\omega^V(\CE^{\CV}_{(\CP_N)}[(\Omega_\nu)])+O_\varepsilon\left(\mathbf{N}(\Pi(\CP_N))^{n_0c+\varepsilon}h(B)\right)
		\end{equation} for a certain $c>0$.
		The key point of Theorem \ref{thm:Selbergsieve} is that, although it only provides an upper bound, the secondary term which depends polynomially on $N$ is considerably small, compared to the one in \eqref{eq:CLEE} which grows exponentially in $N$. This is crucial in obtaining the estimates in Theorems \ref{thm:ratptsfibration} and \ref{thm:primevalue}.
	\end{enumerate}
\end{remarks}
\begin{proof}[Proof of Theorem \ref{thm:Selbergsieve}]
	Define the set of square-free ideals $$I:=\{\mathfrak{b}:\mathfrak{b}\mid\Pi(\CP_N)\}$$ of $\CO_k$ and let
	$$g(\mathfrak{p}_\nu):=\begin{cases}
		\frac{\omega_\nu^V(\Omega_\nu^c)}{\omega_\nu^V(V(k_\nu))} &\text{ if } \nu\in \CP_N;\\ 0 &\text{ otherwise}.
	\end{cases}$$ We extend $g$ to a multiplicative function supported on the set $I$ and clearly $0<g(\mathfrak{p}_\nu)<1$ whenever $\nu\in\CP_N$.
	
	We consider the sequence $\CA=(a_{\mathfrak{b}}(B))_{\mathfrak{b}\in I}$ where $$a_{\mathfrak{b}}(B):=\CN_V\left(\CE^{\CV}_{(\CP(\mathfrak{b}))}[(\Omega_{\nu}^c)]\setminus\left(\bigcup_{\mathfrak{b}'\in I\setminus\{\mathfrak{b}\},\mathfrak{b}\mid\mathfrak{b}'}\CE^{\CV}_{(\CP(\mathfrak{b}'))}[(\Omega_{\nu}^c)]\right);B\right).$$ 
	Since $\Omega_{\nu}^c\neq\varnothing$ for every $\nu\in \CP_N$, by Lemma \ref{le:CLCEfc}, by the assumption \eqref{eq:CLbdk}, and using the bound \eqref{eq:bdCL}, condition \textbf{(EE)} and Proposition \ref{prop:EEimpliesgeneral} imply that for every $\mathfrak{d}\mid \Pi(\CP_N)$, 
	\begin{equation}\label{eq:sumab}
	\begin{split}
		\sum_{\mathfrak{b}\in I,\mathfrak{d}\mid \mathfrak{b}}a_{\mathfrak{b}}(B)&=\CN_V\left(\CE^{\CV}_{(\CP(\mathfrak{d}))}[(\Omega_{\nu}^c)];B\right)\\ 	&=\left(\prod_{\nu\mid\mathfrak{d}}\frac{\omega_\nu^V(\Omega_\nu^c)}{\omega_\nu^V(V(k_\nu))}\right)\omega^V(V(\RA_k))+O_\varepsilon\left(\mathbf{N}(\mathfrak{d})^{n_0(\gamma+\dim V)+\varepsilon}h(B)\right)\\ &=\left(\prod_{\nu\mid \mathfrak{d}} g(\mathfrak{p}_\nu)\right)\mathfrak{X}+O_\varepsilon\left(\mathbf{N}(\mathfrak{d})^{n_0(\gamma+\dim V)+\varepsilon}h(B)\right),
	\end{split} 	
\end{equation} where $\mathfrak{X}:=\omega^V(V(\RA_k))$.

	We are now in a position to apply Theorem \ref{thm:Selberg}.
	With the notation above, \begin{align*}
		\CN_V(\CE^{\CV}_{(\CP_N)}[(\Omega_\nu)];B)= S(\CA,\CP_N).
	\end{align*}
	We obtain that uniformly for $N\geqslant 1$, on using \eqref{eq:sumab}, on taking $D=N^2$ in Theorem \ref{thm:Selberg}, and on using the bound $\tau_3(\mathfrak{d})\ll_\varepsilon \mathbf{N}(\mathfrak{d})^\varepsilon$, we obtain
	\begin{align*}
		\CN_V(\CE^{\CV}_{(\CP_N)}[(\Omega_\nu)];B)&\ll_\varepsilon \frac{\mathfrak{X}}{J(\CP_N,N^2)}+h(B)\sum_{\substack{\mathfrak{d}\mid\Pi(\CP_N)\\\mathbf{N}(\mathfrak{d})<N^2}}\tau_3(\mathfrak{d})\mathbf{N}(\mathfrak{d})^{n_0(\gamma+\dim V)+\varepsilon}\\ &\ll_\varepsilon \left(\sum_{\substack{\mathfrak{d}\mid \Pi(\CP_N)\\ \mathbf{N}(\mathfrak{d})<N}}\prod_{\substack{\nu\mid \mathfrak{d}}}\frac{g(\mathfrak{p}_\nu)}{1-g(\mathfrak{p}_\nu)}\right)^{-1}+h(B)\sum_{\mathfrak{d}:\mathbf{N}(\mathfrak{d})<N^2}\mathbf{N}(\mathfrak{d})^{n_0(\gamma+\dim V)+\varepsilon}\\ &\ll G(N)^{-1}+N^{2n_0(\gamma+\dim V+1)+\varepsilon} h(B).
	\end{align*}
	This finishes the proof.
\end{proof}
\begin{corollary}\label{cor:Selberg}
	Under the assumption of Theorem \ref{thm:Selbergsieve}, 
	uniformly  for $B, N$ and $(\Omega_\nu)_{\nu\notin S}$, we have
	$$\CN_V(\CE^{\CV}_{(\mathfrak{M}(k)\setminus S)}[(\Omega_\nu)];B)\ll_\varepsilon G(N)^{-1}+N^{2n_0(\gamma+\dim V+1)+\varepsilon} h(B).$$
\end{corollary}
\begin{proof}
	This follows directly from Theorem \ref{thm:Selbergsieve}, since for every $N\geqslant 1$, \begin{equation*}
		\CN_V(\CE^{\CV}_{(\mathfrak{M}(k)\setminus S)}[(\Omega_\nu)];B)\leqslant \CN_V(\CE^{\CV}_{(\CP_N)}[(\Omega_\nu)];B).\qedhere
	\end{equation*}
\end{proof}

\section{Universal torsor method}\label{se:univtor}
The notion of \emph{universal torsors} is first introduced by Colliot-Thélène and Sansuc in \cite{CT-Sansuc}. Roughly speaking, having simpler geometry, they are torsors under the Néron-Severi torus, and rational points can be lifted to integral points on universal torsors, which are often easier to handle. Main references of this section are \cite[\S5]{Salberger}, \cite[\S4]{Peyre2}, \cite[\S4.3]{PeyreBeyond}.

\subsection{Lifting into universal torsors}\label{se:lifting}
Let $V$ be an almost-Fano variety as in Setting \ref{def:almostFano}. Universal torsors over $V$ exist (cf. \cite[\S2.0.4]{CT-Sansuc}, \cite[Construction 4.20]{PeyreBeyond}).
By the general theory due to Colliot-Th\'el\`ene--Sansuc \cite[Théorème 2.7.3]{CT-Sansuc}, there exists a finite family of \emph{versal torsors} \begin{equation}\label{eq:pii}
	(\pi_i:\CTUT_i\to V)_{i\in I}
\end{equation} under the N\'eron-Severi torus $\CTNS$ (i.e., $k$-forms of universal torsors) such that $\CTUT_i(k)\neq\varnothing$ and 
$$V(k)=\bigsqcup_{i\in I} \pi_i(\CTUT_i(k)).$$
Write $$d:=\dim V,\quad r:=\dim\CTNS,\quad n:=\dim\CTUT_i,$$ so that $n=r+d$. 
\begin{hypothesis}[cf. \cite{Peyre2} \S4.3 \S4.4, \cite{PeyreBeyond} p. 245--p. 252]\label{hyp:univtor}
Let $V$ be as in Setting \ref{hyp:almost-Fano}. There exists $S$ a finite set of places containing $\infty_k$ such that, for every $i\in I$:
	\begin{itemize}
		\item The map $\pi_i$ extends to an $\CO_S$-torsor $\widetilde{\pi}_i:\FTUT_i\to\CV$ under $\FTNS$, where $\FTUT_i,\CV,\FTNS$ are smooth integral models of $\CTUT_i,V,\CTNS$ respectively over $\CO_S$ with $\CV$ proper, such that the map $\widetilde{\pi}_{i,\nu}:\FTUT_i(\CO_\nu)\to\CV(\CO_{\nu})=V(k_\nu)$ is surjective for all $\nu\notin S$;
		\item There exists a family of local measures $(\omega_\nu^{\CTUT_i})_{\nu\notin S}$ so that the measures $\omega^V_\nu,\omega_\nu^{\CTUT_i}$ ``patch together'' for all $\nu\notin S$. Namely, for every open-closed $\CE_\nu\subset \CV(\CO_\nu)$, \begin{equation}\label{eq:patch}
		\omega_\nu^{\CTUT_i}(\widetilde{\pi}_{i}^{-1}(\CE_\nu))=L_\nu(1,\Pic(\overline{V}))^{-1}\omega_\nu^V(\CE_\nu);
		\end{equation}
		\item There exists a fundamental domain $\varDelta_{i,S}\subset \CTUT_i(k_S)$
	 under the action of $\FTNS(\CO_S)$ modulo $W(\CTNS)$ on $\CTUT_i(k_S)$, where  $W(\CTNS)$ is the torsion subgroup of $\CTNS(k)$.
	\end{itemize}
\end{hypothesis}

\begin{remarks}
	\hfill
	\begin{enumerate}
		\item 	If $\CTNS$ is split (i.e. $\CTNS\simeq_k \mathbb{G}_{\operatorname{m},k}^{r}$), then $\# I =(\#\operatorname{Pic}(\CO_S))^r$ (cf. \cite[Theorem 2.7]{Frei-Pieropan}).
		\item  If $\FTNS(\CO_S)=W(\CTNS)$ (e.g. when $S=\infty_k$ and $k=\BQ$ or an imaginary quadratic field), then we have $\varDelta_{i,S}=\CTUT_i(k_S)$.
	\end{enumerate}
	Smooth proper split toric varieties over $\BQ$ studied in this article satisfy both conditions above. 
\end{remarks}

We define the measure
\begin{equation}\label{eq:measS}
\widehat{\omega}_S^{\CTUT_i}:=d_k^{-\frac{n}{2}}\bigotimes_{\nu\notin S}\omega_\nu^{\CTUT_i}\text{ on } \FTUT_i(\widehat{\CO}_S).
\end{equation} In view of \eqref{eq:patch}, the absolute convergence of the product measure $\widehat{\omega}_S^{\CTUT_i}$ follows from that of $\omega^V$ \eqref{eq:Tamagawameas}.
Consider the natural maps induced by \eqref{eq:pii} $$\pi_{i,S}:\CTUT_i(k_S)\to V(k_S),\quad \widehat{\pi}_{i,S}:\FTUT_i(\widehat{\CO}_S)\to \CV(\widehat{\CO}_S).$$ Let $\widetilde{\omega}_{i,S}^V$ be the restriction of $\widetilde{\omega}_S^V$ \eqref{eq:omegaSV} with support in $\pi_{i,S}\left(\CTUT_i(k_S)\right)$, which is an open-closed subset of $V(k_S)$ (cf. \cite[Proposition 5.20]{Salberger}).  Let $\tau(\CTNS)$ be the \emph{Tamagawa number} of $\CTNS$ (cf. \cite[D\'efinition 3.1.5]{Peyre2}, \cite[p. 251]{PeyreBeyond}).
\begin{hypothesis1}\label{hyp:Tamamagauniv}
	For every $i\in I$, there exists $(\Theta_i(B))_B$ a family of counting measures on $V(k_S)\times\FTUT_i(\widehat{\CO}_S)$ which converges weakly to  the measure \begin{equation}\label{eq:Tamagamameaunivtor}
		\frac{\tau(\CTNS)\# W(\CTNS)}{\beta(V)}\widetilde{\omega}_{i,S}^V\otimes \widehat{\omega}_S^{\CTUT_i}.
	\end{equation}
\end{hypothesis1}
\begin{theorem}[cf. \cite{Salberger} Theorem 6.19, \cite{Peyre2} Théorème 5.3.1, \cite{PeyreBeyond} Theorem 4.33]\label{thm:equidistunivtor}
	Under Setting \ref{hyp:univtor} and Hypothesis \ref{hyp:Tamamagauniv}, $V$ satisfies Hypothesis \ref{hyp:deltaB} with respect the family of counting measures $(\Delta(B))_B$ defined by \begin{equation}\label{eq:DeltaTheta}
		\Delta(B):=\sum_{i\in I}(\operatorname{Id}_{V(k_S)}\times\widehat{\pi}_{i,S})_{*}\Theta_i(B).
	\end{equation}
\end{theorem}
\begin{remark}
	Let $H:V(k)\to\BR$ be the associated anticanonical height function, and $M$ be a fixed thin set. Let
	\begin{equation}\label{eq:Thetai}
		\Theta_i(B):=\frac{1}{\alpha(V)\beta(V)B(\log B)^{r-1}}\sum_{\substack{y\in \CTUT_i(k)\cap (\varDelta_{i,S}\times\FTUT_i(\widehat{\CO}_S))\\\pi_i(y)\notin M, H(\pi_i(y))\leqslant B}}\delta_{\pi_{i,S}((y)_{\nu\in S})\times (y)_{\nu\notin S}}
	\end{equation} on $V(k_S)\times\FTUT_i(\widehat{\CO}_S)$. 
	Theorem \ref{thm:equidistunivtor} shows that if $(\Theta_i(B))_B$ converges weakly to \eqref{eq:Tamagamameaunivtor} for all $i\in I$, then $V$ satisfies Principle \ref{prin:Manin-Peyre}.
\end{remark}
\begin{proof}[Proof of Theorem \ref{thm:equidistunivtor}]
	 Indeed, for every $\widetilde{\omega}_S^V$-continuous $\CE_1\subset V(k_S)$ and for every open-closed $\CE_2\subset \CV(\widehat{\CO}_S)$ such that $\CE:=\CE_1\times\CE_2\subset V(\RA_k)^{\operatorname{Br}}$, let the counting function  $\CN_{V}(\CE;B)$ be defined by \eqref{eq:countingDelta} regarding \eqref{eq:DeltaTheta}. For every $i\in I$ and for every open-closed $\CF_i\subset \FTUT_i(\widehat{\CO}_S)$, consider the counting function  \begin{equation}\label{eq:UTcounting}
	\CN_{\FTUT_i}(\CE_1,\CF_i;B):=\int_{\CE\times\CF_i} \operatorname{d}\Theta_i(B).
	\end{equation}
	It follows from Hypothesis \ref{hyp:Tamamagauniv} and the properties in  Setting \ref{hyp:univtor} that 
	\begin{align*}
		\CN_{V}(\CE;B)&=\frac{1}{\# W(\CTNS)}\sum_{i\in I}\CN_{\FTUT_i}(\CE_1,\widehat{\pi}_{i,S}^{-1}(\CE_2);B)\sim \frac{\tau(\CTNS)}{\beta(V)}\sum_{i\in I}\widetilde{\omega}_{i,S}^V\left(\CE_1\right) \widehat{\omega}_S^{\CTUT_i}(\widehat{\pi}_{i,S}^{-1}(\CE_2)).	\end{align*} 
	To complete the proof it remains to show that
	\begin{equation}\label{eq:Ono0}
		\frac{\tau(\CTNS)}{\beta(V)}\sum_{i\in I}\widetilde{\omega}_{i,S}^V\left(\CE_1\right) \widehat{\omega}_S^{\CTUT_i}(\widehat{\pi}_{i,S}^{-1}(\CE_2))=\omega^V(\CE).
	\end{equation}
	
 By a theorem of Salberger \cite[Lemma 6.17]{Salberger}, we have for any $x\in V(\RA_k)^{\operatorname{Br}}$, 
	$$\#\Sha^1(k,\CTNS)=\#\{i\in I:x\in \pi_i(\CTUT_i(\RA_k))\},$$ 	 where
	$$\Sha^1(k,\CTNS):=\ker\left(H^1(k,\CTNS)\to \prod_{\nu\in\mathfrak{M}(k)}H^1(k_\nu,\CTNS)\right).$$ We then have, on recalling \eqref{eq:measS},
	\begin{equation}\label{eq:Ono2}
		\begin{split}
				\sum_{i\in I}\widetilde{\omega}_S^V\left(\CE_1\cap\pi_{i,S}\left(\CTUT_i(k_S)\right)\right)&=d_k^{-\frac{n}{2}}\left(\bigotimes_{\nu\notin S}L_\nu(1,\Pic(\overline{V}))^{-1}\omega_\nu^V\right)(\CE_2)\sum_{i\in I}\widetilde{\omega}_S^V\left(\CE_1\cap\pi_{i,S}\left(\CTUT_i(k_S)\right)\right)\\&=\frac{\#\Sha^1(k,\CTNS)}{d_k^{\frac{n}{2}}}\widetilde{\omega}_S^V(\CE_1).
		\end{split}
	\end{equation}
	By Ono's theorem \cite[Main Theorem]{Ono}, we have
\begin{equation}\label{eq:Ono1}
	\tau(\CTNS)=\frac{d_k^{\frac{r}{2}}\beta(V)}{\# \Sha^1(k,\CTNS)}\left(\lim_{s\to 1}(s-1)^rL_S(s,\Pic(\overline{V}))\right).
\end{equation}
	Comparing \eqref{eq:Ono2} \eqref{eq:Ono1} with \eqref{eq:Tamagawameas}, we finally deduce \eqref{eq:Ono0}.
Theorem  \ref{thm:portmanteau} then allows to conclude the proof.
\end{proof}

\subsection{The effective equidistribution condition on universal torsors (EEUT)}\label{se:EEUT}
In this subsection, we continue to work under Setting \ref{hyp:univtor}.
Let $\mathfrak{F}$ be a family of $\widetilde{\omega}_S^V$-continuous sets contained in $\operatorname{pr}_S(V(\RA_k)^{\operatorname{Br}})$. 
Our goal now is to formulate a quantitative version of Hypothesis \ref{hyp:Tamamagauniv} for $\mathfrak{F}$ and for the family of versal torsors $(\FTUT_i)_{i\in I}$. We recall the counting function \eqref{eq:UTcounting}.
\begin{condEEUT}
There exist $\gamma_0\geqslant 0$ and a continuous function $h_0:\BR_{>0}\to\BR_{>0}$ with $h_0(x)=o(1),x\to\infty$, satisfying the following property:

		Let
		$\CF\in\mathfrak{F}$. For every $i\in I$, let $\CE^{\FTUT_i}(\mathfrak{l};\Xi)\subset \FTUT_i(\widehat{\CO}_S)$ be a congruence neighbourhood associated to  $\FTUT_i,\Xi\in\prod_{\nu\mid\mathfrak{l}}\FTUT_i(\CO_\nu)$ of level $\mathfrak{l}$. Then we have, as $B\to\infty$,
		\begin{equation*}
			\CN_{\FTUT_i}(\CF,\CE^{\FTUT_i}(\mathfrak{l};\Xi);B)=
			 \frac{\tau(\CTNS)\# W(\CTNS)}{\beta(V)}\widetilde{\omega}_{i,S}^V(\CF)\widehat{\omega}_S^{\CTUT_i}(\CE^{\FTUT_i}(\mathfrak{l};\Xi))+O_{\CF}((\#\CO_S/\mathfrak{l})^{\gamma_0} h_0(B)),	
		\end{equation*}
		uniformly for every such $\mathfrak{l},\Xi$, where the implied constant depends on $\CF$.
\end{condEEUT}
\begin{proposition}\label{prop:univtorEEGS}
Condition \textbf{(EEUT)} for $((\FTUT_i)_{i\in I},\mathfrak{F})$  implies condition \textbf{(EE)} for $(\CV,\mathfrak{F})$ (regarding \eqref{eq:DeltaTheta})
 with $$\gamma=\gamma_0+n+\varepsilon\quad\text{and}\quad h=h_0,$$
where we recall $n=\dim \CTUT_i$ and $\varepsilon>0$ is arbitrary. 
\end{proposition}
\begin{proof}
We fix throughout $\CF:=\CF_0\times \CE^\CV(\mathfrak{l};\bxi)\subset V(\RA_k)^{\operatorname{Br}}$, where $\CF_0\in\mathfrak{F}$ and $\CE^\CV(\mathfrak{l};\bxi)\subset \CV(\widehat{\CO}_S)$ is a congruence neighbourhood  associated to  $\CV,\bxi=(\xi_\nu)_{\nu\mid \mathfrak{l}}\in\prod_{\nu\mid \mathfrak{l}} \CV(\CO_\nu)$ of level $\mathfrak{l}$.
Now we fix $i\in I$. Then for every $\Xi_i\in \prod_{\nu\mid\mathfrak{l}}\widetilde{\pi}_{i,\nu}^{-1}(\xi_\nu)\subset\prod_{\nu\mid \mathfrak{l}}\FTUT_i(\CO_\nu)$, we have $\widehat{\pi}_{i,S}(\CE^{\FTUT_i}(\mathfrak{l};\Xi_i))\subset \CE^\CV(\mathfrak{l};\bxi)$. 
	Consequently, we can find finitely many such $\Xi_{i,k}\in\prod_{\nu\mid\mathfrak{l}}\FTUT_i(\CO_\nu)$ so that 
	$$\widehat{\pi}_{i,S}^{-1}(\CE^\CV(\mathfrak{l};\bxi))=\bigsqcup_{\Xi_{i,k}}\CE^{\FTUT_i}(\mathfrak{l};\Xi_{i,k}).$$ The number of such $\Xi_{i,k}$ is, by the Lang-Weil estimate \cite{Lang-Weil} and Hensel's lemma as in the proof of Proposition \ref{prop:EEimpliesgeneral}, $$\leqslant\#\FTUT_i(\CO_S/\mathfrak{l}) =\prod_{\nu\mid \mathfrak{l}}\#\FTUT_i(\CO_\nu/\mathfrak{m}_\nu^{\operatorname{ord}_\nu(\mathfrak{l})})\ll_\varepsilon(\#\CO_S/\mathfrak{l})^{n+\varepsilon}.$$
	It follows from condition \textbf{(EEUT)} and \eqref{eq:Ono0} in the proof of Theorem \ref{thm:equidistunivtor} that
	\begin{align*}
		\CN_V(\CF;B)=&\sum_{i\in I}\sum_{\Xi_{i,k}}\frac{1}{\#W(\CTNS)}\CN_{\FTUT_i}(\CF_0,\CE^{\FTUT_i}(\mathfrak{l};\Xi_{i,k});B)\\ =&\frac{\tau(\CTNS)}{\beta(V)}
		\sum_{i\in I}\sum_{\Xi_{i,k}}\left(\widetilde{\omega}_{i,S}^V\left(\CF_0\right)\widehat{\omega}_S^{\CTUT_i}(\CE^{\FTUT_i}(\mathfrak{l};\Xi_{i,k}))+O_{\CF_0}((\#\CO_S/\mathfrak{l})^{\gamma_0} h_0(B))\right)\\ =&\frac{\tau(\CTNS)}{\beta(V)}\sum_{i\in I}\widetilde{\omega}_{i,S}^V\left(\CF_0\right)\sum_{\Xi_{i,k}} \widehat{\omega}_S^{\CTUT_i}(\CE^{\FTUT_i}(\mathfrak{l};\Xi_{i,k}))+O_{\CF_0,\varepsilon}((\#\CO_S/\mathfrak{l})^{\gamma_0+n+\varepsilon} h_0(B))\\ =&\frac{\tau(\CTNS)}{\beta(V)}\sum_{i\in I}\widetilde{\omega}_{i,S}^V\left(\CF_0\right)\widehat{\omega}_S^{\CTUT_i}(\widehat{\pi}_{i,S}^{-1}(\CE^\CV(\mathfrak{l};\bxi)))+O_{\CF_0,\varepsilon}((\#\CO_S/\mathfrak{l})^{\gamma_0+n+\varepsilon} h_0(B))\\ =&\omega^V(\CF)+O_{\CF_0,\varepsilon}((\#\CO_S/\mathfrak{l})^{\gamma_0+n+\varepsilon} h_0(B)). \qedhere
	\end{align*}
\end{proof}

\section{Parametrisation and heights on toric varieties}\label{se:toricparaheight}
In this section, we recall basic toric geometry, toric construction of universal torsors, toric Tamagawa measures and formulas for toric height functions. 
\subsection{Construction and lifting}
(See \cite[\S8]{Salberger}, \cite[\S2]{Huang}.)

Fix an integral lattice $N\simeq \BZ^d$. Let $N^\vee:=\operatorname{Hom}_\BZ(N,\BZ)$ be the dual lattice.  A fan $\triangle$ that we consider in this article is a collection of (strongly convex, rational polyhedral and simplicial) cones inside $N_\BR\simeq \BR^d$ containing the vector $\mathbf{0}$ and satisfying the following.
\begin{itemize}
	\item Any face of a cone is also a cone in $\triangle$;
	\item The intersection of any two cones in $\triangle$ is the common face of them;
\end{itemize}
We moreover assume that $\triangle$ is \emph{regular} and \emph{complete}, i.e.
\begin{itemize}
	\item Every cone in $\triangle$ is generated by integral vectors which form part of a basis of the lattice $N$;
	\item The support $\bigcup_{\sigma\in\triangle}\sigma$ is $N_\BR$.
\end{itemize}

Let $R$ be either $k$ or $\CO_k$ in what follows. For every $\sigma\in\triangle$, we consider  $\sigma^\vee\subset N^\vee_\BR$ its dual cone and associate the affine open neighbourhood $U_\sigma=\operatorname{Spec}(R[\sigma^\vee\cap N^\vee])$ attached to the semigroup $\sigma^\vee\cap N^\vee$. Clearly if $\tau$ is a face of $\sigma$, then $U_\tau\subset U_\sigma$, all of which contain the open orbit \begin{equation}\label{eq:openorbit}
	\CT_{O}:=U_{\mathbf{0}}=\operatorname{Spec}(R[N^\vee])\simeq \mathbb{G}_{\operatorname{m},R}^d.
\end{equation}  
The toric variety $\operatorname{Tor}_R(\triangle)$ attached to $\triangle$ is constructed by gluing the data $(U_\sigma)_{\sigma\in\triangle}$, which is a smooth, flat and proper scheme over $R$.

Let $\triangle_{\max}$ be the set of all \emph{maximal cones}, i.e. those whose $\BR$-span is of dimension $d$. Let $\triangle(1)\subset \triangle$ be the set of one-dimensional cones, which we shall call \emph{rays}. Each ray $\rho$ corresponds to a $\CT_{O}$-invariant boundary divisor $D_\rho$. We denote by $n_\rho$ the unique primitive element of $\rho\cap N$. Similarly, for every cone $\sigma\in\triangle$, let $\sigma(1)$ be the set of rays of $\sigma$. Consider the map $h:\BZ^{\triangle(1)}\to N$ given by for $(a_\rho)_{\rho\in\triangle(1)}\in \BZ^{\triangle(1)}$, 
$$h((a_\rho)_{\rho\in\triangle(1)}):=\sum_{\rho\in\triangle(1)}a_\rho \rho.$$ Then we have the following exact sequence of $\BZ$-modules  \begin{equation}\label{eq:exactsq}
	\xymatrix{
	0\ar[r]&\ker(h)\ar[r]& \BZ^{\triangle(1)}\ar[r]^-h&
	N\ar[r]& 0.}
\end{equation} We write $r=\operatorname{rank}\ker(h)$, so that $\#\triangle(1)=n:=r+d$, and write $\triangle(1)=\{\rho_1,\cdots,\rho_{r+d}\}$. For every $\sigma\in\triangle_{\max}$, by convention, we call an ordering of rays  \emph{admissible} with respect to $\sigma$ if $\sigma(1)=\{\rho_{r+j},1\leqslant j\leqslant d\}$. 
We can identify $\ker(h)$ as the dual of the Picard group (cf. \cite[Theorem 4.1.3]{Coxbook}). 
The exact sequence \eqref{eq:exactsq} induces the following exact sequence 
\begin{equation}\label{eq:toriexactseq}
	\xymatrix{ 1\ar[r]& \CTNS\ar[r]& \BG_{\operatorname{m},R}^{\triangle(1)}\ar[r]&\CT_{O}\ar[r]& 1}\index{TNS@$\mathcal{T}_{\operatorname{NS}},\widetilde{\mathcal{T}}_{\operatorname{NS}}$}
\end{equation} between split $R$-tori.

To every $\sigma\in\triangle$, we associate the cone $\sigma_0:=h_{\BR}^{-1}(\sigma)\subset N_{0,\BR}$. It is precisely the cone generated by the unit vectors in $\BR^{\triangle(1)}$ corresponding to the elements of $\sigma(1)$, and it defines an affine scheme $U_{0,\sigma}=\operatorname{Spec}(R[\sigma_0^\vee\cap \BZ^{\triangle(1)}])$. The collection of cones $(\sigma_0)_{\sigma\in\triangle}$ inside $\BR^{\triangle(1)}$ forms a fan $\triangle_0$, to which we associate the toric $R$-scheme $\operatorname{Tor}_R(\triangle_0)$. This is an open toric subscheme of $\BA^{\triangle(1)}_R\simeq \operatorname{Spec}(R[\BZ^{\triangle(1)}])$. 
We write here and after $$X:=\operatorname{Tor}_k(\triangle),\quad X_0:=\operatorname{Tor}_k(\triangle),\quad  \CX:=\operatorname{Tor}_{\CO_k}(\triangle),\quad \CX_0:=\operatorname{Tor}_{\CO_k}(\triangle_0),$$
Then $X/k$ is smooth proper split with the smooth proper integral model $\CX/\CO_k$, and $X_0/k$ is an open subset of $\BA^{\triangle(1)}_k$ and hence quasi-affine, with the quasi-affine integral model $\CX_0/\CO_k$.
The morphisms $U_{0,\sigma}\to U_\sigma,\sigma\in\triangle$ induced by the map $h$ glue together to the toric morphism \begin{equation}\label{eq:morptoric}
	\pi:X_0\to X. 
\end{equation}

\begin{theorem}[Colliot-Thélène--Sansuc \& Salberger, cf. \cite{Salberger} Propositions 8.4, 8.5, Corollary 8.8]\label{thm:univtor}
The quasi-affine variety $X_0$ is a universal torsor over $X$ under the	N\'eron-Severi torus $\mathcal{T}_{\operatorname{NS}}$, which is unique up to $k$-isomorphism.
\end{theorem}
Following Salberger \cite[p. 191]{Salberger}, we call $X_0$ the \emph{principal universal torsor}.

According to \cite[Theorem 3.2]{Huang} and \cite[p. 419]{Pieropan}, there exists a finite family of ``twisted'' universal torsors 
$(\pi_{\mathbf{c}}:\CX_0^{\mathbf{c}}\to\CX)_{\mathbf{c}\in\mathfrak{c}_k^r}$ of $\CX_0$ over $\CO_k$, such that $$X(k)=\bigsqcup_{\mathbf{c}\in\mathfrak{c}_k^r}\pi_{\mathbf{c}}\left(\CX_0^{\mathbf{c}}(\CO_k)\right).$$ The coprimality condition for the $(r+d)$-tuples in $\CX_0^{\mathbf{c}}(\CO_k)$ is explicitly described by \cite[Theorem 2.7]{Frei-Pieropan} and \cite[(26)]{Huang}.
\subsection{Toric norms and toric Tamagawa measures}
(See \cite[\S9]{Salberger}.)

The corresponding divisor $D_0$ of the anti-canonical line bundle $K_X^{-1}$ satisfies (cf. \cite[Theorem 8.2.3]{Coxbook}) $$D_0=\sum_{\rho\in\triangle(1)}D_{\rho}.$$ 
From now on we suppose that $K_X^{-1}$ is globally generated. 
For every $\sigma\in\triangle_{\max}$, let $m_{D_0}(\sigma)\in M$ be the unique element characterised by the property that $\langle m_{D_0}(\sigma),n_\rho\rangle=1$ for every $\rho\in\sigma(1)$.\footnote{This notation adopted here follows \cite{Salberger}. It differs from \cite{Huang} by a minus sign.} 
For every $\sigma\in\triangle_{\max}$, let \begin{equation}\label{eq:D0sigma}
	D_0(\sigma):=D_0+\sum_{\rho\in\triangle(1)}\langle -m_{D_0}(\sigma),n_{\rho}\rangle D_{\rho}.
\end{equation} Since $D_0$ is globally generated, $D_0(\sigma)$ is an effective divisor with support in $\cup_{\rho\in\triangle(1)\setminus\sigma(1)} D_\rho$ for every $\sigma\in\triangle_{\max}$.

For every $(r+d)$-tuple $\XX=(X_\rho)\in k^{\triangle(1)}$, and for every $D=\sum_{\rho\in\triangle(1)}a_\rho D_\rho$, such that $a_\rho=0\Rightarrow X_\rho\neq 0$, we shall use the notation
\begin{equation}\label{eq:XP0}
	\XX^{D}:=\prod_{\rho\in\triangle(1)}X_\rho^{a_{\rho}}.
\end{equation}

For every $\sigma\in\triangle_{\max}$, let us choose an admissible ordering such that $\sigma(1)=\{\rho_{r+1},\cdots,\rho_{r+d}\}$, and let $\{n_{\rho_{r+1}}^{\vee},\cdots,n_{\rho_{r+d}}^{\vee}\}$ be the corresponding dual base. Let \begin{equation}\label{eq:Fj}
	F_\sigma(j):=\sum_{\rho\in\triangle(1)}\langle n_{\rho_{r+j}}^\vee,n_{\rho}\rangle D_\rho=D_{\rho_{r+j}}+\sum_{i=1}^{r}\langle n_{\rho_{r+j}}^\vee,n_{\rho_i}\rangle D_{\rho_i},\quad 1\leqslant j\leqslant d.
\end{equation}
According to the exact sequence \eqref{eq:exactsq}, the restriction of $\pi:X_0\to X$ to $U_{0,\sigma}\to U_\sigma$ can be written in the coordinate form
\begin{equation}\label{eq:univmap}
	\begin{split}
		\XX=(X_\rho)_{\rho\in\triangle(1)}&\longmapsto  \left(z_{r+j}:=\XX^{F_\sigma(j)}\right)_{1\leqslant j\leqslant d},
	\end{split}
\end{equation} which we shall call \emph{the parametrisation given by} (the admissible ordering of) $\sigma$. This is clearly well-defined since $\XX\in U_{0,\sigma}\Rightarrow X_{\rho_i}\neq 0$ for all $1\leqslant i\leqslant r$.

For every $\nu\in\mathfrak{M}(k)$, let us consider the group homomorphism \cite[p. 196]{Salberger} $L_\nu:\CT_{O}(k_\nu)\to N_\BR$ given by regarding $\CT_{O}(k_\nu)=\Hom_\BZ(M,k_\nu^\times),N_\BR=\Hom_\BZ(M,\BR)$ and composing with the map $\log (|\cdot|_\nu):k_\nu^\times\to\BR$. For every  $\sigma\in\triangle_{\max}$, let $C_\sigma(k_\nu)$ be the $\nu$-adic closure of $L_\nu^{-1}(-\sigma)$ in $X(k_\nu)$. Since the fan $\triangle$ is complete, the $\nu$-adic locus $X(k_\nu)$ has the subdivision (cf. \cite[Proposition 9.2]{Salberger})
\begin{equation}\label{eq:divisionRpoints}
	X(k_\nu)=\bigcup_{\sigma\in\triangle_{\max}} C_\sigma(k_\nu).
\end{equation}
Under the parametrisation given by $\sigma$ \eqref{eq:univmap}, if $\mathbf{z}=(z_{r+1},\cdots,z_{r+d})\in \CT_{O}(k_\nu)$, then $$L_{\nu}(\mathbf{z})=\sum_{j=1}^{d}\log |z_{r+j}|_{\nu}n_{\rho_{r+j}}\in -\sigma\Longleftrightarrow \log |z_{r+j}|_{\nu}<0 \text{ for all }j.$$
So  \begin{equation}\label{eq:CsigmaR}
	C_{\sigma}(k_\nu)=\{(z_{r+1},\cdots,z_{r+d})\in k_\nu^d:|z_{r+j}|_\nu\leqslant 1 \text{ for all }j\}.
\end{equation}

Let us now recall the definition of \emph{toric norm} $(\|\cdot\|_{\operatorname{tor},\nu})_{\nu\in\mathfrak{M}(k)}$ on $D_0$.
For every $\nu\in\mathfrak{M}(k)$, and for every $P\in X(k_\nu)$, we choose $\sigma\in\triangle_{\max}$ such that $P\in C_\sigma(k_\nu)$. The associated character $\chi^{m_{D_0}(\sigma)}\in\Hom(\CT_{O},\BG_{\operatorname{m}})$ defines $D_0$ on $U_\sigma$ and lifts to a global section of $K_X^{-1}$. So on $U_\sigma$, $K_X^{-1}$ trivialises and is generated by the section $\chi^{-m_{D_0}(\sigma)}$, and $\chi^{-m_{D_0}(\sigma)}(P)\neq 0$ for every $P\in U_\sigma(k_\nu)$. Now for every local section $s$ of $D_0$ defined at $P$ such that $s(P)\neq 0$, we set (cf. \cite[Proposition 9.2]{Salberger})
\begin{equation}\label{eq:toricnorm}
	\|s(P)\|_{\operatorname{tor},\nu}:=\left|\frac{s}{\chi^{-m_{D_0}(\sigma)}}(P)\right|_\nu.
\end{equation}
It is independent of the cone $\sigma\in\triangle_{\max}$ with $P\in C_\sigma(k_\nu)$.
We write $(\omega^X_{\operatorname{tor},\nu})$ for the associated adelic measures, which we use to define the Tamagawa measure $\omega^X_{\operatorname{tor}}$ (resp. $\omega^X_{\operatorname{tor},f}$) on $X(\RA_k)$ (resp. $X(\RA_k^{\infty_k})$) by \eqref{eq:Tamagawameas}.

\begin{theorem}[cf. \cite{Salberger} Propositions 9.12 \& 9.16]\label{thm:realTamagawameas}
	\hfill	\begin{itemize}
		\item 	Let $\nu\in\mathfrak{M}(k)$ be an archimedean place. Then  under the parametrisation given by any $\sigma\in\triangle_{\max}$ \eqref{eq:univmap}, the adelic measure $\omega^X_{\operatorname{tor},\nu}$ associated to the toric norm $\|\cdot\|_{\operatorname{tor},\nu}$ restricted to $C_\sigma(k_\nu)$ equals the Lebesgue measure $\operatorname{d}z_{r+1}\cdots\operatorname{d}z_{r+d}$;
		\item 	Let $\nu\in\mathfrak{M}(k)$ be a non-archimedean place. Then the toric norm $\|\cdot\|_{\operatorname{tor},\nu}$ \eqref{eq:toricnorm} on $K_X^{-1}$ (resp. the adelic measure $\omega^X_{\operatorname{tor},\nu}$) coincides with the model norm (resp. the model measure) determined by $(\CX,K_{\CX}^{-1})$.
	\end{itemize}
\end{theorem}

We define the toric norm $(\|\cdot\|_{X_0,\nu})_{\nu\in\mathfrak{M}(k)}$ on $K_{X_0}^{-1}$ as follows. For every $\nu\in\mathfrak{M}(k)$, and for every $\XX\in X_0(k_\nu)$ and every local section $s:=f(x_1,\cdots,x_n) \frac{\partial}{\partial x_{1}}\wedge\cdots\wedge\frac{\partial}{\partial x_{n}}$ defined at $\XX$ with $f$ measurable,  (cf. \cite[Theorem 9.12]{Salberger}) 
\begin{equation}\label{eq:normX0}
	\|s(\XX)\|_{X_0,\nu}:=\frac{|f(\XX)|_\nu}{\sup_{\sigma\in\triangle_{\max}}|\XX^{D_0(\sigma)}|_\nu}.
\end{equation}
We let $(\omega^{X_0}_{\nu})_{\nu\in\mathfrak{M}(k)}$ be the family of associated adelic measures.
\begin{theorem}[cf. \cite{Salberger} Proposition 9.14]\label{thm:nonarchTmeas}
	Let $\nu\in\mathfrak{M}(k)$ be a non-archimedean place. Then the toric norm $\|\cdot\|_{X_0,\nu}$ on $K_{X_0}^{-1}$ (resp. the adelic measure $\omega^{X_0}_{\nu}$) coincides with the model norm (resp. the model measure) determined by $(\CX_0,K^{-1}_{\CX_0})$.
\end{theorem}
Define the product measure (cf. \eqref{eq:measS}) \begin{equation}\label{eq:omegaXf}
	\widehat{\omega}_{\infty_k}^{X_0}:=d_k^{-n/2}\bigotimes_{\nu\notin \infty_k}\omega_\nu^{X_0}\text{ on }\CX_0(\widehat{\CO}_k).
\end{equation}

\subsection{Toric height functions}
Let $H_{\operatorname{tor}}=H_{K_X^{-1},(\|\cdot\|_{\operatorname{tor},\nu})}:X(k)\to\BR_{>0}$ be the \emph{toric height function} induced by the toric norm $(\|\cdot\|_{\operatorname{tor},\nu})_{\nu\in\mathfrak{M}(k)}$, i.e., for every $P\in X(k)$ (recall \eqref{eq:heightmetric}), 
\begin{equation}\label{eq:toricheight}
	H_{\operatorname{tor}}(P)=\prod_{\nu\in\mathfrak{M}(k)}\|s(P)\|^{-1}_{\operatorname{tor},\nu},
\end{equation} for any local section $s$ of $K_X^{-1}$ at $P$ with $s(P)\neq 0$. 

\begin{theorem}[cf. \cite{Salberger} Proposition 11.3, \cite{Pieropan} Proposition 2, \cite{Huang} Proposition 3.4]\label{thm:H0}
For every $P\in X(k)$ that lifts to $\XX(P)\in\CX_0^{\mathbf{c}}(\CO_k)$ for a certain $\mathbf{c}\in\mathfrak{c}_k^r$, we have	$$H_{\operatorname{tor}}(P)=\mathbf{N}(\mathbf{c}^{D_0(\sigma_0)})^{-1}\prod_{\nu\in\mathfrak{M}(k)}\sup_{\sigma\in\triangle_{\max}}|\XX(P)^{D_0(\sigma)}|_\nu,$$ for any $\sigma_0\in\triangle_{\max}$.
\end{theorem}

\begin{remark}
	Assume $k=\BQ$. Then $X(\BQ)=\pi(\CX_0(\BZ))$ and $\CX_0(\BZ)$ consists of precisely the $(r+d)$-tuples $(X_\rho)_{\rho\in \triangle(1)}\subset\BZ^{\triangle(1)}$ satisfying the coprimality condition \begin{equation}\label{eq:coprime}
		\gcd_{\sigma\in\triangle_{\max}}\left(\prod_{\rho\notin\sigma(1)}X_\rho\right)=1.
	\end{equation} Every $P\in X(\BQ)$ lifts to $\XX(P)\in \CX_0(\BZ)$, lifts differing by the action of $\FTNS(\BZ)$. The  coprimality condition \eqref{eq:coprime} implies $\gcd\left(\XX(P)^{D_0(\sigma)},\sigma\in\triangle_{\max}\right)=1$ since $\XX^{D_0(\sigma)}$ is a polynomial in $X_\rho,\rho\notin\sigma(1)$ for all $\sigma\in\triangle_{\max}$. Hence by Theorem \ref{thm:H0} the toric height function \eqref{eq:toricheight} has the simple expression (cf. \cite[Proposition 11.3]{Salberger})
$$H_{\operatorname{tor}}(P)=\sup_{\sigma\in\triangle_{\max}}|\XX(P)^{D_0(\sigma)}|,$$ where $|\cdot|$ is the usual absolute value on $\BR$.
\end{remark}

\section{Effective equidistribution of rational points on toric varieties}\label{se:purityproof}
Throughout  this section, we fix $X$ a smooth proper split toric variety over $\BQ$ with its principal universal torsor $X_0$, and we write respectively $\CX_0,\CX$ their canonical integral models over $\BZ$. We continue to assume that the anticanonical line bundle $K_X^{-1}$ is globally generated. Recall the morphism \eqref{eq:morptoric} between them. 

 Recall from \eqref{eq:divisionRpoints} that we have the subdivision $$X(\BR)=\bigcup_{\sigma\in\triangle_{\max}}C_\sigma(\BR)$$ of the real locus. For every $\sigma\in\triangle_{\max}$ with an admissible ordering, recall from \eqref{eq:CsigmaR} that
\begin{equation*}
	C_\sigma(\BR)=\{(z_{r+1},\cdots,z_{r+d})\in\BR^d:|z_{r+j}|\leqslant 1\}\subset X(\BR),
\end{equation*} where the coordinates $(z_{r+j})_{1\leqslant j\leqslant d}$ correspond to the parametrisation given by $\sigma$ \eqref{eq:univmap}. 
 For $0<a_j<b_j\leqslant 1,1\leqslant j\leqslant d$, let us consider  \begin{equation}\label{eq:standardcube}
	\CF_\infty^+(\sigma)=\prod_{j=1}^{d}]a_j,b_j]\subset C_{\sigma}(\BR),
\end{equation} which we call a \emph{standard cube} (with respect to $\sigma\in\triangle_{\max}$).
 The group $\CT_{O}(\BR)/\CT_{O}(\BR)^+\simeq\{\pm 1\}^d$ acts via switching the sign of coordinates. Let \begin{equation}\label{eq:familyF}
 	\mathfrak{F}:=\{(\CT_{O}(\BR)/\CT_{O}(\BR)^+)\cdot\CF_\infty^+(\sigma):\CF_\infty^+(\sigma) \text{ a standard cube}, \sigma\in\triangle_{\max}\}.
 \end{equation} It is clear that $\mathfrak{F}$ forms a topological basis for the real open orbit $\CT_{O}(\BR)$ \eqref{eq:openorbit}.

This section is devoted to the proof of the following theorem. 
\begin{theorem}\label{thm:effectiveEEUT}
	Condition \textbf{(EEUT)} holds for $(\CX_0,\mathfrak{F})$ with $$\gamma_0=\varepsilon,\quad  h_0(B):=(\log B)^{-\frac{1}{2}+\varepsilon},$$ for any $\varepsilon>0$.
\end{theorem}
\begin{corollary}\label{thm:effectiveEE}
Condition \textbf{(EE)} holds for $(\CX,\mathfrak{F})$ with $$\gamma=\dim X+\operatorname{rank}\Pic(X)+\varepsilon,\quad h(B)=(\log B)^{-\frac{1}{2}+\varepsilon},$$ for any $\varepsilon>0$.
\end{corollary}
\begin{proof}[Proof of Corollary \ref{thm:effectiveEE} assuming Theorem \ref{thm:effectiveEEUT}]
	It remains to apply Proposition \ref{prop:univtorEEGS}.
\end{proof}
\begin{proof}[Proof of Theorem \ref{thm:mainequidist} assuming Corollary \ref{thm:effectiveEE}]
	It remains to apply Proposition \ref{prop:EEimpliesgeneral}. 
\end{proof}
The proof of Theorem \ref{thm:effectiveEEUT} will be given at the end of this section.
\subsection{Counting integral points with congruence conditions}
We first state two results due to Salberger. We consider the set (recall the notation \eqref{eq:XP0})
\begin{equation}\label{eq:CAB}
	\CA(B):=\left\{\XX\in\BZ_{>0}^n:\max_{\sigma\in\triangle_{\max}}\XX^{D_0(\sigma)}\leqslant B\right\}. \footnote{Note that Salberger \cite[11.6]{Salberger} uses the notation $A(B)$ to denote the cardinality of this set.}
\end{equation}
Up to a product of local densities, $\CA(B)$ contributes to the constant  $\alpha(X)$ \eqref{eq:alphaV} and the order of growth $B(\log B)^{r-1}$.  
Following  \cite[Notation 11.6 (c)]{Salberger}, we consider the set
\begin{equation}\label{eq:CCB}
	\CC_0(B)^+:=\left\{\XX\in\BZ_{>0}^n:\max_{\sigma\in\triangle_{\max}}\XX^{D_0(\sigma)}\leqslant B,\eqref{eq:coprime}\text{ holds}\right\}.
\end{equation} The gcd condition gives rise to the non-archimedean local densities of $\CX_0$.  

Let the constant $\alpha_0$ denote the minimum of integers $\alpha\in\BZ_{>0}$ such that there exist a collection of $\alpha$ rays of $\triangle$ not contained in any cone of $\triangle$. Note that we always have $\alpha_0\geqslant 2$ since every ray is a cone.
\begin{lemma}[\cite{Salberger} ``Main Lemma'' 11.27]\label{thm:Salbergermainlemma} We have
$$\#\CC_0(B)^+=\kappa\#\CA(B)+O_\varepsilon(B(\log B)^{r-2+\frac{1}{\alpha_0}+\varepsilon}),$$ where $\kappa:=\prod_{p<\infty}\kappa_p$ and for every prime $p$, \begin{equation}\label{eq:kappa}
	\kappa_p:=\frac{\#\CX_0(\BZ/p\BZ)}{p^n}.
\end{equation} 
\end{lemma}
In fact, we have $\kappa=\widehat{\omega}_{\infty}^{X_0}(\CX_0(\widehat{\BZ}))$, where $\widehat{\omega}_{\infty}^{X_0}$ is the finite part of the Tamagawa measure \eqref{eq:omegaXf} on $\CX_0(\widehat{\BZ})$ by Lemma \ref{le:modelmeasurecomp} and Theorem \ref{thm:nonarchTmeas}.

Let
\begin{equation}\label{eq:CDB}
	\CD(B):=\left\{\YY\in\BR_{\geqslant 1}^n:\max_{\sigma\in\triangle_{\max}}\YY^{D_0(\sigma)}\leqslant B\right\},
\end{equation}
and \begin{equation}\label{eq:CIB}
	\CI(B):=\int_{\CD(B)}\left(\max_{\sigma\in\triangle_{\max}}\YY^{D_0(\sigma)}\right)\operatorname{d}\omega^{X_0}_{\infty}(\YY).
\end{equation}
On recalling \eqref{eq:normX0} (see also  \cite[Proof of Lemma 11.29]{Salberger}), by taking $f$ to be the characteristic function of the region $\CD(B)$, we have
$$\CI(B)=\int_{\CD(B)}\left(\max_{\sigma\in\triangle_{\max}}\xx^{D_0(\sigma)}\right)\left\|\frac{\partial}{\partial x_{1}}\wedge\cdots\wedge\frac{\partial}{\partial x_{n}}(\xx)\right\|_{X_0,\infty}\operatorname{d}\xx=\int_{\CD(B)}\operatorname{d}\xx,$$ where $\operatorname{d}\xx=\operatorname{d}x_1\cdots\operatorname{d}x_n$ is the usual $n$-dimensional Lebesgue measure on $\BR^n$.

\begin{lemma}[\cite{Salberger} Lemma 11.29]\label{le:exchangeABIB}
We have	$$\#\CA(B)=\CI(B)+O(B(\log B)^{r-2}).$$
\end{lemma}
Combining the computation of $\CI(B)$ (cf. \cite[11.42]{Salberger}), we obtain
\begin{equation}\label{eq:CCB2}
	\#\CC_0(B)^+=\kappa\#\triangle_{\max}\alpha(X)B(\log B)^{r-1}+O(B(\log B)^{r-2}).
\end{equation} 
Our goal is to refine \eqref{eq:CCB2}. We establish Theorem \ref{thm:equidistrcong} which gives asymptotic formulas for the number of integral points in $\CC_0(B)^+$ lying in a given real neighbourhood with congruence conditions.

 We fix in the remaining of this section a cone $\sigma_0\in\triangle_{\max}$ with an admissible ordering. 
 For every $1\leqslant j\leqslant d$, let $\lambda_j\in \mathopen]0,1\mathclose]$, and consider  \begin{equation}\label{eq:Binfty}
	\CB_\infty:=\prod_{j=1}^d \mathopen]0,\lambda_j\mathclose]\subset C_{\sigma_0}(\BR)\subset X(\BR),
\end{equation} which we shall a \emph{zero-cube}.
For every integer $l\in\BZ_{>0}$ and for every residue $\bxi_l\in\CX_0(\BZ/l\BZ)$,
consider the set \begin{equation}\label{eq:CC0xilVBinfty}
	\CC_0([\bxi_l,\CB_\infty];B)^+:=\{\XX\in\CC_0(B)^+: \pi(\XX)\in\CB_\infty,\XX\equiv \bxi_l\bmod l\}.
\end{equation}

\begin{theorem}\label{thm:equidistrcong}
	We have, uniformly for any pair $(l,\bxi_l)$,
	$$\#\CC_0([\bxi_l,\CB_\infty];B)^+=\kappa_{(l)}\left(\prod_{j=1}^{d}\lambda_{j}\right)\alpha(X)B(\log B)^{r-1}+O_{\varepsilon,\CB_\infty}(l^\varepsilon B(\log B)^{r-2+\frac{1}{\alpha_0}+\varepsilon}),$$ where $\alpha(X)$ is given by \eqref{eq:alphaV} and \begin{equation}\label{eq:kappal}
		\kappa_{(l)}:=\frac{1}{l^n}\left(\prod_{p\nmid l}\kappa_p\right).
	\end{equation} 
\end{theorem}
By convention, all implied constants in the rest of \S\ref{se:purityproof} are allowed to depend on $\CB_\infty$.
\subsection{Ingredients of proof of Theorem \ref{thm:equidistrcong}}\label{se:comparelatticepts}
The main results of this subsection are Propositions \ref{prop:eqdistCAinftyB} and \ref{co:CAsigmaWBCDWsigmaB}, and we deduce Theorem \ref{thm:equidistrcong} based on them.

Recall \eqref{eq:Binfty} \eqref{eq:CC0xilVBinfty}. We define 
\begin{equation}\label{eq:CACBinfty}
	\CA(\CB_\infty;B):=\CA(B)\cap \pi^{-1}(\CB_\infty).
\end{equation} The following generalises  Lemma \ref{thm:Salbergermainlemma}.
\begin{proposition}\label{prop:eqdistCAinftyB}
We have, uniformly for any pair $(l,\bxi_l)$,
$$\#\CC_0([\bxi_l,\CB_\infty];B)^+=\kappa_{(l)}\#\CA(\CB_\infty;B)+O_\varepsilon\left(l^\varepsilon B(\log B)^{r-2+\frac{1}{\alpha_0}+\varepsilon}\right),$$
where $\kappa_{(l)}$ is defined by \eqref{eq:kappal}.
\end{proposition}
We next define
\begin{equation}\label{eq:CDCBinfty}
	\CD(\CB_\infty;B):=\CD(B)\cap \pi^{-1}(\CB_\infty),
\end{equation}
\begin{equation}\label{eq:CICBinfty}
	\CI(\CB_\infty;B):=\int_{\CD(\CB_\infty;B)}\left(\max_{\sigma\in\triangle_{\max}}\YY^{D_0(\sigma)}\right)\operatorname{d}\omega^{X_0}_{\infty}(\YY)=\int_{\CD(\CB_\infty;B)}\operatorname{d}\xx. 
\end{equation}
The following generalises Lemma \ref{le:exchangeABIB}, which compares the cardinality of $\CA(\CB_\infty;B)$ \eqref{eq:CACBinfty} with the integral $\CI(\CB_\infty;B)$ \eqref{eq:CICBinfty}. 
\begin{proposition}\label{co:CAsigmaWBCDWsigmaB}
	We have $$\#\CA(\CB_\infty;B)=\CI(\CB_\infty;B)+O(B(\log B)^{r-2}\log\log B).$$
\end{proposition}

\begin{proof}[Proof of Theorem \ref{thm:equidistrcong} assuming Propositions \ref{prop:eqdistCAinftyB} and \ref{co:CAsigmaWBCDWsigmaB}]\label{se:proofequistr}
We first introduce some more frequently used notation from  \cite[\S11]{Salberger}.
For every $\sigma\in\triangle_{\max}$ with an admissible ordering, we now define (recall \eqref{eq:Fj})
\begin{equation}\label{eq:Ej}
	E_\sigma(j):=D_{\rho_{r+j}}-F_\sigma(j),\quad 1\leqslant j\leqslant d.
\end{equation}
Note that $E_\sigma(j)$ has support in $\cup_{i=1}^r D_{\rho_i}$, so $\XX^{E_\sigma(j)}$ is a Laurent monomial in $X_1,\cdots,X_r$ for every $1\leqslant j\leqslant d$, and we have (recall \eqref{eq:D0sigma})
\begin{equation}\label{eq:EjDsigma}
	\begin{split}
		\sum_{j=1}^{d}E_\sigma(j)=\sum_{j=1}^{d}D_{\rho_{r+j}}+\sum_{\rho\in\triangle(1)}\langle -m_{D_0}(\sigma),n_\rho\rangle D_\rho=D_0(\sigma)-\sum_{i=1}^{r}D_{\rho_i}.
	\end{split}
\end{equation} 
We have that (cf. \cite[p. 249]{Salberger}) in $\Pic(X)$, for every $\sigma\in\triangle_{\max}$ with an admissible ordering, \begin{equation}\label{eq:linequiv}
	[D_0(\sigma)]=[D_0],\quad [D_{r+j}]=[E_{\sigma}(j)],\quad 1\leqslant j\leqslant d.
\end{equation}
If we denote by $x_k,1\leqslant k\leqslant n$ the coordinate regular functions on $X_0\subset\BA^n$, then on $\CTNS(\BR)=\Hom_\BZ(\Pic(X),\BR)\subset \BG_{\operatorname{m}}^{\triangle(1)}(\BR)$, thanks to \eqref{eq:linequiv}, we have
\begin{equation}\label{eq:coordfuneq}
	\xx^{D_0(\sigma)}=\xx^{D_0},\quad \text{and}\quad\xx^{E_{\sigma}(j)}=x_{r+j},\quad\text{for every }\sigma\in\triangle_{\max},1\leqslant j\leqslant d.
\end{equation}

The Lang--Weil estimate \cite{Lang-Weil} implies
$\kappa_p=1+O(p^{-\frac{1}{2}})$, whence \begin{equation}\label{eq:LWkappa}
	l^{-\varepsilon}\ll_\varepsilon \prod_{p\mid l}\kappa_p\ll_\varepsilon l^\varepsilon.
\end{equation} Therefore we have \begin{equation}\label{eq:kappabd}
\kappa_{(l)}\ll_\varepsilon l^{-n+\varepsilon}=O(1).
\end{equation}

Returing to the fixed zero-cube $\CB_\infty\subset C_{\sigma_0}(\BR)$, we shall prove 
$$\CI(\CB_\infty;B)=\left(\prod_{j=1}^{d}\lambda_{j}\right)\alpha(X)B(\log B)^{r-1}+O(B(\log B)^{r-2}\log\log B),$$ 
from which Theorem \ref{thm:equidistrcong} follows on combining Propositions \ref{prop:eqdistCAinftyB}, \ref{co:CAsigmaWBCDWsigmaB} and \eqref{eq:kappabd}.

We introduce (cf. \cite[11.34]{Salberger})
\begin{equation}\label{eq:OmegasigmaB}
	\Omega_{\sigma_0}(B):=\{\ZZ\in\BR_{\geqslant 1}^r:\ZZ^{D_0(\sigma_0)}\leqslant B,\min_{1\leqslant j\leqslant d}\ZZ^{E_{\sigma_0}(j)}\geqslant 1\},
\end{equation} and (cf. \cite[11.37]{Salberger})
\begin{equation}\label{eq:CFB}
	\CF(B):=\{\YY\in \BG_{\operatorname{m}}^{\triangle(1)}(\BR)\cap \CTNS(\BR):\YY^{D_0}\leqslant B,\min_{1\leqslant k\leqslant n}Y_k\geqslant 1\}\subset \CTNS(\BR).
\end{equation}
Then $\CF(B)$ is $\BR$-diffeomorphic to $\Omega_{\sigma_0}(B)$ via projection onto the first $r$ coordinates.
Let $\varpi_{\CTNS}$ be the global $\CTNS$-invariant differential form (cf. \cite[(1.13) and Lemma 11.38]{Salberger}) represented in the coordinate system of $U_{\sigma_0}$ with respect to an admissible ordering by $$\frac{\operatorname{d}x_1}{x_1}\wedge\cdots\wedge\frac{\operatorname{d}x_r}{x_r}.$$ Then we compute, in a similar manner as in \cite[11.33--11.41]{Salberger}, 
\begin{align*}
	\CI(\CB_\infty;B) &=\int_{\Omega_{\sigma_0}(B)}\left(\prod_{j=1}^{d}(\lambda_{j}\xx^{E_{\sigma_0}(j)}-1)\right)\operatorname{d}x_1\cdots\operatorname{d}x_r\\&=\int_{\Omega_{\sigma_0}(B)}\frac{\xx^{D_0(\sigma_0)}}{\prod_{i=1}^{r}x_i}\left(\prod_{j=1}^{d}(\lambda_{j}-\xx^{-E_{\sigma_0}(j)})\right)\operatorname{d}x_1\cdots\operatorname{d}x_r\\&=\int_{\CF(B)}\xx^{D_0}\left(\prod_{j=1}^{d}(\lambda_{j}-x^{-1}_{r+j})\right)\frac{\operatorname{d}x_1}{x_1}\cdots\frac{\operatorname{d}x_r}{x_r}\\&=\left(\prod_{j=1}^{d}\lambda_{j}\right)\int_{\CF(B)}\xx^{D_0}\operatorname{d}\varpi_{\CTNS}+O\left(\sum_{j=1}^{d}\int_{\CF(B)}\left(\frac{\xx^{D_0}}{x_{r+j}}\right)\operatorname{d}\varpi_{\CTNS}\right)\\ &=\left(\prod_{j=1}^{d}\lambda_{j}\right)\alpha(X)B(\log B)^{r-1}+O(B(\log B)^{r-2}).
\end{align*}
We are done.
\end{proof}

\subsection{Estimates of boundary sums}
Observe that for every $\YY\in \CD(\CB_\infty;B)$, we have, on recalling \eqref{eq:Fj} \eqref{eq:univmap} and \eqref{eq:Ej},
\begin{equation}\label{eq:conditionEsigmaW}	
	Y_{r+j}\leqslant \lambda_j \YY^{E_{\sigma_0}(j)} \quad \text{for every }1\leqslant j\leqslant d. \end{equation} 
To estimate the lattice points in $\CD(\CB_\infty;B)$, we need satisfactory control on the points that are, among other conditions, sufficiently close to the boundary of \eqref{eq:conditionEsigmaW}. The goal of this subsection is to deduce bounds on the cardinalities of two types of boundary sets $\nu^{(a)},\nu^{(b)}$, which are the key ingredients of the proof of Propositions \ref{prop:eqdistCAinftyB} and \ref{co:CAsigmaWBCDWsigmaB}.

We shall frequently invoke the following Fubini-type summation lemma.
\begin{lemma}[cf. \cite{Salberger} Sublemma 11.24 and \cite{Pieropan} Lemma 12]\label{le:sublemma}
	Let $m\in\BZ_{>0}$. Let $e\in\BR_{>0}$ and $\ee=(e_1,\cdots,e_m)\in\BZ_{\geqslant 0}^m$. For $B>3$ and $\lambda\in\BR_{>0}$ such that $1\leqslant \lambda \leqslant B$, we consider the following  sum and integral 
	$$S_{e,\ee}^{[m]}(B,\lambda):=\sum_{(g_1,\cdots,g_m)\in\BZ_{>0}^m}^{*}\prod_{i=1}^{m}g_i^{\frac{e_i}{e}-1},\quad R_{e,\ee}^{[m]}(B,\lambda):=\int_{(g_1,\cdots,g_m)\in\BR_{\geqslant 1}^m}^{*}\prod_{i=1}^{m}g_i^{\frac{e_i}{e}-1}\operatorname{d}\mathbf{g},$$ where $*$ means that the $m$-tuples $\mathbf{g}=(g_1,\cdots,g_m)$ satisfy
	\begin{equation}\label{eq:condition2}
		\prod_{i=1}^m g_i^{e_i}\leqslant \frac{B}{\lambda},\quad \text{and}\quad \max_{1\leqslant i\leqslant m}g_i\leqslant B.
	\end{equation} Then
	$$\max\left(S_{e,\ee}^{[m]}(B,\lambda),R_{e,\ee}^{[m]}(B,\lambda)\right)=O\left(\max\left((\log B)^m,\left(\frac{B}{\lambda}\right)^{\frac{1}{e}}(\log B)^{m-1}\right)\right),$$ where the implied constant depends on $m,e,\ee$, but does not depend on $\lambda$.
\end{lemma}
\begin{remark}\label{rmk:sublemma}
	The upper bound $(\log B)^m$ is achieved precisely when $\ee=\mathbf{0}$ and $\lambda\geqslant B/(\log B)^e$.
\end{remark}
 
For every $\sigma\in\triangle_{\max}$, we recall that the divisor $D_0(\sigma)$ \eqref{eq:D0sigma} is effective with support in $\triangle(1)\setminus\sigma(1)$.
\begin{lemma}\label{co:slicing} 
	For every $\sigma\in\triangle_{\max}$ with an admissible ordering,  for every $1\leqslant i_0\leqslant r$, and for $B>2$,  consider the sum $$\CN_\sigma(B)_{i_0,\eta[B]}:=\sum_{\substack{(X_1,\cdots,X_r)\in\BZ_{>0}^{r}\\ \XX^{D_0(\sigma)}\leqslant B,\max_{1\leqslant i\leqslant r} X_i\leqslant B\\ X_{i_0}=\eta[B](X_i,i\neq i_0)}}\frac{\XX^{D_0(\sigma)}}{\prod_{i=1}^{r} X_i},$$ where $\eta[B]:\BZ_{>0}^{r-1}\to\BR_{\geqslant 1}$ is a map depending on $B$.
	Then $$\CN_\sigma(B)_{i_0,\eta[B]}=O(B(\log B)^{r-2}).$$
	The implied constant is independent of the map $\eta[B]$.
\end{lemma}
\begin{proof}[Proof of Lemma \ref{co:slicing}]
	
	For every $1\leqslant i\leqslant r$, let
	\begin{equation}\label{eq:arhosigma}
		a_i(\sigma):=1+\langle -m_{D_0}(\sigma),n_{\rho_i}\rangle,
	\end{equation} so that (recall \eqref{eq:D0sigma}) $$D_0(\sigma)=\sum_{1\leqslant i\leqslant r}a_{i}(\sigma)D_{\rho_i},\quad \XX^{D_0(\sigma)}=\prod_{i=1}^{r}X_i^{a_i(\sigma)}.$$ 
	We write $\aa(\sigma):=(a_1(\sigma),\cdots,a_r(\sigma))\in\BZ_{\geqslant 0}^r$ and for every $i=1,\cdots,r$, let $\aa_{*i}(\sigma)$ denote the $(r-1)$-subtuple of $\aa(\sigma)$ without the term $a_{i}(\sigma)$.
	
	We may assume that $i_0=1$ and we write $\CN_\sigma(B)_{i_0,\eta[B]}=\CN_\sigma(B)_{\eta}$ for simplicity. We shall discuss two cases separately. 
	
	If $a_1(\sigma)>0$, then the condition $\prod_{i=1}^r X_i^{a_i(\sigma)}\leqslant B$ implies that $$\eta[B](X_2,\cdots,X_n)\leqslant \left(B/\prod_{i=2}^{r} X_i^{a_i(\sigma)}\right)^{\frac{1}{a_1(\sigma)}}.$$ 
	Then, we have, since $\eta[B]^{a_1(\sigma)}\geqslant\eta[B]\geqslant 1$, using Lemma \ref{le:sublemma} and its notation,
	\begin{align*}
		\CN_\sigma(B)_{\eta} &\leqslant \sum_{\substack{(X_2,\cdots,X_r)\in\BZ_{>0}^{r-1}\\ \prod_{i=2}^r X_i^{a_i(\sigma)}\leqslant B/(\eta[B](X_i,i\geqslant 2))^{a_1(\sigma)}\\\max_{2\leqslant i\leqslant r} X_i\leqslant B}}\left(B/\prod_{i=2}^r X_i^{a_i(\sigma)}\right)^{\frac{a_1(\sigma)-1}{a_1(\sigma)}} \prod_{i=2}^r X_i^{a_i(\sigma)-1}\\ &\leqslant B^{1-\frac{1}{a_1(\sigma)}}\sum_{\substack{(X_2,\cdots,X_r)\in\BZ_{>0}^{r-1}\\ \prod_{i=2}^r X_i^{a_i(\sigma)}\leqslant B\\ \max_{2\leqslant i\leqslant r} X_i\leqslant B}}\prod_{i=2}^{r}X_i^{\frac{a_i(\sigma)}{a_1(\sigma)}-1}\\ &=B^{1-\frac{1}{a_1(\sigma)}}S^{[r-1]}_{a_1(\sigma),\aa_{*1}(\sigma)}(B,1)=O(B(\log B)^{r-2}).
	\end{align*}
	
	However if $a_1(\sigma)=0$, then again since $\eta[B]\geqslant 1$,
	\begin{align*}
		\CN_\sigma(B)_{\eta}&=\sum_{\substack{(X_2,\cdots,X_r)\in\BZ_{>0}^{r-1}\\ \prod_{i=2}^r X_i^{a_i(\sigma)}\leqslant B/\eta[B](X_i,i\geqslant 2)\\\max_{2\leqslant i\leqslant r} X_i\leqslant B}}\frac{1}{\eta[B](X_2,\cdots,X_n)}\prod_{i=2}^r X_i^{a_i(\sigma)-1}\\ &\leqslant \sum_{\substack{(X_2,\cdots,X_r)\in\BZ_{>0}^{r-1}\\ \prod_{i=2}^r X_i^{a_i(\sigma)}\leqslant B\\\max_{2\leqslant i\leqslant r} X_i\leqslant B}}\prod_{i=2}^r X_i^{a_i(\sigma)-1}= S^{[r-1]}_{1,\aa_{*1}(\sigma)}(B,1)=O(B(\log B)^{r-2}).
	\end{align*} 
	
	None of the estimates above depend on the map $\eta[B]$, neither do any implied constants. This finishes the proof.
\end{proof}

We recall \eqref{eq:CAB} and define
\begin{equation}\label{eq:CAsigmaB}
	\CA_\sigma(B):=\CA(B)\cap \pi^{-1}(C_{\sigma}(\BR)).
\end{equation} 
On recalling the parametrisation \eqref{eq:univmap}
(cf. \cite[11.22, 11.23]{Salberger}), we have
\begin{equation}\label{eq:CAsigmaBcond}
	\CA_\sigma(B)=\{\XX\in\BZ_{>0}^n:\max_{1\leqslant k\leqslant n}X_k\leqslant B,\XX^{D_0(\sigma)}\leqslant B,\text{ for every }1\leqslant j\leqslant d,X_{r+j}\leqslant \XX^{E_\sigma(j)}\}.
\end{equation}
For every $1\leqslant i\leqslant r$ and $k\in\BZ_{>0}$, we introduce the set
\begin{equation}\label{eq:CAsigmaik}
	\CA_{\sigma}(B)_{i,k}:=\{\XX\in\CA_{\sigma}(B):X_i=k\}.
\end{equation} Then it follows directly from Lemma \ref{co:slicing}  that
\begin{corollary}\label{co:CAsigmaikbd}
	We have $$\#\CA_{\sigma}(B)_{i,k}=O(B(\log B)^{r-2}).$$
\end{corollary} 

For each $1\leqslant i_0\leqslant r$, consider the sum \begin{equation}\label{eq:Sprimebd}
	S^{(1)}_{\sigma,i_0}(B):=\sum_{\substack{\XX\in\BZ_{>0}^r:\XX^{D_0(\sigma)}\leqslant B\\ \max_{1\leqslant i\leqslant r}X_i\leqslant B}}\frac{\XX^{D_0(\sigma)}}{X_{i_0}\prod_{i=1}^{r}X_i}.
\end{equation}
\begin{lemma}\label{le:twospecsum2}
	We have	$$S^{(1)}_{\sigma,i_0}(B)=O(B(\log B)^{r-2}).$$
\end{lemma}
\begin{proof}
	We may assume that $i_0=1$. Recalling \eqref{eq:arhosigma}, we have $$S^{(1)}_{\sigma,1}(B)=\sum_{\substack{\XX\in\BZ_{>0}^r:\XX^{D_0(\sigma)}\leqslant B\\ \max_{1\leqslant i\leqslant r}X_i\leqslant B}}X_{1}^{a_{1}(\sigma)-2}\prod_{\substack{2\leqslant i\leqslant r}}X_i^{a_i(\sigma)-1}.$$
	If $a_{1}(\sigma)=0$, then by Lemma \ref{le:sublemma},
	\begin{align*}
		S^{(1)}_{\sigma,1}(B)=\left(\sum_{X_{1}=1}^{\lfloor B\rfloor}\frac{1}{X_{1}^2}\right) S_{1;\aa_{*1}(\sigma)}^{[r-1]}(B,1)=O(B(\log B)^{r-2}).
	\end{align*}
	If $a_{1}(\sigma)>0$, then again by by Lemma \ref{le:sublemma},
	\begin{align*}
		S^{(1)}_{\sigma,1}(B)&\leqslant\sum_{X_{1}=1}^{\lfloor B^{1/a_{1}(\sigma)}\rfloor}X_1^{a_{1}(\sigma)-2}S_{1;\aa_{*1}(\sigma)}^{[r-1]}\left(B,X_{1}^{a_{1}(\sigma)}\right)\\ &\ll \sum_{X_{1}=1}^{\lfloor B^{^{1/a_{1}(\sigma)}}\rfloor}X_1^{a_{1}(\sigma)-2}\max\left(\frac{B}{X_{1}^{a_{1}(\sigma)}}(\log B)^{r-2},(\log B)^{r-1}\right)\\ &\ll B(\log B)^{r-2}+B^{1-\frac{1}{a_{1}(\sigma)}}(\log B)^{r-1}+(\log B)^r=O(B(\log B)^{r-2}). \qedhere
	\end{align*}
\end{proof}

For every $1\leqslant j\leqslant d$, since $E_\sigma(j)$ \eqref{eq:Ej} has support in $\cup_{i=1}^r D_{\rho_i}$, we let $\bb^{(j,\sigma)}=(b_1^{(j,\sigma)},\cdots,b_r^{(j,\sigma)})\in\BZ^r$ be such that
\begin{equation}\label{eq:Esigmajbij}
	E_\sigma(j)=\sum_{i=1}^{r}b_i^{(j,\sigma)}D_{\rho_i},
\end{equation} so that
\begin{equation}\label{eq:bijsigma}
	\XX^{E_\sigma(j)}=\prod_{i=1}^rX_i^{b_i^{(j,\sigma)}}.
\end{equation}
Note that $$b_i^{(j,\sigma)}=-\langle n_{\rho_{r+j}}^\vee,n_{\rho_i}\rangle$$ in view of \eqref{eq:Fj} \eqref{eq:Ej}. We observe that $\bb^{(j,\sigma)}\neq \mathbf{0}$ for any $j$, thanks to the completeness and regularity of the fan $\triangle$.

Returning to the fixed $\sigma_0\in\triangle_{\max}$ with a fixed admissible ordering and a given parametrisation \eqref{eq:univmap}, and to the fixed zero cube $\CB_\infty$ \eqref{eq:Binfty}, we next derive estimates of lattice points lying around the boundary of the real neighbourhood $\CB_\infty$.

For $1\leqslant j_0\leqslant d$, let $\nu_{j_0}^{(a)}(\CB_\infty;B)$  be the subset of $\XX\in \CA(B)$ such that
\begin{equation}\label{eq:nusigmaA}
	X_i>1 \quad\text{for every } i\in\{1,\cdots,r\};
\end{equation}	
\begin{equation}\label{eq:nusigmaB2}
	X_{r+j}\leqslant\lambda_{j} 2^{\sum_{i=1}^{r}|b_{i}^{(j,\sigma_0)}|}\XX^{E_{\sigma_0}(j)},\text{ for every } j\neq j_0;
\end{equation}
	\begin{equation}\label{eq:nusigmaB1}
	X_{r+j_0}\leqslant\lambda_{j_0} \XX^{E_{\sigma_0}(j_0)};
\end{equation} 
\begin{equation}\label{eq:sigmaa}
	X_{r+j_0}+3>\lambda_{j_0} \XX^{E_{\sigma_0}(j_0)}.
\end{equation}
\begin{proposition}\label{prop:nusigmaWBa} We have
	$$\#\nu_{j_0}^{(a)}(\CB_\infty;B)=O(B(\log B)^{r-2}\log\log B).$$
\end{proposition}
\begin{proof}[Proof of Proposition \ref{prop:nusigmaWBa}]
Conditions \eqref{eq:nusigmaB1} and \eqref{eq:sigmaa} imply that, having fixed an $r$-tuple $(X_1,\cdots,X_r)$, $X_{r+j_0}$ can take at most $3$ different values. Together with \eqref{eq:nusigmaB2},  under the additional assumption $\XX^{E_{\sigma_0}(j_0)}\geqslant \log B$, on using \eqref{eq:EjDsigma} and Lemma \ref{le:sublemma}, we deduce that
\begin{align*}
	 \#\{\XX\in\nu_{j_0}^{(a)}(\CB_\infty;B):\XX^{E_{\sigma_0}(j_0)}\geqslant \log B\}\ll &\sum_{\substack{\XX\in\BZ_{>0}^r:\XX^{D_0(\sigma)}\leqslant B\\ \max_{1\leqslant i\leqslant r}X_i\leqslant B\\ \XX^{E_{\sigma_0}(j_0)}\geqslant \log B}}\prod_{\substack{1\leqslant j\leqslant d\\ j\neq j_0}} \XX^{E_{\sigma_0}(j)}\\ \ll &(\log B)^{-1}\sum_{\substack{\XX\in\BZ_{>0}^r:\XX^{D_0(\sigma)}\leqslant B\\ \max_{1\leqslant i\leqslant r}X_i\leqslant B}}\frac{\XX^{D_0(\sigma_0)}}{\prod_{i=1}^{r}X_i}\ll B(\log B)^{r-2}.
\end{align*} 

We next assume $\XX^{E_{\sigma_0}(j_0)}<\log B$, which implies \begin{equation}\label{eq:Xr+j0}
	X_{r+j_0}< \lambda_{j_0}\log B 
\end{equation} by \eqref{eq:nusigmaB1}. 
Let $\sigma_{j_0}\in\triangle_{\max}$ be the unique maximal cone adjacent to $\sigma_0$ (cf. \cite[Lemma 8.9]{Salberger}) such that $\sigma_{j_0}(1)\cap \sigma_0(1)=\sigma_0(1)\setminus\{\rho_{r+j_0}\}$.  Let $1\leqslant i_0\leqslant r$ be such that $\sigma_{j_0}(1)\setminus\sigma_0(1)=\{\rho_{i_0}\}$. Conditions \eqref{eq:nusigmaB1} \eqref{eq:sigmaa} together imply $$\lambda_{j_0}^{-1}\leqslant\frac{\XX^{E_{\sigma_0}(j_0)}}{X_{r+j_0}}< \lambda_{j_0}^{-1}\left(1+\frac{3}{X_{r+j_0}}\right)\leqslant 4\lambda_{j_0}^{-1}.$$  According to \cite[Remark 11.23]{Salberger}, we then have $$\frac{\XX^{D_0(\sigma_{j_0})}}{\XX^{D_0(\sigma_0)}}=\left(\XX^{F_{\sigma_0}(j_0)}\right)^{a_{r+j_0}(\sigma_{j_0})}=\left(\frac{X_{r+j_0}}{\XX^{E_{\sigma_0}(j_0)}}\right)^{a_{r+j_0}(\sigma_{j_0})}\asymp 1,$$ whether $$a_{r+j_0}(\sigma_{j_0}):=1+\langle -m_{D_0}(\sigma_{j_0}),n_{\rho_{r+j_0}}\rangle$$ is zero or not. 
On the other hand, note that $$\rho_{i_0}=-\sum_{1\leqslant j\leqslant d}b_{i_0}^{(j,\sigma_0)}\rho_{r+j}.$$ We necessarily have $b_{i_0}^{(j_0,\sigma_0)}=1$, thanks to the completeness and regularity of $\triangle$.  
Therefore, having fixed an $r$-tuple $(X_{r+j_0},X_i,i\neq i_0)$, conditions \eqref{eq:nusigmaB1} \eqref{eq:sigmaa} can be rewritten as
\begin{equation}\label{eq:aa}
	\lambda_{j_0}^{-1}X_{r+j_0}\prod_{i\neq i_0} X_i^{-b_{i}^{(j_0,\sigma_0)}}\leqslant X_{i_0}< \lambda_{j_0}^{-1}(X_{r+j_0}+3)\prod_{i\neq i_0} X_i^{-b_{i}^{(j_0,\sigma_0)}}.
\end{equation}
We may assume that the interval in \eqref{eq:aa} captures a non-zero integer (otherwise the sum is zero). Assuming this interval has length $>1$, we have
$$\sum_{X_{i_0}\in\BZ_{>0}:\eqref{eq:aa} \text{ holds}} \frac{1}{X_{i_0}}\ll \log\left(3+\frac{1}{X_{r+j_0}}\right)\ll 1,$$ and this bound remains true if this interval has length $<1$ since $X_{i_0}\geqslant 1$ is then uniquely determined.

We now put everything together. Note that $\XX^{D_0(\sigma_{j_0})}$ is a Laurent polynomial in $X_i,1\leqslant i\leqslant r,i\neq i_0$ and $X_{r+j_0}$. 
First fixing $X_{r+j_0},X_i,1\leqslant i\leqslant r$, summing over $X_{r+j},1\leqslant j\leqslant d,j\neq j_0$ and then using \eqref{eq:Xr+j0}, we obtain 
\begin{align*}
	\#\{\XX\in\nu_{j_0}^{(a)}(\CB_\infty;B):\XX^{E_{\sigma_0}(j_0)}<\log B\}\ll& \sum_{\substack{(X_{r+j_0},X_i,1\leqslant i\leqslant r)\in\BZ_{>0}^{r+1}\\ \XX^{D_0(\sigma)}\leqslant B,\max_{1\leqslant i\leqslant r}X_i\leqslant B\\ \eqref{eq:nusigmaB1}\eqref{eq:sigmaa}\text{ hold}, \XX^{E_{\sigma_0}(j_0)}<\log B}}\prod_{\substack{1\leqslant j\leqslant d\\ j\neq j_0}} \XX^{E_{\sigma_0}(j)}\\ \ll & \sum_{\substack{(X_{r+j_0},X_i,1\leqslant i\leqslant r)\in\BZ_{>0}^{r+1}\\ \XX^{D_0(\sigma)}\leqslant B,\max_{1\leqslant i\leqslant r}X_i\leqslant B\\ \eqref{eq:nusigmaB1}\eqref{eq:sigmaa}\text{ hold}, X_{r+j_0}\ll\log B}}\frac{\XX^{D_0(\sigma_0)}}{X_{r+j_0}\prod_{i=1}^{r}X_i}\\ \ll &\sum_{\substack{(X_{r+j_0},X_i,i\neq i_0)\in\BZ_{>0}^{r}\\ X^{D_0(\sigma_{j_0})}\ll B\\ \max_{i\neq i_0}X_i\leqslant B,X_{r+j_0}\ll\log B}}\frac{\XX^{D_0(\sigma_{j_0})}}{X_{r+j_0}\prod_{i\neq i_0}X_i}\sum_{\substack{X_{i_0}\in\BZ_{>0}\\\eqref{eq:aa} \text{ holds}}} \frac{1}{X_{i_0}}.
\end{align*}
If $a_{r+j_0}(\sigma_{j_0})=0$, applying Lemma \ref{le:sublemma}, we conclude that the sum above is
$$\ll S^{[r-1]}_{1,\aa_{*r+j_0}(\sigma_{j_0})}(B,1)\sum_{k\ll \log B}\frac{1}{k}\ll B(\log B)^{r-2}\log\log B.$$ If $a_{r+j_0}(\sigma_{j_0})>0$, again applying accordingly Lemma \ref{le:sublemma}, we estimate the sum above as follows:
$$\ll \sum_{k\ll \log B}k^{a_{r+j_0}(\sigma_{j_0})-1}S^{[r-1]}_{1,\aa_{*r+j_0}(\sigma_{j_0})}(B,k^{a_{r+j_0}(\sigma_{j_0})})\ll B(\log B)^{r-2}\sum_{k\ll \log B}\frac{1}{k}\ll B(\log B)^{r-2}\log\log B.$$
This finishes the proof. 
\end{proof}
 For any fixed $1\leqslant j_0\leqslant d$, we consider the set \begin{equation}\label{eq:CAsigmaj0}
	\CA_{j_0}^{\natural}(\CB_\infty;B):=\left\{\XX\in\CA(\CB_\infty;B):\XX^{F_\sigma(j_0)}=\lambda_{j_0}\right\},
\end{equation} so that $\bigcup_{j_0=1}^{d}	\CA_{j_0}^{\natural}(\CB_\infty;B)$ consists of points of $\CA_{\sigma_0}(B)$ lying on the boundary of $\CB_\infty$. Proposition \ref{prop:nusigmaWBa} admits the following consequence.
\begin{corollary}\label{co:partialCsigma}
	We have
	$$\#\CA_{j_0}^{\natural}(\CB_\infty;B)=O\left(B(\log B)^{r-2}\log\log B\right).$$
\end{corollary}
\begin{proof}
Clearly the condition $\XX^{F_\sigma(j_0)}=\lambda_{j_0}$ implies  \eqref{eq:nusigmaB1} \eqref{eq:sigmaa}. Hence
\begin{equation*}
	\#\CA_{j_0}^{\natural}(\CB_\infty;B)\leqslant \#\nu_{j_0}^{(a)}(\CB_\infty;B)+ \sum_{i=1}^{r}\#\CA_{\sigma_0}(B)_{i,1}=O\left(B(\log B)^{r-2}\log\log B\right).\qedhere
\end{equation*}
\end{proof}

For $\zz\in\BR_{\geqslant 2}^r$, we introduce the following notation \begin{equation}\label{eq:zzhat}
	\widehat{\zz}^{E_{\sigma_0}(j_0)}:=\prod_{\substack{1\leqslant i\leqslant r\\b_{i}^{(j_0,\sigma_0)}>0}}\left(z_i-1\right)^{b_{i}^{(j_0,\sigma_0)}}\prod_{\substack{1\leqslant i\leqslant r\\b_i^{(j_0,\sigma_0)}<0}}\left(z_i+1\right)^{b_{i}^{(j_0,\sigma_0)}}.
\end{equation}
Now let $\nu_{j_0}^{(b)}(\CB_\infty;B)$ be the subset of $\XX\in \CA(B)$ such that conditions \eqref{eq:nusigmaA} \eqref{eq:nusigmaB2} \eqref{eq:nusigmaB1} and
\begin{equation}\label{eq:sigmabb}
	X_{r+j_0}>\lambda_{j_0}\widehat{\XX}^{E_{\sigma_0}(j_0)}
\end{equation} hold but \eqref{eq:sigmaa} does not hold.
\begin{proposition}\label{prop:nusigmab}
	We have
	$$\#\nu_{j_0}^{(b)}(\CB_\infty;B)\ll B(\log B)^{r-2}.$$
\end{proposition}
\begin{proof}
	Having fixed an $r$-tuple $(X_1,\cdots,X_r)$, conditions \eqref{eq:nusigmaB1} and  \eqref{eq:sigmabb} imply that $X_{n+j_0}$ lies in the interval 
	$$\left]\lambda_{j_0}\widehat{\XX}^{E_{\sigma_0}(j_0)},\lambda_{j_0}\XX^{E_{\sigma_0}(j_0)}\right].$$ By the assumption opposed to \eqref{eq:sigmaa} we must have $$X_{r+j_0}+3\leqslant \lambda_{j_0} \XX^{E_{\sigma_0}(j_0)}.$$ So if this interval has length $<1$, then such an $X_{r+j_0}$ does not exist. We may therefore assume that it has length $>1$. Note that its length can be bounded from above by
	\begin{align*}
		\lambda_{j_0}\left(\XX^{E_{\sigma_0}(j_0)}-\widehat{\XX}^{E_{\sigma_0}(j_0)}\right)\ll \sum_{i_0=1}^{r}\sum_{i_0:b_{i_0}^{(j_0,\sigma_0)}\neq 0}X_{i_0}^{b_{i_0}^{(j_0,\sigma_0)}-1}\prod_{i\neq i_0}X_i^{b_{i}^{(j_0,\sigma_0)}}.
	\end{align*}
So for any fixed $(X_1,\cdots,X_r)$, the number of $d$-tuples $(X_{r+1},\cdots,X_{r+d})$ is at most
	\begin{equation*}
		\begin{split}
			\ll\prod_{j\neq j_0}\XX^{E_{\sigma_0}(j)}\sum_{i_0=1}^{r}\sum_{i_0:b_{i_0}^{(j_0,\sigma_0)}\neq 0}X_{i_0}^{b_{i_0}^{(j_0,\sigma_0)}-1}\prod_{i\neq i_0}X_i^{b_{i}^{(j_0,\sigma_0)}}= \prod_{j\neq j_0}\XX^{E_{\sigma_0}(j)}\sum_{i_0=1}^{r}\frac{\XX^{E_{\sigma_0}(j_0)}}{X_{i_0}}=\sum_{i_0=1}^{r}\frac{\XX^{D_0(\sigma_0)}}{X_{i_0}\prod_{i=1}^{r}X_i}.
	\end{split}\end{equation*}
	On recalling \eqref{eq:Sprimebd}, we therefore obtain
	\begin{align*}
	 \#\nu_{j_0}^{(b)}(\CB_\infty;B)\ll\sum_{i_0=1}^{r} S^{(1)}_{\sigma_0,i_0}(B)=O\left(B(\log B)^{r-2}\right),
	\end{align*}
	by Lemma  \ref{le:twospecsum2}. This finishes the proof.
\end{proof}

\subsection{Proof of Proposition \ref{co:CAsigmaWBCDWsigmaB}}
In comparing sums and integrals we frequently need to make use of the following kind of set. For every $1\leqslant k\leqslant n$, $\uu_k$ denotes the $n$-tuple of integers whose $k$-th coordinate is $1$ and and zero elsewhere. Now let us consider 
\begin{equation}\label{eq:deltaB}
	\begin{split}
		\delta(B):=\{\XX\in\CA(B):\text{ there exists } 1\leqslant k\leqslant n,\max_{\sigma\in\triangle_{\max}}(\XX+\uu_k)^{D_0(\sigma)}>B\}.	\end{split}
\end{equation}
\begin{lemma}[cf. \cite{Salberger}, Lemma 11.25 (b)]\label{le:deltaB}
	We have $$\#\delta(B)=O(B(\log B)^{r-2}).$$
\end{lemma}
With this lemma we can immediately prove Lemma \ref{le:exchangeABIB}.
Indeed, we need to cover the boundary introduced by the height condition by closed unit cubes with integral vertices (\emph{integral unit cubes} for short), at least one vertex of which lies in $\delta(B)$. We then have $$\left|\#\CA(B)-\CI(B)\right|\leqslant 2^n\#\delta(B)=O(B(\log B)^{r-2})$$ by Lemma \ref{le:deltaB}.

Let us write
\begin{equation}\label{eq:nuCB}
	\nu(\CB_\infty;B):=\bigcup_{j_0=1}^d \left(\nu_{j_0}^{(a)}(\CB_\infty;B)\cup \nu_{j_0}^{(b)}(\CB_\infty;B)\right).
\end{equation}
Then Propositions \ref{prop:nusigmaWBa} and \ref{prop:nusigmab} show that $$\nu(\CB_\infty;B)=O\left(B(\log B)^{r-2}\log\log B\right).$$

	The deviation between the cardinality of $\CA(\CB_\infty;B)$, i.e. the number of lattice points in $\CD(\CB_\infty;B)$, with the integral $\CI(\CB_\infty;B)$ is controlled by the integral unit cubes with at least one vertex lying in 
	$$\delta(B)\bigcup\nu(\CB_\infty;B)\bigcup \left(\bigcup_{i=1}^r\left(\CA_{\sigma_0}(B)_{i,1}\cup \CA_{\sigma_0}(B)_{i,2}\right)\right),$$ where we recall \eqref{eq:CAsigmaik}. 
	Admitting this, we have
	\begin{align*}
		\left|\#\CA(\CB_\infty;B)-\CI(\CB_\infty;B)\right| &\leqslant 2^n\left(\#\delta(B)+\#\nu(\CB_\infty;B)+\sum_{i=1}^{r}\left(\#\CA_{\sigma_0}(B)_{i,1}+\#\CA_{\sigma_0}(B)_{i,2}\right)\right)\\ &=O(B(\log B)^{r-2}\log\log B).
	\end{align*}

We fix $\YY_0=(Y_1,\cdots,Y_n)\in \CD(\CB_\infty;B)$ and $1\leqslant j_0\leqslant d$ such that $$Y_{r+j_0}=\lambda_{j_0}\YY_0^{E_{\sigma_0}(j_0)}\quad\text{and}\quad Y_i\geqslant 3\text{ for every }1\leqslant i\leqslant r,$$ and that it is not contained in any integral unit cube with at least one vertex in $\delta(B)$, that is, \begin{equation}\label{eq:conddeltaB}
	\{\XX\in\BZ_{>0}^n: |Y_k-X_k|\leqslant 1\text{ for all }1\leqslant k\leqslant n\}\subset \CA(B).
\end{equation} Let $\yy_{0}=(y_1,\cdots,y_n)\in\BZ_{>0}^n$ be defined by
\begin{equation}\label{eq:yi1}
	y_i=\begin{cases}
		\text{the unique integer in } [Y_i,Y_i+1[ & \text{ if } 1\leqslant i\leqslant r \text{ and } b_{i}^{(j_0,\sigma_0)}>0,\\ \text{the unique integer in } ]Y_i-1,Y_i] & \text{ if } 1\leqslant i\leqslant r \text{ and } b_{i}^{(j_0,\sigma_0)}\leqslant 0;
	\end{cases}
\end{equation}
\begin{equation}\label{eq:yr+j}
	y_{r+j}=\text{the unique integer in } ]Y_{r+j}-1,Y_{r+j}]\quad \text{if } 1\leqslant j\leqslant d,j\neq j_0;
\end{equation}
	\begin{equation}\label{eq:yr+j0}
		y_{r+j_0}=\begin{cases}
			\text{the unique integer in } [Y_{r+j_0},Y_{r+j_0}+1[ & \text{ if } \lambda_{j_0}\yy_0^{E_{\sigma_0}(j_0)}-\lambda_{j_0}\widehat{\yy_0}^{E_{\sigma_0}(j_0)}> 2\\ &~\text{ and } Y_{r+j_0}-1\leqslant \lambda_{j_0}\widehat{\yy_0}^{E_{\sigma_0}(j_0)};\\ \text{the unique integer in } ]Y_{r+j_0}-1,Y_{r+j_0}] & \text{ otherwise}.
		\end{cases}
	\end{equation}
	Clearly $\yy_0\in\CA(B)$ thanks to \eqref{eq:conddeltaB}.
	
	Let us now check that $\yy_0\in\nu(\CB_\infty;B)$. Clearly $y_i\geqslant 2,y_i-1<Y_i<y_i+1,1\leqslant i\leqslant r$ and hence \begin{equation}\label{eq:casestar}
	\begin{split}
		b_i^{(j,\sigma_0)}>0 &\Longrightarrow Y_i^{b_i^{(j,\sigma_0)}}<(y_i+1)^{b_i^{(j,\sigma_0)}}\leqslant 2^{b_i^{(j,\sigma_0)}} y_i^{b_i^{(j,\sigma_0)}};\\ b_i^{(j,\sigma_0)}<0 & \Longrightarrow Y_i^{b_i^{(j,\sigma_0)}}<(y_i-1)^{b_i^{(j,\sigma_0)}}\leqslant 2^{-b_i^{(j,\sigma_0)}}y_i^{b_i^{(j,\sigma_0)}}.
	\end{split}	
\end{equation} By \eqref{eq:conditionEsigmaW}	 and \eqref{eq:yr+j}, it follows that for every $j\neq j_0$ \begin{equation}\label{eq:jneqj0}
y_{r+j}\leqslant Y_{r+j}\leqslant\lambda_{j}\YY_0^{E_{\sigma_0}(j)}\leqslant\lambda_{j} 2^{\sum_{i=1}^{r}|b_i^{(j,\sigma_0)}|}\yy_0^{E_{\sigma_0}(j)},
\end{equation}  i.e., condition \eqref{eq:nusigmaB2} is verified. 

It remains to verify the conditions with respect to the index $j_0$. Recalling \eqref{eq:zzhat}, we observe that \begin{equation}\label{eq:casest2}
	\lambda_{j_0}\widehat{\yy_0}^{E_{\sigma_0}(j_0)}<Y_{r+j_0}=\lambda_{j_0}\YY^{E_{\sigma_0}(j_0)}\leqslant \lambda_{j_0}\yy_0^{E_{\sigma_0}(j_0)}.
\end{equation}  
	\begin{enumerate}
			\item Case $\lambda_{j_0}\yy_0^{E_{\sigma_0}(j_0)}-\lambda_{j_0}\widehat{\yy_0}^{E_{\sigma_0}(j_0)}\leqslant 2$.  It follows from  \eqref{eq:yr+j0} and \eqref{eq:casest2} that $$y_{r+j_0}\leqslant Y_{r+j_0}\leqslant\lambda_{j_0}\yy_0^{E_{\sigma_0}(j_0)},$$ and
		\begin{align*}
			y_{r+j_0}+3>Y_{r+j_0}+2&\geqslant Y_{r+j_0}+\lambda_{j_0}\yy_0^{E_{\sigma_0}(j_0)}-\lambda_{j_0}\widehat{\yy_0}^{E_{\sigma_0}(j_0)}\\ &=\lambda_{j_0}\yy_0^{E_{\sigma_0}(j_0)}+\lambda_{j_0}\YY^{E_{\sigma_0}(j_0)}-\lambda_{j_0}\widehat{\yy_0}^{E_{\sigma_0}(j_0)}>\lambda_{j_0}\yy_0^{E_{\sigma_0}(j_0)}.
		\end{align*}
		So in this case $\yy_0$ satisfies \eqref{eq:sigmaa} and hence $\yy_0\in \nu_{j_0}^{(a)}(\CB_\infty;B)$. 
		\item Case $\lambda_{j_0}\yy_0^{E_{\sigma_0}(j_0)}-\lambda_{j_0}\widehat{\yy_0}^{E_{\sigma_0}(j_0)}>2$.
\begin{itemize}
		\item If $Y_{r+j_0}-1\leqslant \lambda_{j_0}\widehat{\yy_0}^{E_{\sigma_0}(j_0)}$, then  by \eqref{eq:yr+j0} 
	 $$\lambda_{j_0}\widehat{\yy_0}^{E_{\sigma_0}(j_0)}<\lambda_{j_0}\YY^{E_{\sigma_0}(j_0)}=Y_{r+j_0}\leqslant y_{r+j_0}<Y_{r+j_0}+1\leqslant\lambda_{j_0}\widehat{\yy_0}^{E_{\sigma_0}(j_0)}+2<\lambda_{j_0}\yy_0^{E_{\sigma_0}(j_0)}.$$ 
		\item If $Y_{r+j_0}-1> \lambda_{j_0}\widehat{\yy_0}^{E_{\sigma_0}(j_0)}$, then by \eqref{eq:yr+j0} and \eqref{eq:casest2}, $$\lambda_{j_0}\widehat{\yy_0}^{E_{\sigma_0}(j_0)}<Y_{r+j_0}-1<y_{r+j_0}\leqslant Y_{r+j_0}= \lambda_{j_0}\YY^{E_{\sigma_0}(j_0)}\leqslant \lambda_{j_0}\yy_0^{E_{\sigma_0}(j_0)}.$$ 
	\end{itemize}  	
This shows that in both subcases $\yy_0$ satisfies \eqref{eq:sigmabb} and hence $\yy_0\in \nu_{j_0}^{(b)}(\CB_\infty;B)$ whenever \eqref{eq:sigmaa} is not verified.
	\end{enumerate} 
	
	We have shown that $\yy_0\in\nu(\CB_\infty;B)$.	This finishes the proof.
\qed

\subsection{Lattice points of bounded toric height in real neighbourhoods}
For every square-free integer $l\in\BZ_{>0}$ and for every residue $\bxi_l\in\CX_0(\BZ/l\BZ)$, we also define
\begin{equation}\label{eq:CAxilBinfty}
	\CA([\bxi_l,\CB_\infty];B):=\{\XX\in\CA(\CB_\infty;B):\XX\equiv \bxi_l\bmod l\},
\end{equation}
For $\dd_1=(d_{1,1},\cdots,d_{1,n}),\dd_2=(d_{2,1},\cdots,d_{2,n})\in\BZ_{>0}^n$, we write $\dd_1\mid \dd_2$ to mean that $d_{1,i}\mid d_{2,i}$ for every $1\leqslant i\leqslant n$.
Now for every $\dd\in\BZ_{>0}^n$, we define \begin{equation}\label{eq:CAdxilBinfty}
	\CA_{\dd}([\bxi_l,\CB_\infty];B):=\{\XX\in\CA([\bxi_l,\CB_\infty];B):\dd\mid \XX\},
\end{equation} and we write \begin{equation}\label{eq:Pid}
	\Pi(\dd):=\prod_{i=1}^{n}d_i.
\end{equation} 
By convention, $(\dd,l)=1$ means that $p\mid l\Rightarrow p\nmid d_i$ for every $1\leqslant i\leqslant n$.
As one of the main ingredients of proof of Proposition \ref{prop:eqdistCAinftyB}, the following generalises \cite[Lemma 11.26]{Salberger}, which relates the cardinality of $\CA(\CB_\infty;B)$ \eqref{eq:CACBinfty} with its subsets $\CA_{\dd}([\bxi_l,\CB_\infty];B)$.
\begin{proposition}\label{prop:CAxilBCAB}
	Uniformly for any pair $(l,\bxi_l)$ and $\dd\in\BZ_{>0}^n$ such that $(\dd,l)=1$, we have (recall \eqref{eq:Pid})
	$$l^n\Pi(\dd)\#\CA_{\dd}([\bxi_l,\CB_\infty];B)=\#\CA(\CB_\infty;B)+O(l^n\Pi(\dd) B(\log B)^{r-2}\log\log B).$$
\end{proposition}

\begin{proof}
	We perform a refinement of the proof of \cite[Lemma 11.26]{Salberger}, taking special care of the uniformity of dependency on $\dd$ and on $(l,\bxi_l)$. 
	As $X$ satisfies weak approximation, we may assume $\CA(\CB_\infty;B)\neq\varnothing$ and $\CA_{\dd}([\bxi_l,\CB_\infty];B)\neq\varnothing$.

	For every $\XX\in\BZ^n$ and $\dd_{0}=(d_{0,1},\cdots,d_{0,n})\in\BZ_{>0}^n$, we define the box
	$$\square_{\XX,\dd_0}:=\{\YY\in\BR_{>0}^n:\text{ for every }i,X_i\leqslant Y_i<X_i+d_{0,i}\}.$$
	And we write
	$$\CJ_1:=\bigcup_{\XX\in\CA_{\dd}([\bxi_l,\CB_\infty];B)} \square_{\XX,l\dd},\quad \text{with} \quad \CJ_2:=\bigcup_{\XX\in \CA(\CB_\infty;B)}\square_{\XX,\underline{\mathbf{1}}}.$$
	Since the boxes $(\square_{\XX,l\dd})_{\XX\in\CA_{\dd}([\bxi_l,\CB_\infty];B)}$ are mutually disjoint, the region $\CJ_1$ has area $l^n\Pi(\dd)\#\CA_{\dd}([\bxi_l,\CB_\infty];B)$ and $\CJ_2$ has area $\#\CA(\CB_\infty;B)$.
	
	Let us now compare $\CJ_1$ with $\CJ_2$.
	Write $$\blacksquare_{\XX,\dd_0}:=\{\YY\in\BR^n:\text{ for every }i,X_i-2d_{0,i}\leqslant Y_i\leqslant X_i+2d_{0,i}\}.$$ 
	Firstly, we have
	$$\CJ_1\subset \CJ_2\bigcup\left(\bigcup_{\XX\in\delta^{[l\dd]}(B)\cup\nu^{[l\dd]}(\CB_\infty;B)}\blacksquare_{\XX,\one}\right)$$ where, on recalling $\delta(B)$ \eqref{eq:deltaB} and $\nu(\CB_\infty;B)$ \eqref{eq:nuCB}, $$\delta^{[l\dd]}(B):=\{\XX\in\BZ^n:\blacksquare_{\XX,l\dd}\cap\delta(B)\neq\varnothing\},$$ and $$\nu^{[l\dd]}(\CB_\infty;B):=\{\XX\in\BZ^n:\blacksquare_{\XX,l\dd}\cap\nu(\CB_\infty;B)\neq\varnothing\}.$$ 
	Indeed, as we have seen in the proof of Proposition \ref{co:CAsigmaWBCDWsigmaB}, the number of boxes intersecting the boundary can be controlled by $\delta(B)\cup\nu(\CB_\infty;B)$. Furthermore, if any box $\square_{\XX,l\dd}$ does not hit the boundary of $\CD(\CB_\infty;B)$, then clearly $\square_{\XX,l\dd}\cap\BZ^n\subset\CA(\CB_\infty;B)$ and hence  $\square_{\XX,l\dd}=\bigcup_{\XX'\in \square_{\XX,l\dd}\cap\BZ^n} \square_{\XX',\one}\subset \CJ_2$. 
	On the other hand, we similarly have
	$$\CJ_2\subset \CJ_1\bigcup\left(\bigcup_{\XX\in\delta^{[l\dd]}(B)\bigcup \nu^{[l\dd]}(\CB_\infty;B)\bigcup\left( \cup_{i=1}^{r}\cup_{k=1}^{l\max_{1\leqslant i\leqslant n} d_i} \CA_{\sigma_0}(B)_{i,k}\right)}\blacksquare_{\XX,\one}\right),$$
	where we recall $\CA_{\sigma_0}(B)_{i,k}$ \eqref{eq:CAsigmaik}.
	
	To estimate $\#\delta^{[l\dd]}(B)$, we note that for every $\XX\in\delta^{[l\dd]}(B)$, there exists $\ee=(e_1,\cdots,e_n)\in\BZ^n$ with $|e_i|\leqslant 2ld_i$ such that $\XX+\ee\in\delta(B)$. So using Lemma \ref{le:deltaB}, 
	$$\#\delta^{[l\dd]}(B)\leqslant 4^n\Pi(l\dd)\#\delta(B)=O(l^n\Pi(\dd)B(\log B)^{r-2}).$$
	
	Similarly, to estimate $\#\nu^{[l\dd]}(\CB_\infty;B)$, for every $\XX\in\nu^{[l\dd]}(\CB_\infty;B)$, there exists $\ff=(f_1,\cdots,f_n)\in\BZ^n$ such that $|f_i|\leqslant 2ld_i$ and $\XX+\ff\in \nu(\CB_\infty;B)$. By Propositions \ref{prop:nusigmaWBa} and \ref{prop:nusigmab}, we similarly have
	$$\#\nu^{[l\dd]}(\CB_\infty;B)\leqslant 4^n \Pi(l\dd)\#\nu(\CB_\infty;B)=O(l^n\Pi(\dd)B(\log B)^{r-2}\log\log B).$$
	
	On using again Lemma \ref{co:CAsigmaikbd}, we finally obtain
	\begin{align*}
		&\left|l^n\Pi(\dd)\#\CA_{\dd}([\bxi_l,\CB_\infty];B)-\#\CA(\CB_\infty;B)\right|\\ \leqslant &4^n\left(\#\delta^{[l\dd]}(B)+\#\nu^{[l\dd]}(\CB_\infty;B)+\sum_{i=1}^{r}\sum_{k=1}^{l\max_{1\leqslant i\leqslant n} d_i}\#\CA_{\sigma_0}(B)_{i,k}\right)\\ =&O(l^n\Pi(\dd)B(\log B)^{r-2}\log\log B).
	\end{align*}
	This finishes the proof of Proposition \ref{prop:CAxilBCAB}. 
\end{proof}

\subsection{Proof of Proposition \ref{prop:eqdistCAinftyB}}
	We recall our convention that $(\dd,l)=1$ means $p\mid l\Rightarrow p\nmid d_i,1\leqslant i\leqslant n$, and we write $(\dd,l)>1$ if there exist $p$ and $d_i$ such that $p\mid (d_i,l)$.
	
	We have, on taking $\dd=\underline{\mathbf{1}}$ in Proposition  \ref{prop:CAxilBCAB} \begin{equation}\label{eq:s1}
		l^n\#\CA([\bxi_l,\CB_\infty];B)=\#\CA(\CB_\infty;B)+O(l^n B(\log B)^{r-2}\log\log B),
	\end{equation} and then we replace the term $\#\CA(\CB_\infty;B)$ in Proposition \ref{prop:CAxilBCAB} on inserting \eqref{eq:s1} and divide both side by $l^n$ to obtain, uniformly for all $(l,\bxi_l)$ and any $\dd\in\BZ_{>0}^n$ such that $(\dd,l)=1$,
	\begin{equation}\label{eq:s2}
		\Pi(\dd)\#\CA_{\dd}([\bxi_l,\CB_\infty];B)=\#\CA([\bxi_l,\CB_\infty];B)+O(\Pi(\dd) B(\log B)^{r-2}\log\log B).
	\end{equation}
	
	We denote by $\mu_X$ the generalised Möbius function defined in \cite[p. 234]{Salberger}. It is characterised by the property that, for every $\dd=(d_1,\cdots,d_n)\in\BZ_{>0}^n$,
	$$\sum_{\dd^\prime\mid \dd}\mu_X(\dd^\prime)=1\Longleftrightarrow \gcd_{\sigma\in\triangle_{\max}}(\dd^{D_0(\sigma)})=1.$$
	Note that $\mu_X(\dd)=0$ if there exist $p$  and $d_i$ such that $p^2\mid d_i$.
	For every prime $p$, according to \cite[(11.14)]{Salberger}, the $p$-adic density $\kappa_p$ \eqref{eq:kappa} has the expression
	\begin{equation}\label{eq:kappanuX}
		\kappa_p=\sum_{\substack{(e_1,\cdots,e_n)\in\BZ_{\geqslant 0}^n\\\dd=(p^{e_1},\cdots,p^{e_n})}}\frac{\mu_X(\dd)}{\Pi(\dd)}.
	\end{equation}
	Recall that $\alpha_0$ is the smallest integer such that there exist $\alpha_0$ rays in $\triangle(1)$ which does not generate a cone in $\triangle$. We then have
	\begin{lemma}[\cite{Salberger} Lemma 11.19]\label{le:summud}
		The formal series $$\sum_{\dd\in\BZ_{>0}^n}\frac{|\mu_X(\dd)|}{\Pi(\dd)^s}$$ is convergent for $s>\frac{1}{\alpha_0}$. Consequently, we have
		\begin{itemize}
			\item $\sum_{\dd\in\BZ_{>0}^n:\Pi(\dd)\leqslant b}|\mu_X(d)|=O_\varepsilon(b^{\frac{1}{\alpha_0}+\varepsilon})$,
			\item $\sum_{\dd\in\BZ_{>0}^n:\Pi(\dd)> b}\frac{|\mu_X(d)|}{\Pi(\dd)}=O_\varepsilon(b^{\frac{1}{\alpha_0}-1+\varepsilon})$.
		\end{itemize}
	\end{lemma}
	
	By the Möbius inversion, we have 
	\begin{align*}
		\#\CC_0([\bxi_l,\CB_\infty];B)^+&=\sum_{\dd\in\BZ_{>0}^n}\mu_X(\dd)\#\CA_{\dd}([\bxi_l,\CB_\infty];B)\\ &=\left(\sum_{\dd\in\BZ_{>0}^n:(\dd,l)=1}+\sum_{\dd\in\BZ_{>0}^n:(\dd,l)>1}\right)\mu_X(\dd)\#\CA_{\dd}([\bxi_l,\CB_\infty];B).
	\end{align*} 
	
	Recall $\kappa_{(l)}$ \eqref{eq:kappal} and the expression \eqref{eq:kappanuX}. We have $$\kappa_{(l)}=\frac{1}{l^n}\left(\sum_{\dd\in\BZ_{>0}^n:(\dd,l)=1}\frac{\mu_X(\dd)}{\Pi(\dd)}\right).$$ So by \eqref{eq:s1},
	\begin{align*}
		\kappa_{(l)}\#\CA(\CB_\infty;B)=&\left(\sum_{\dd\in\BZ_{>0}^n:(\dd,l)=1}\frac{\mu_X(\dd)}{\Pi(\dd)}\right)\frac{\#\CA(\CB_\infty;B)}{l^n}\\ =&\left(\sum_{\dd\in\BZ_{>0}^n:(\dd,l)=1}\frac{\mu_X(\dd)}{\Pi(\dd)}\right)\left(\#\CA([\bxi_l,\CB_\infty];B)+O(B(\log B)^{r-2}\log\log B)\right)\\ =&\left(\sum_{\dd\in\BZ_{>0}^n:(\dd,l)=1}\frac{\mu_X(\dd)}{\Pi(\dd)}\right)\#\CA([\bxi_l,\CB_\infty];B)+O_\varepsilon\left(l^\varepsilon B(\log B)^{r-2}\log\log B\right),
	\end{align*} on using \eqref{eq:LWkappa}.
	Moreover, using \eqref{eq:s2} and Lemma \ref{le:summud} we have 
	\begin{align*}
		&\sum_{\dd\in\BZ_{>0}^n:(\dd,l)=1}\mu_X(\dd)\#\CA_{\dd}([\bxi_l,\CB_\infty];B)-\left(\sum_{\dd\in\BZ_{>0}^n:(\dd,l)=1}\frac{\mu_X(\dd)}{\Pi(\dd)}\right)\#\CA([\bxi_l,\CB_\infty];B)\\ \leqslant&\sum_{\substack{\dd\in\BZ_{>0}^n\\\Pi(\dd)\leqslant \log B}}|\mu_X(\dd)|\left|\#\CA_{\dd}([\bxi_l,\CB_\infty];B)-\frac{\#\CA([\bxi_l,\CB_\infty];B)}{\Pi(\dd)}\right|+\#\CA(B)\sum_{\substack{\dd\in\BZ_{>0}^n\\\Pi(\dd)>\log B}}\frac{|\mu_X(\dd)|}{\Pi(\dd)}\\ \ll_\varepsilon& B(\log B)^{r-2+\frac{1}{\alpha_0}+\varepsilon},
	\end{align*} where for the sum over $\dd$ with $\Pi(\dd)>\log B$, we have used the upper bounds (cf. \cite[Lemma 11.26]{Salberger})
	$$\Pi(\dd)\#\CA_{\dd}([\bxi_l,\CB_\infty];B)\leqslant \Pi(\dd)\#\CA_{\dd}(B)\leqslant \#\CA(B),\quad \#\CA([\bxi_l,\CB_\infty];B)\leqslant \#\CA(B).$$
	We therefore obtain
	\begin{align*}
		&\sum_{\dd\in\BZ_{>0}^n:(\dd,l)=1}\mu_X(\dd)\#\CA_{\dd}([\bxi_l,\CB_\infty];B)-\kappa_{(l)}\#\CA(\CB_\infty;B)=O_\varepsilon(l^\varepsilon B(\log B)^{r-2+\frac{1}{\alpha_0}+\varepsilon}).
	\end{align*}
	
	To finish the proof, we are now going to show $$\sum_{\dd\in\BZ_{>0}^n:(\dd,l)>1}\mu_X(\dd)\#\CA_{\dd}([\bxi_l,\CB_\infty];B)=0.$$ 
	It suffices to consider $\dd=(d_1,\cdots,d_n)\in\BZ_{>0}^n$ such that each $d_i$ is square-free (otherwise $\mu_X(\dd)=0$), a condition which we denote by ``$\dd~\square$-free'' by a slight abuse of terminology. For every such $\dd$ with $(\dd,l)>1$, write $\dd=\dd_0\dd_{(l)}$ be the unique factorisation such that $\dd_0$ is maximal among $\ee\in\BZ_{>0}$ with $\ee\mid \dd$ and $(\ee,l)=1$.
	Note that $(\dd,l)>1$ implies that $\dd_{(l)}\neq \underline{\mathbf{1}}=(1,\cdots,1)$. 
	By the Chinese remainder theorem, we can lift $\bxi_l$ to $\Xi_l\in\CX_0(\BZ)$. It then satisfies $\gcd_{\sigma\in\triangle_{\max}}(\Xi_l^{D_0(\sigma)})=1$ (recall \eqref{eq:coprime}). Hence we have $$1=\sum_{\dd\in\BZ_{>0}^n:\dd\mid \Xi_l}\mu_X(\dd).$$ Since $\mu_X(\underline{\mathbf{1}})=1$, this implies $$\sum_{\substack{\dd\in\BZ_{>0}^n\\\dd\mid \Xi_l,\dd\neq\underline{\mathbf{1}}}}\mu_X(\dd)=0.$$  
	We further observe that, since every component of $\dd$ (and hence $\dd_{(l)}$) is square-free, $\dd_{(l)}\mid (l,\cdots,l)$ and $$\CA_{\dd}([\bxi_l,\CB_\infty];B)\neq\varnothing\Longrightarrow\dd_{(l)}\mid \Xi_l.$$ Moreover, under this assumption we have $$\CA_{\dd_0\dd_{(l)}}([\bxi_l,\CB_\infty];B)=\CA_{\dd_0}([\bxi_l,\CB_\infty];B).$$
	We can now compute:
	\begin{align*}
		&\sum_{\dd\in\BZ_{>0}^n:(\dd,l)>1}\mu_X(\dd)\#\CA_{\dd}([\bxi_l,\CB_\infty];B)\\=&\sum_{\substack{\dd\in\BZ_{>0}^n,\dd~\square-\text{free}\\ (\dd,l)>1,\dd_{(l)}\mid \Xi_l}}\mu_X(\dd)\#\CA_{\dd}([\bxi_l,\CB_\infty];B)\\=&\sum_{\substack{\dd_0\in\BZ_{>0}^n\\\dd_0~\square-\text{free},(\dd_0,l)=1}}\sum_{\substack{\dd_{(l)}\in\BZ_{>0}^n,\dd_{(l)}~\square-\text{free}\\\dd_{(l)}\neq\underline{\mathbf{1}},\dd_{(l)}\mid \Xi_l}}\mu_X(\dd_0\dd_{(l)})\#\CA_{\dd_0\dd_{(l)}}([\bxi_l,\CB_\infty];B)\\
		=&\sum_{\dd_0\in\BZ_{>0}^n:(\dd_0,l)=1}\mu_X(\dd_0)\#\CA_{\dd_0}([\bxi_l,\CB_\infty];B)\sum_{\substack{\dd_{(l)}\in\BZ_{>0}^n\\\dd_{(l)}\neq\underline{\mathbf{1}},\dd_{(l)}\mid \Xi_l}}\mu_X(\dd_{(l)})=0,
	\end{align*}
	on observing that the sums are in fact all finite.
\qed

\subsection{Proof of Theorem \ref{thm:effectiveEEUT}}

Since our toric variety is over $\BQ$ and split, we have $$X(\RA_\BQ)^{\operatorname{Br}}=X(\RA_\BQ),\quad \beta(X)=1,\quad \tau(\CTNS)=1,\quad W(\CTNS)=\#\FTNS(\BZ)=2^r.$$

Let us start by fixing throughout $\CF_\infty(\sigma_0)\in\mathfrak{F}$ with respect to a fixed $\sigma_0\in\triangle_{\max}$ and a congruence neighbourhood $\CE^{\CX_0}(l;\Xi_l)\subset \CX_0(\widehat{\BZ})$ of level $l\in\BZ_{>0}$ associated to $\Xi_l=(\Xi_p)_{p\mid l}\in\prod_{p\mid l}\CX_0(\BZ_p)$. By the action of $\CT_{O}(\BR)/\CT_{O}(\BR)^+$, we may assume that $\CF_\infty(\sigma_0)$ is a standard cube of the form \eqref{eq:standardcube}.

We now express the leading term in Theorem \ref{thm:equidistrcong} as the Tamagawa measure of the corresponding adelic neighbourhood. By Lemma \ref{le:modelmeasurecomp} and Theorem \ref{thm:nonarchTmeas}, we have, on recalling the constant $\kappa_{(l)}$ \eqref{eq:kappal},
$$\widehat{\omega}_{\infty}^{X_0}(\CE^{\CX_0}(l;\Xi_l))=\prod_{p\mid l}\omega^{X_0}_{p}(\CE_p^{\CX_0}(l;\Xi_p))\times \prod_{p\nmid l}\omega^{X_0}_{p}(\CX_0(\BZ_p))=\kappa_{(l)}.$$
We first assume that $\CF_\infty(\sigma_0)=\CB_\infty$ is a zero-cube of the form \eqref{eq:Binfty}. By Theorem \ref{thm:realTamagawameas}, we have $$\omega^X_{\operatorname{tor},\infty}(\CB_\infty)=\prod_{j=1}^{d}\lambda_{j}.$$ 

Recall the counting function \eqref{eq:UTcounting} regarding \eqref{eq:Thetai}. We let $\bxi_l\in \CX_0(\BZ/l\BZ)$ be the image of $\Xi_l$ under the reduction modulo $l$ map
$$\prod_{p\mid l}\CX_0(\BZ_p)\to\prod_{p\mid l}\CX_0(\BZ/p^{\operatorname{ord}_p(l)}\BZ)\simeq\CX_0(\BZ/l\BZ).$$ Theorem \ref{thm:equidistrcong} now implies that as $B\to\infty$ depending on $\CB_\infty$,
\begin{equation}\label{eq:stepCB}
	\begin{split}
	\CN_{\CX_0}(\CB_\infty,\CE^{\CX_0}(l;\Xi_l);B)=&\frac{\#\FTNS(\BZ)}{\alpha(X)B(\log B)^{r-1}}\#\CC_0([\bxi_l,\CB_\infty];B)^+\\
=&2^{r}\omega^X_{\operatorname{tor},\infty}(\CB_\infty)\widehat{\omega}_{\infty}^{X_0}(\CE^{\CX_0}(l;\Xi_l))+O_{\CB_\infty,\varepsilon}(l^\varepsilon (\log B)^{-1+\frac{1}{\alpha_0}+\varepsilon}).\end{split}
\end{equation}

In general by inclusion-exclusion, recalling the set \eqref{eq:CAsigmaj0}, there exists a family of zero-cubes $\left\{(\CB_{k_m}^{m})_{k_m\leqslant \binom{d}{m}},0\leqslant m\leqslant d\right\}$, the number of which is $\leqslant \sum_{m=0}^d\binom{d}{m}=2^d$, such that $$\omega^X_{\operatorname{tor},\infty}(\CF_\infty(\sigma_0))=\sum_{m=0}^d(-1)^m\sum_{k_m\leqslant \binom{d}{m}}\omega^X_{\operatorname{tor},\infty}(\CB_{k_m}^{m}).$$
Now on writing $\CF^m_{k_m}:=\CB_{k_m}^{m}\times\CE^{\CX_0}(l;\Xi_l)$ for short, it follows from \eqref{eq:stepCB} and Corollary \ref{co:partialCsigma}, allowing the implied constant to depend on the family above and hence on $\CF_\infty(\sigma_0)$ that
\begin{align*}
	&\CN_{\CX_0}(\CF_\infty(\sigma_0),\CE^{\CX_0}(l;\Xi_l);B)\\ =&\sum_{m=0}^d(-1)^m\sum_{k_m\leqslant \binom{d}{m}}\CN_{\CX_0}(\CF^m_{k_m};B)+O\left((B(\log B)^{r-1})^{-1}\sum_{j_0=1}^{d}\sum_{m=0}^d\sum_{k_m\leqslant \binom{d}{m}}\#\CA_{j_0}^{\natural}(\CB_{k_m}^{m};B)\right)\\ =& 2^{r}\omega^X_{\operatorname{tor},\infty}(\CF_\infty(\sigma_0))\widehat{\omega}_{\infty}^{X_0}(\CE^{\CX_0}(l;\Xi_l))+O_{\CF_\infty(\sigma_0),\varepsilon}(l^\varepsilon (\log B)^{-1+\frac{1}{\alpha_0}+\varepsilon}). 
\end{align*}
We recall that $\alpha_0\geqslant 2$, so $-1+\frac{1}{\alpha_0}\leqslant -\frac{1}{2}$. From this we obtain the desired error term. \qed

\begin{remark}\label{rmk:effectiveCFinfty}
	We would like to point out that the dependency on $\CB_\infty$ in the error term of Theorem \ref{thm:equidistrcong} and (hence in Theorem \ref{thm:mainequidist}) can be made explicit (by modifying appropriately the proof of Propositions \ref{prop:nusigmaWBa} and \ref{prop:nusigmab}). In principle we should be able to work out the dependency on a general real continuous set via exploiting a ``good covering'' of it. We choose not to work out details in the current article, as it would involve covering lemmas from geometric measure theory (see e.g. \cite[p. 110]{Folland}), which is far from our initial scope.
\end{remark}

\section{Toric van der Corput method}\label{se:toricvandercorput}
In this section, let $X$ be a smooth proper split toric variety over $\BQ$ defined by a complete regular fan $\triangle$ with globally generated anticanonical line bundle $K_X^{-1}$. Let $X_0$ be the principal universal torsor over $\BQ$ with the morphism $\pi:X_0\to X$ as in \eqref{eq:morptoric}.  
For every $\sigma\in\triangle_{\max}$ with an admissible ordering, we recall \eqref{eq:CAsigmaB}, and for every $A\geqslant 1$ we let 
\begin{equation}\label{eq:CAsigmaAB}
	\CA^{(A)}_\sigma(B):=\{\XX\in\CA_\sigma(B):\text{ for every }1\leqslant j\leqslant d, \XX^{E_\sigma(j)}\geqslant (\log B)^A\},
\end{equation}  
Our goal in this section is to establish the following, which shows that the difference between the sets $\CA_\sigma(B),\CA^{(A)}_\sigma(B)$ is negligible.
\begin{theorem}\label{thm:CABCAAB}
	For every $\sigma\in\triangle_{\max}$, we have $$\#\left(\CA_\sigma(B)\setminus\CA^{(A)}_\sigma(B)\right)=O_A(B(\log B)^{r-2}\log\log B).$$
\end{theorem}	

\subsection{Ingredients of the proof of Theorem \ref{thm:CABCAAB}}
Our goal of this subsection is to state Propositions \ref{prop:keyprop2} and \ref{prop:keyprop1}, and to deduce Theorem \ref{thm:CABCAAB} based on them.

Recall \eqref{eq:CDB} \eqref{eq:CIB}. For every $\sigma\in\triangle_{\max}$, let (cf. \cite[11.31]{Salberger})
\begin{equation}\label{eq:CDsigmaB}
		\CD_\sigma(B):=\CD(B)\cap \pi^{-1}(C_{\sigma}(\BR)), 
\end{equation}
\begin{equation}\label{eq:CIsigmaB}
	\CI_\sigma(B):=\int_{\CD_\sigma(B)}\YY^{D_0(\sigma)}\operatorname{d}\omega^{X_0}_{\infty}(\YY)=\int_{\CD_\sigma(B)}\operatorname{d}\xx.
\end{equation} 
Now we introduce certain subsets of $\CD_\sigma(B)$ with more complicated conditions similar to \eqref{eq:CAsigmaAB}. For every $A\geqslant 1$, we define 
\begin{equation}\label{eq:CDsigmaAB}
	\CD_\sigma^{(A)}(B):=\{\YY\in\CD_\sigma(B):\text{ for every } 1\leqslant j\leqslant d,\YY^{E_\sigma(j)}\geqslant (\log B)^A\},
\end{equation} \begin{equation}\label{eq:CIsigmaAB}
	\CI_\sigma^{(A)}(B):=\int_{\CD_\sigma^{(A)}(B)}\YY^{D_0(\sigma)}\operatorname{d}\omega^{X_0}_{\infty}(\YY)=\int_{\CD_\sigma^{(A)}(B)}\operatorname{d}\xx.
\end{equation}

\begin{proposition}\label{prop:keyprop2}
	For every $\sigma\in\triangle_{\max}$, uniformly for every $A\geqslant 1$, we have  $$\#\CA_\sigma^{(A)}(B)=\CI_\sigma^{(A)}(B)+O(B(\log B)^{r-2}\log\log B).$$ 
\end{proposition}

\begin{proposition}\label{prop:keyprop1}
	For every $\sigma\in\triangle_{\max}$, we have
	$$\CI_\sigma(B)-\CI_\sigma^{(A)}(B)=O_A(B(\log B)^{r-2}\log\log B).$$
\end{proposition}

\begin{proof}[Proof of Theorem \ref{thm:CABCAAB} assuming Propositions \ref{prop:keyprop2} and \ref{prop:keyprop1}]
	Thanks to
	\begin{align*}
		\#\left(\CA_\sigma(B)\setminus\CA_\sigma^{(A)}(B)\right)&=\#\CA_\sigma(B)-\#\CA_\sigma^{(A)}(B)\\ &\leqslant \left|\#\CA_\sigma(B)-\CI_\sigma(B)\right|+\left(\CI_\sigma(B)-\CI_\sigma^{(A)}(B)\right)+\left|\CI_\sigma^{(A)}(B)-\#\CA_\sigma^{(A)}(B)\right|,
	\end{align*}
	the desired estimate is a direct consequence of the following comparison results: Proposition \ref{co:CAsigmaWBCDWsigmaB} (between $\#\CA_\sigma(B)$ and $\CI_\sigma(B)$ on taking $\CB_\infty=\prod_{j=1}^{d}]0,1]$), Proposition \ref{prop:keyprop1} (between $\CI_\sigma(B)$ and $\CI_\sigma^{(A)}(B)$) and Proposition \ref{prop:keyprop2} (between $\CI_\sigma^{(A)}(B)$ and $\#\CA_\sigma^{(A)}(B)$).
\end{proof}

\begin{remark}
	The comparison between $\#\CA_\sigma(B)$ and $\#\CA^{(A)}_\sigma(B)$ is (indirectly) proceeded through the auxiliatory integral  $\CI_\sigma^{(A)}(B)$. The proof of Proposition \ref{prop:keyprop2} (comparing $\#\CA_\sigma(B)$ with $\CI_\sigma^{(A)}(B)$) given in \S\ref{se:bdsumvandercorput} shares similarities with that of Proposition \ref{co:CAsigmaWBCDWsigmaB}. While the proof of Proposition \ref{prop:keyprop1} (comparing $\CI_\sigma(B)$ with $\CI_\sigma^{(A)}(B)$) given in \S\ref{se:intvandercorput} allows to take advantage of the toric geometry. That is why we term this ``toric van der Corput method''.
\end{remark}

\subsection{Estimate of a boundary sum}\label{se:bdsumvandercorput}
We introduce the following set related to the boundary condition in \eqref{eq:CAsigmaAB}. Recall notation \eqref{eq:zzhat}. Let us fix $\sigma_0\in\triangle_{\max}$ and $1\leqslant j_0\leqslant d$.
For any $A\geqslant 1$, let $\CG_{\sigma_0,j_0}^{(A)}(B)$ be the set consisting of $\XX\in\CA(B)$ satisfying
$$X_{i}>1\quad \text{for every } i\in\{1,\cdots,r\};$$
\begin{equation}\label{eq:CGsigmaA}
	X_{r+j_0}\leqslant \XX^{E_{\sigma_0}(j_0)} ,\quad X_{r+j}\leqslant 2^{\sum_{i=1}^{r}|b_i^{(j,\sigma_0)}|}\XX^{E_{\sigma_0}(j)},\text{ for every }j\in\{1,\cdots,d\}\setminus\{j_0\},
\end{equation} 
 \begin{equation}\label{eq:condpos}
		\widehat{\XX}^{E_{\sigma_0}(j_0)}<(\log B)^A\leqslant \XX^{E_{\sigma_0}(j_0)}.
	\end{equation}

 The goal of this subsection is to establish 
\begin{proposition}\label{le:CGAB}
	We have $$\#\CG_{\sigma_0,j_0}^{(A)}(B)=O(B(\log B)^{r-2}),$$ where the implied constant is independent of $A$.
\end{proposition}
\begin{proof}
	For every $\XX\in\CG_{\sigma_0,j_0}^{(A)}(B)$, condition \eqref{eq:condpos} implies that there exist a collection of indices $J\subset \{1\leqslant i\leqslant r:b_{i}^{(j_0,\sigma_0)}\neq 0\}$ and $i_0\in J$, such that  \begin{equation}\label{eq:pos1}
		X_{i_0}^{b_{i_0}^{(j_0,\sigma_0)}}\left(\prod_{\substack{1\leqslant i\leqslant r\\i\not\in J}}X_i^{b_{i}^{(j_0,\sigma_0)}}\right)\left(\prod_{\substack{i\in J\setminus\{i_0\}\\b_{i}^{(j_0,\sigma_0)}>0}}\left(X_i-1\right)^{b_{i}^{(j_0,\sigma_0)}}\right)\left(\prod_{\substack{i\in J\setminus\{i_0\}\\b_i^{(j_0,\sigma_0)}<0}}\left(X_i+1\right)^{b_{i}^{(j_0,\sigma_0)}}\right)\geqslant (\log B)^A,
	\end{equation}
	\begin{equation}\label{eq:pos2}
		\left(\prod_{\substack{1\leqslant i\leqslant r\\ i\not\in J}}X_i^{b_{i}^{(j_0,\sigma_0)}}\right)\left(\prod_{\substack{i\in J\\b_{i}^{(j_0,\sigma_0)}>0}}\left(X_i-1\right)^{b_{i}^{(j_0,\sigma_0)}}\right)\left(\prod_{\substack{i\in J\\b_i^{(j_0,\sigma_0)}<0}}\left(X_i+1\right)^{b_{i}^{(j_0,\sigma_0)}}\right)< (\log B)^A.
	\end{equation}
	We define the real number 
	$$\Gamma:=(\log B)^A\left(\prod_{\substack{1\leqslant i\leqslant r\\ i\not\in J}}X_i^{-b_{i}^{(j_0,\sigma_0)}}\right)\left(\prod_{\substack{i\in J\setminus\{i_0\}\\b_{i}^{(j_0,\sigma_0)}>0}}\left(X_i-1\right)^{-b_{i}^{(j_0,\sigma_0)}}\right)\left(\prod_{\substack{i\in J\setminus\{i_0\}\\b_i^{(j_0,\sigma_0)}<0}}\left(X_i+1\right)^{-b_{i}^{(j_0,\sigma_0)}}\right),$$ which depends on $B,A,J,i_0$ and $X_i,i\neq i_0$.
	Conditions \eqref{eq:pos1} and \eqref{eq:pos2} imply that $X_{i_0}$ is the unique integer lying in the interval
	\begin{equation}\label{eq:condition4}
		\begin{cases}
			\left[\Gamma^{1/b_{i_0}^{(j_0,\sigma_0)}},\Gamma^{1/b_{i_0}^{(j_0,\sigma_0)}}+1\right[ &\text{ if } b_{i_0}^{(j_0,\sigma_0)}>0;\\ \left]\Gamma^{1/b_{i_0}^{(j_0,\sigma_0)}}-1,\Gamma^{1/b_{i_0}^{(j_0,\sigma_0)}}\right]&\text{ if } b_{i_0}^{(j_0,\sigma_0)}<0.
		\end{cases}
	\end{equation} 
	Define the map $\eta_{i_0,J}^{(A)}[B]:\BZ_{>0}^{r-1}\to\BZ_{>0}$ which associates $(X_i,i\neq i_0)$ to this integer.	
	We use condition \eqref{eq:CGsigmaA} and apply Lemma \ref{co:slicing} to obtain
	\begin{align*}
		\#\CG_{\sigma_0,j_0}^{(A)}(B)\leqslant &\sum_{\substack{1\leqslant i_0\leqslant d:b_{i_0}^{(j_0,\sigma_0)}\neq 0\\ i_0\in J\subset \{1\leqslant i\leqslant r:b_{i}^{(j_0,\sigma_0)}\neq 0\}}}\sum_{\substack{(X_1,\cdots,X_r)\in\BZ_{>0}^{r}\\ \XX^{D_0(\sigma_0)}\leqslant B,\max_{1\leqslant i\leqslant r} X_i\leqslant B\\ X_{i_0}=\eta_{i_0,J}^{(A)}[B](X_i,i\neq i_0)}}2^{\sum_{j\neq j_0}\sum_{i=1}^{r}|b_i^{(j,\sigma_0)}|}\prod_{j=1}^{d}\XX^{E_{\sigma_0}(j)}\\ \ll &\sum_{\substack{1\leqslant i_0\leqslant d:b_{i_0}^{(j_0,\sigma_0)}\neq 0\\ i_0\in J\subset \{1\leqslant i\leqslant r:b_{i}^{(j_0,\sigma_0)}\neq 0\}}}\CN_\sigma(B)_{i_0,\eta_{i_0,J}^{(A)}[B]}=O(B(\log B)^{r-2}),
	\end{align*} where the implied constant is independent of the map $\eta_{i_0,J}^{(A)}[B]$ and hence independent of $A$. 
\end{proof}

\subsection{Proof of Proposition \ref{prop:keyprop2}}\label{se:proofkeyProp2}
We fix throughout $\sigma_0\in\triangle_{\max}$. We recall \eqref{eq:nuCB} and we write $$\nu_{\sigma_0}(B):=\nu(]0,1]^d;B).$$
We want to prove that, any (real) point satisfying at least one of the boundary conditions in the definition of $\CA_{\sigma_0}^{(A)}(B)$ \eqref{eq:CAsigmaAB} is contained in an integral unit cube, at least one vertex of which lies in  $$\delta(B)\bigcup\nu_{\sigma_0}(B)\bigcup\left(\bigcup_{j_0=1}^{d}\CG^{(A)}_{\sigma_0,j_0}(B)\right)\bigcup\left(\bigcup_{i=1}^{r} \left(\CA_{\sigma_0}(B)_{i,1}\cup \CA_{\sigma_0}(B)_{i,2}\right)\right).$$
Admitting this, we henceforth conclude by Corollary \ref{co:CAsigmaikbd}, Lemma \ref{le:deltaB} and Propositions \ref{prop:nusigmaWBa}, \ref{prop:nusigmab} and \ref{le:CGAB} that the deviation resulting from switching the cardinality of $\CA_\sigma^{(A)}(B)$ to the integral $\CI_\sigma^{(A)}(B)$ is  
\begin{align*}
	\left|\#\CA_\sigma^{(A)}(B)-\CI_\sigma^{(A)}(B)\right|&\leqslant 2^n\left(\#\delta(B)+\sum_{j_0=1}^{d}\#\CG_{\sigma_0,j_0}^{(A)}(B)+\#\nu_{\sigma_0}(B)+\sum_{i=1}^{r}\#\left(\CA_{\sigma_0}(B)_{i,1}\cup \CA_{\sigma_0}(B)_{i,2}\right)\right)\\ &=O(B(\log B)^{r-2}\log\log B),
\end{align*} uniformly in $A$,
which is the desired estimate of Proposition \ref{prop:keyprop2}.

In view of the proof of Propositions \ref{co:CAsigmaWBCDWsigmaB}, we only need to deal with the extra boundary condition introduced in \eqref{eq:CAsigmaAB}. Now fix $1\leqslant j_0\leqslant d$ and $\YY_0\in\CD_{\sigma_0}^{(A)}(B)$ such that 
\begin{equation}\label{eq:condbd}
	\YY_0^{E_{\sigma_0}(j_0)}=(\log B)^A\text{ for a certain }1\leqslant j_0\leqslant d,\quad\text{and}\quad Y_i\geqslant 3\text{ for all }1\leqslant i\leqslant r,
\end{equation}  and that \eqref{eq:conddeltaB} holds.
We want to show that it is contained in an integral unit cube with at least one vertex in $\CG^{(A)}_{\sigma_0,j_0}(B)$.

Let the integral vector $\yy_{0}=(y_1,\cdots,y_n)\in\BZ_{\geqslant 2}^n$ be defined as in \eqref{eq:yi1} \eqref{eq:yr+j} and
\begin{equation}\label{eq:yr+j02}
	y_{r+j_0}=\text{the unique integer in } ]Y_{r+j_0}-1,Y_{r+j_0}].
\end{equation}
Clearly $y_i\geqslant 2$ for every $1\leqslant i\leqslant r$ by the second condition in \eqref{eq:condbd}, and condition \eqref{eq:conddeltaB} guarantees that $\yy_{0}\in\CA(B)$.  
For every fixed $j\neq j_0$, the argument \eqref{eq:casestar} shows that for all $j\neq j_0$,
$$y_{r+j}\leqslant Y_{r+j}\leqslant \YY_0^{E_{\sigma_0}(j)}\leqslant 2^{\sum_{i=1}^{r}|b_i^{(j,\sigma_0)}|}\yy_0^{E_{\sigma_0}(j)}.$$
Moreover, by the first assumption of \eqref{eq:condbd} and \eqref{eq:yr+j02}, $$\yy_0^{E_{\sigma_0}(j_0)}\geqslant \YY_0^{E_{\sigma_0}(j_0)}=(\log B)^A\geqslant Y_{r+j_0}\geqslant y_{r+j_0}.$$   Thus $\yy_0$ satisfies \eqref{eq:CGsigmaA}.
Finally, since $\bb^{(j_0,\sigma_0)}\neq \mathbf{0}$, we have similarly to \eqref{eq:casest2} that
$$\widehat{\yy_0}^{E_{\sigma_0}(j_0)}<\YY_0^{E_{\sigma_0}(j_0)}=(\log B)^A.$$ This shows that $\yy_0$ satisfies \eqref{eq:condpos}.  

So far we have proved $\yy_0\in\CG^{(A)}_{\sigma_0,j_0}(B)$, and $\YY_0$ is in one of the integral unit cubes generated by $\yy_0$. This finishes the proof.
\qed

\subsection{Proof of Proposition \ref{prop:keyprop1}}\label{se:intvandercorput}
We fix throughout $\sigma_0\in\triangle_{\max}$ with an admissible ordering.
Our idea is, according to \cite[p. 249--p. 250]{Salberger}, to use Fubini's theorem to reduce the computation of $\CI^{(A)}_{\sigma_0}(B)$  to integrals over the Neron-Severi torus $\CTNS(\BR)$ against the global $\CTNS$-invariant differential form $\varpi_{\CTNS}$.

 Recall that Salberger \cite[Lemma 11.38]{Salberger} proves (recall the region $\CF(B)$ \eqref{eq:CFB})	$$\CI_{\sigma_0}(B)=\int_{\CF(B)}\xx^{D_0}\operatorname{d}\varpi_{\CTNS}+O(B(\log B)^{r-2}).$$ Consider the regions
	$$\Omega_{\sigma_0}^{(A)}(B):=\{\ZZ\in\BR_{\geqslant 1}^r:\ZZ^{D_0(\sigma_0)}\leqslant B,\min_{1\leqslant j\leqslant d}\ZZ^{E_{\sigma_0}(j)}\geqslant (\log B)^A\}, $$
\begin{align*}
	\CF_{\sigma_0}^{(A)}(B)&:=\{\YY\in\CF(B):\min_{1\leqslant j\leqslant d}Y_{r+j}\geqslant (\log B)^A\}\\ &=\{\YY\in \BG_{\operatorname{m}}^{\triangle(1)}(\BR)\cap \CTNS(\BR):\YY^{D_0}\leqslant B,\min_{1\leqslant k\leqslant n}Y_k\geqslant 1,\min_{1\leqslant j\leqslant d}Y_{r+j}\geqslant (\log B)^A\},
\end{align*} which are diffeomorphic to each other.	We have, similarly to computation done in the proof of Theorem \ref{thm:equidistrcong}, uniformly for $A\geqslant 1$, \begin{align*}
\CI_{\sigma_0}^{(A)}(B)&=\int_{\Omega_{\sigma_0}^{(A)}(B)}\left(\prod_{j=1}^{d}(\xx^{E_{\sigma_0}(j)}-1)\right)\operatorname{d}x_1\cdots\operatorname{d}x_r\\ &=\int_{\CF_{\sigma_0}^{(A)}(B)}\xx^{D_0}\operatorname{d}\varpi_{\CTNS}+O\left(\sum_{j=1}^{d}\int_{\CF_{\sigma_0}^{(A)}(B)}\left(\frac{\xx^{D_0}}{x_{r+j}}\right)\operatorname{d}\varpi_{\CTNS}\right)\\&=\int_{\CF_{\sigma_0}^{(A)}(B)}\xx^{D_0}\operatorname{d}\varpi_{\CTNS}+O(B(\log B)^{r-2}).
\end{align*}

	We next decompose $$\CF(B)\setminus\CF_{\sigma_0}^{(A)}(B)=\bigcup_{j_0=1}^d \CH_{\sigma_0,j_0}^{(A)}(B),$$ where for each $j_0\in\{1,\cdots,d\}$ according to the previously fixed admissible ordering of $\sigma_0$, $$\CH_{\sigma_0,j_0}^{(A)}(B):=\{\YY\in \BG_{\operatorname{m}}^{\triangle(1)}(\BR)\cap \CTNS(\BR):\YY^{D_0}\leqslant B,\min_{1\leqslant k\leqslant n}Y_k\geqslant 1, Y_{r+j_0}<(\log B)^A\}.$$
	We now fix such a $j_0$. We choose $\sigma_1\in\triangle_{\max}$ with $-n_{\rho_{r+j_0}}\in\sigma_1$ thanks to the completeness of the fan $\triangle$, in particular $\rho_{r+j_0}\not\in \sigma_1(1)$. The maximal cone adjacent to $\sigma_0$ without the ray $\rho_{r+j_0}$ will do.
	For the computation of $\int_{\CH_{\sigma_0,j_0}^{(A)}(B)}\xx^{D_0}\operatorname{d}\varpi_{\CTNS}$, we now switch to an admissible ordering of $\sigma_1$ such that $\rho_{r+j_0}\mapsto\rho_r$ and $\sigma_1(1)=\{\rho_{r+1},\cdots,\rho_{r+d}\}$. Recall \eqref{eq:arhosigma}. Since $\langle-m_{D_0}(\sigma),n_{\rho_r}\rangle=\langle m_{D_0}(\sigma),-n_{\rho_r}\rangle>0$ thanks to our choice of $\sigma_1$, it follows that $$a_r(\sigma_1)>1.$$ Using \eqref{eq:coordfuneq}, we then have
	\begin{align*}
		\int_{\CH_{\sigma_0,j_0}^{(A)}(B)}\xx^{D_0}\operatorname{d}\varpi_{\CTNS}&=\int_{\CH_{\sigma_0,j_0}^{(A)}(B)}\xx^{D_0(\sigma_1)}\operatorname{d}\varpi_{\CTNS}\\&=\int_{\substack{\forall 1\leqslant i\leqslant r,1\leqslant x_i\leqslant B\\ 1\leqslant x_r<(\log B)^A, \xx^{D_0(\sigma_1)}\leqslant B}} \xx^{D_0(\sigma_1)}\frac{\operatorname{d}x_1}{x_1}\cdots\frac{\operatorname{d}x_r}{x_r}\\
		&=\int_{1\leqslant x_r<(\log B)^A}x_r^{a_r(\sigma_1)-1}\left(\int_{\star}\left(\prod_{i=1}^{r-1}x_i^{a_i(\sigma_1)-1}\right)\operatorname{d}x_1\cdots\operatorname{d}x_{r-1}\right)\operatorname{d}x_r,
	\end{align*}
	where condition $\star$ means \begin{equation}\label{eq:condint}
		\text{for all } 1\leqslant i\leqslant r-1,\quad 1\leqslant x_i\leqslant B,\quad \text{and}\quad\prod_{i=1}^{r-1}x_i^{a_i(\sigma_1)}\leqslant B/x_r^{a_r(\sigma_1)}.
	\end{equation}
	In terms of the notation in Lemma \ref{le:sublemma}, the integral above with condition \eqref{eq:condint} is
	\begin{equation}\label{eq:intR}
		\int_{1\leqslant x_r<(\log B)^A}x_r^{a_r(\sigma_1)-1} R_{1,\aa_{*r}(\sigma_1)}^{[r-1]}(B,x_r^{a_r(\sigma_1)})\operatorname{d}x_r.
	\end{equation}
	If $\aa_{*r}(\sigma_1)\neq\mathbf{0}$, then since $x_r\leqslant (\log B)^A$, by Lemma \ref{le:sublemma}, the integral \eqref{eq:intR} is 
	\begin{align*}
		\int_{1\leqslant x_r<(\log B)^A}x_r^{a_r(\sigma_1)-1}\frac{B}{x_r^{a_r(\sigma_1)}}\left(\log B\right)^{r-2}\operatorname{d}x_r&=B(\log B)^{r-2}\int_{1\leqslant x_r<(\log B)^A}\frac{\operatorname{d}x_r}{x_r}\\&=O_A(B(\log B)^{r-2}\log\log B).
	\end{align*}
	If $\aa_{*r}(\sigma_1)=\mathbf{0}$, then by Remark \ref{rmk:sublemma}, the integral \eqref{eq:intR} is $$\ll (\log B)^{r-1}\int_{1\leqslant x_r<(\log B)^A}x_r^{a_r(\sigma_1)-1}\operatorname{d}x_r=O((\log B)^{r-1+Aa_r(\sigma_1)}).$$ We finally conclude that 
	\begin{align*}
		\CI_{\sigma_0}(B)-\CI_{\sigma_0}^{(A)}(B)&=\int_{\CF(B)\setminus\CF_{\sigma_0}^{(A)}(B)}\xx^{D_0}\operatorname{d}\varpi_{\CTNS}+O(B(\log B)^{r-2})\\ &\leqslant\sum_{j_0=1}^{d}	\int_{\CH_{{\sigma_0},j_0}^{(A)}(B)}\xx^{D_0}\operatorname{d}\varpi_{\CTNS}+O(B(\log B)^{r-2})=O_A(B(\log B)^{r-2}\log\log B).
	\end{align*}
	This finishes the proof of Proposition \ref{prop:keyprop1}.
\qed

\section{Counting integral points on subvarieties of universal torsors}\label{se:countingsubvar}
We keep working under the setting of \S\ref{se:purityproof} and \S\ref{se:toricvandercorput} and using the notation therein. The main result of this section is Proposition \ref{prop:subvar}, which counts integral points of bounded height lying on a proper closed subvariety of $X_0$. We then deduce Corollary \ref{cor:subvar} from it.
\begin{proposition}\label{prop:subvar}
	Let $\phi$ be a non-zero integral polynomial in $n$-variables. Let us consider the set 
\begin{equation}\label{eq:CAphi}
		\CA_{\phi}(B):=\{\XX\in\CA(B):\phi(\XX)=0\}.
\end{equation}
	Then we have
	$$\#\CA_{\phi}(B)\ll (\deg \phi) B(\log B)^{r-2}\log\log B,$$ where the implied constant does not depend on $\phi$.
\end{proposition}
\begin{corollary}\label{cor:subvar}
	Let $Y\subset X$ be a proper closed subvariety.
	Then $$\#\{P\in (Y\cap \CT_{O})(\BQ):H_{\operatorname{tor}}(P)\leqslant B\}=O_Y(B(\log B)^{r-2}\log\log B).$$
\end{corollary}
\begin{remark}
	Proposition \ref{prop:subvar} and Corollary \ref{cor:subvar} provide an extra power saving of $\log B$ compared to the result $O(B(\log B)^{r-2+\iota})$ for a certain $0<\iota<1$, obtained via Corollary \ref{thm:effectiveEE} plus sieving techniques such as Theorem \ref{thm:Selbergsieve}.
\end{remark}
\begin{proof}[Proof of Corollary \ref{cor:subvar} assuming Proposition \ref{prop:subvar}]
	We consider the preimage $Y_0:=\pi^{-1}(Y)$ as a proper closed subvariety of $X_0$.
	We choose an integral non-zero polynomial $\phi_0$ in $n$-variables which vanishes on $Y_0$. 
	Ignoring the coprimality condition, by Proposition \ref{prop:subvar}, we clearly have 
	\begin{equation*}
		\#\{P\in (Y\cap \CT_{O})(\BQ):H_{\operatorname{tor}}(P)\leqslant B\}\ll \#\CA_{\phi_0}(B)\ll_{\phi_0}B(\log B)^{r-2}\log\log B. \qedhere
	\end{equation*} 
\end{proof}

Before entering into the proof of Proposition \ref{prop:subvar}, we consider, for every $\sigma\in\triangle_{\max}$ and $1\leqslant k_0\leqslant n$, the set \begin{equation}\label{eq:CABik}
	\CA_\sigma(B)_{k_0,\widetilde{\eta}[B]}:=\{\XX\in\CA_\sigma(B):X_{k_0}=\widetilde{\eta}[B](X_k,1\leqslant k\leqslant n,k\neq k_0)\},
\end{equation} where $\widetilde{\eta}[B]:\BZ_{>0}^{n-1}\to\BR_{\geqslant 1}$ is a map depending on $B$.
\begin{lemma}\label{le:CABik}
		If $1\leqslant k_0\leqslant r$, assume moreover that $\widetilde{\eta}[B]$ depends only on $X_i,1\leqslant i\leqslant r,i\neq k_0$. We then have $$\#\CA_\sigma(B)_{k_0,\widetilde{\eta}[B]}=O(B(\log B)^{r-2}\log\log B).$$
\end{lemma}
\begin{proof}
	If $1\leqslant k_0\leqslant r$, then by our assumption and Lemma \ref{co:slicing}, $$\#\CA_{\sigma}(B)_{k_0,\widetilde{\eta}[B]}\leqslant\sum_{\substack{\XX\in\BZ_{>0}^r\\ \XX^{D_0(\sigma)}\leqslant B,\max_{1\leqslant k\leqslant r} X_k\leqslant B\\ X_{k_0}=\widetilde{\eta}[B](X_k,1\leqslant k< k_0)}}\prod_{j=1}^d\XX^{E_\sigma(j)}= \CN_{\sigma}(B)_{k_0,\widetilde{\eta}[B]}=O(B(\log B)^{r-2}).$$ However if $r+1\leqslant k_0=r+j_0\leqslant n$ for a fixed $1\leqslant j_0\leqslant d$, then upon first fixing $X_i,1\leqslant i\leqslant r$ and $X_{r+j},j\neq j_0$, we have (recall the set \eqref{eq:CAsigmaAB})
	\begin{align*}
		\#\CA_{\sigma}(B)_{k_0,\widetilde{\eta}[B]}&\leqslant \#(\CA_\sigma(B)\setminus \CA_\sigma^{(1)}(B))+ \sum_{\substack{\XX\in\BZ_{>0}^r:\XX^{D_0(\sigma)}\leqslant B\\ \max_{1\leqslant i\leqslant r}X_i\leqslant B,\XX^{E_{\sigma}(j_0)}\geqslant \log B}}\prod_{\substack{1\leqslant j\leqslant d\\ j\neq j_0}}\XX^{E_{\sigma}(j)}\\ &\ll B(\log B)^{r-2}\log\log B+(\log B)^{-1}\sum_{\substack{\XX\in\BZ_{>0}^r:\XX^{D_0(\sigma)}\leqslant B\\ \max_{1\leqslant i\leqslant r}X_i\leqslant B}}\frac{\XX^{D_0(\sigma_0)}}{\prod_{i=1}^{r}X_i}\\&=O(B(\log B)^{r-2}\log\log B),
	\end{align*} on using Lemma \ref{le:sublemma}.
\end{proof}

\begin{proof}[Proof of Proposition \ref{prop:subvar}]
	We recall the set  \eqref{eq:CAsigmaB}. For every $\sigma\in\triangle_{\max}$, we define $$\CA_{\sigma,\phi}(B):=\{\XX\in\CA_\sigma(B): \phi(\XX)=0\}.$$ It suffices to prove,  uniformly for $\phi$, 
\begin{equation}\label{eq:CAsigmaphi}
	\#\CA_{\sigma,\phi}(B)\ll (\deg \phi)B(\log B)^{r-2}\log\log B.
\end{equation}

	We fix an admissible ordering with respect to $\sigma$ and consider the projection $p_n:\BA^n\to\BA^{n-1}$ onto the first $n-1$ coordinates. Then either $\phi$ is constant on the variable $x_{n}$, i.e., $\phi$ factors through $p_n$, or the restriction $\pi_n:=p_n|_{\phi=0}$ is a generically finite dominant morphism to $\BA^{n-1}$. For the latter case, consider the locus $Y_{n-1}\subset\BA^{n-1}$ where $\pi_n$ fails to be finite. If we regard $\phi$ as a (non-constant) polynomial in $x_n$, then for any $\XX\in\BZ^n$ with $\phi(\XX)=0$, $\pi_n(\XX)\in Y_{n-1}$ if and only if all the coefficients (as polynomials in $\BZ[x_1,\cdots,x_{n-1}]$) of $\phi$ vanish on $\pi_n(\XX)$. We can take for example the (non-zero) polynomial $\phi_{n-1}\in\BZ[x_1,\cdots,x_{n-1}]$ among these coefficients of the highest degree, and we define
	$$\CA_{\sigma,\phi}^{(1)}(B):=\{\XX\in\CA_{\sigma,\phi}(B):p_n(\XX)\not\in Y_{n-1}\},$$
	$$\CA_{\sigma,\phi_{n-1}}(B):=\{\XX\in\CA_\sigma(B):\phi_{n-1}(p_n(\XX))=0\},$$
	so that $\deg\phi_{n-1}\leqslant \deg\phi$ and $$\CA_{\sigma,\phi}(B)\subset\CA_{\sigma,\phi}^{(1)}(B)\bigcup \CA_{\sigma,\phi_{n-1}}(B).$$ Note however for the first case where $\phi$ is constant in $x_n$ we just have  $\CA_{\sigma,\phi}(B)=\CA_{\sigma,\phi_{n-1}}(B)$ and $\CA_{\sigma,\phi}^{(1)}(B)=\varnothing$.
	Continuing the analysis for $\BA^{n-1}$ and the non-zero polynomial $\phi_{n-1}$ (if non-constant), we can further break $\CA_{\sigma,\phi_{n-1}}(B)$ into subsets. Finally we obtain 
	\begin{equation}\label{eq:decompCABf}
		\CA_{\sigma,\phi}(B)\subset \bigcup_{k=1}^{n}\CA_{\sigma,\phi}^{(k)}(B),
	\end{equation}
	where each set $\CA_{\sigma,\phi}^{(k)}(B),1\leqslant k\leqslant n$, if non-empty, is constructed  at the $k$-th step with respect to a non-constant polynomial $\phi_{n-k+1}\in\BZ[X_1,\cdots,X_{n-k+1}]$ with $\deg \phi_{n-k+1}\leqslant \deg \phi$ (by convention $\phi_n:=\phi$), and has the property that,  for $1\leqslant k\leqslant n-1$, there exists a finite collection of maps $(\eta_{k}(m):\BZ_{>0}^{n-k}\to\BR_{\geqslant 1})_{m\in \CM(k)}$, such that for every $\XX\in\CA_{\sigma,\phi}^{(k)}(B)$,  $X_{n-k+1}=\eta_{k}(m)(X_1,\cdots,X_{n-k})$ for a certain $m\in \CM(k)$, and $(\eta_{n}(m))_{m\in\CM(n)}$ consists of constant functions corresponding to the solutions of $\phi_1=0$ (in one variable). Moreover for every $1\leqslant k\leqslant n$,  $\#\CM(k)\leqslant\deg \phi_{n-k+1}\leqslant\deg \phi$.
	
	Recalling \eqref{eq:CABik}, by construction, the maps $(\eta_{k}(m):\BZ_{>0}^{n-k}\to\BR_{\geqslant 1})_{m\in \CM(k),1\leqslant k\leqslant n}$ meet the assumption of Lemma \ref{le:CABik}, and so we have $$\#\CA_{\sigma,\phi}^{(k)}(B)\leqslant \sum_{m\in\CM(k)} \#\CA_\sigma(B)_{n-k+1,\eta_k(m)}=O\left((\deg \phi)B(\log B)^{r-2}\log\log B\right).$$
	To conclude, we go back to \eqref{eq:decompCABf} and obtain
	$$\#\CA_{\sigma,\phi}(B)\leqslant \sum_{k=1}^{n}\#\CA_{\sigma,\phi}^{(k)}(B)=O\left((\deg \phi)B(\log B)^{r-2}\log\log B\right).$$
	The proof of \eqref{eq:CAsigmaphi} is thus completed.
\end{proof}

\section{The geometric sieve for toric varieties}\label{se:geomsieve}
The goal of this section is to prove Theorem \ref{thm:maingeomsieve} via establishing the following.
\begin{theorem}\label{thm:geomsieve1}
	Let $f,g$ be integral polynomials in $n$-variables which are coprime (i.e., $(f=g=0)\subset\BA^{n}$ has codimension two). Then uniformly for $N\geqslant 2$,
		$$\#\{\XX\in\CA(B): \text{there exists } p\geqslant N,p\mid \gcd(f(\XX),g(\XX))\}\ll \frac{B(\log B)^{r-1}}{N\log N}+B(\log B)^{r-2}\log\log B,$$ where we recall \eqref{eq:CAB} and the implied constant depends only on $f,g$.
\end{theorem}
\begin{proof}[Proof of Theorem \ref{thm:maingeomsieve} assuming Theorem \ref{thm:geomsieve1}]
 For every Zariski closed $Z\subset X$ with codimension at least two, we write $Z_0:=\pi^{-1}(Z)\subset X_0$. Then $Z_0$ is also closed and has codimension at least two in $X_0$ thanks to the faithful flatness of $\pi$. Let $\CZ_0$ be the Zariski closure of $Z_0$ in $\CX_0$. Then $\CZ_0=\pi^{-1}(\CZ)$. 
Upon considering all irreducible components of $Z_0$, we may assume that $\CZ_0\subset (f_0=g_0=0)\cap \CX_0$ for certain non-zero integral coprime polynomials $f_0,g_0$ in $n$-variables.  Therefore by Theorem \ref{thm:geomsieve1} applied to $f_0,g_0$, we obtain
\begin{align*}
	&\#\{P\in \CT_{O}(\BQ):H_{\operatorname{tor}}(P)\leqslant B, \text{there exists } p\geqslant N,\widetilde{P}\bmod p \in\CZ(\BF_p)\}\\ \ll&\#\{\XX\in\CA(B): \text{there exists } p\geqslant N,p\mid \gcd(f_0(\XX),g_0(\XX))\}\\ \ll& \frac{B(\log B)^{r-1}}{N\log N}+B(\log B)^{r-2}\log\log B.
\end{align*}
The choice of $f_0,g_0$ depends only on $\CZ_0$ and hence on $Z$. This finishes the proof.
\end{proof}
\begin{remark}\label{re:effectiveGS}
	Summarising all known examples (cf. \cite{Ekedahl,Browning-HB} and  \cite[Theorem 3.1 (ii) and its Remarks]{Cao-Huang2}), we expect  that the following effective \textbf{(GS)} condition holds: (Recall the counting function \eqref{eq:CR} in Theorem \ref{thm:keyOmega}.) 
	
	\emph{Let $V/k$ be an almost-Fano variety with $\CV$ an integral model over $\CO_S$, and let $Z\subset V$  be a Zariski closed subset of codimension at least two with $\CZ$ the Zariski closure of $Z$ in $\CV$ and $\CW:=\CV\setminus\CZ$. 
		Let $\Omega_{\nu}=\operatorname{Mod}_{\nu,1}^{-1}\left(\CW(\BF_\nu)\right)\subset \CV(\CO_{\nu})=V(k_\nu)$. 
		Then there exist continuous functions $h_1,h_2:\BR_{>0}\to\BR_{>0}$ such that $h_i(x)=o(1),x\to\infty$ for $i=1,2$, such that, uniformly for every $N\geqslant 2$,
			$$\CR((\Omega_\nu);N,B)\ll h_1(N)+h_2(B).$$
		where the implied constant depends only on $\CV$ and $Z$.}
	
	Theorem \ref{thm:geomsieve1} shows that for any closed $Z\subset X$ of codimension at least two, we can choose 
	$$h_1(x)=\frac{1}{x\log x},\quad h_2(x)=\frac{\log\log x}{\log x}.$$
\end{remark}
The remaining of this section is devoted to proving Theorem \ref{thm:geomsieve1}.
\subsection{Two specific versions of the geometric sieve} 	
We shall fix the coprime polynomials $f,g$ throughout the rest of this section, and we let $\CY_0$ be the subscheme $(f=g=0)\subset\BA^n_\BZ$. 
For every $\sigma\in\triangle_{\max}$with an admissible ordering, consider the projection $\operatorname{pr}_\sigma:\BA^n\to\BA^r$ onto the first $r$ coordinates, and for every $\zz\in \BA^r$, we denote the fibre of the restriction to $\CY_0$ of $\operatorname{pr}_\sigma$ over $\zz$ by $\CY_{0,\zz}(\sigma):=\operatorname{pr}_\sigma^{-1}(\zz)\cap \CY_0\subset \BA_{k(\zz)}^d$, where $k(\zz)$ denotes the residue field of $\zz$. Let $\CZ_0(\sigma)$  be the set of $\zz\in \BA^r$ such that $\CY_{0,\zz}(\sigma)$ has codimension at most one in $\BA_{k(\zz)}^{d}$. Since $\CY_0$ has codimension at least two, then the argument in \cite[p. 360 Middle paragraph]{Poonen} shows that the (constructible) set $\CZ_0(\sigma)$ has codimension at least one in $\BA^r$.  Hence we can choose a non-zero polynomial $h_\sigma\in\BZ[X_1,\cdots,X_r]$ which vanishes on $\CZ_0(\sigma)$, so that for any $\zz\in\BA^r\setminus (h_\sigma=0)$, the fibre $\CY_{0,\zz}(\sigma)$ has codimension at least two in $ \BA_{k(\zz)}^d$.

The following version of the geometric sieve is dedicated to controlling the contribution from arbitrarily large primes. Recall \eqref{eq:CAsigmaAB}.
\begin{proposition}\label{prop:geomsieve1}
For every $\sigma\in\triangle_{\max}$ and for $A,N\geqslant 2$, consider the following set
\begin{equation}\label{eq:CBMB}
 \CK_{\sigma}^{(A)}(B;N):=\{\XX\in\CA_\sigma^{(A)}(B):h_\sigma(\XX)\neq 0 \text{ and there exists } p\geqslant N,p\mid \gcd(f(\XX),g(\XX))\}.
\end{equation}
Then we have, uniformly for such $A,N$,
$$\#\CK_{\sigma}^{(A)}(B;N)\ll\frac{B(\log B)^r}{N\log N}+B(\log B)^{r-A},$$ where the implied constant depends only on $f,g$.
\end{proposition}

To get rid of the extra $\log B$ factor in the first term of Proposition \ref{prop:geomsieve1}, we also need the next sieve result which treats the contribution from every single moderately large prime. We keep using the notation above.
\begin{proposition}\label{prop:geomsieve2}
For every $\sigma\in\triangle_{\max}$, for every prime $p$ and for $A\geqslant 2$, define analogously the set
	$$\CK_{\sigma}^{(A)}(B;\{p\}):=\{\XX\in\CA_\sigma^{(A)}(B):h_\sigma(\XX)\neq 0\text{ and }p\mid \gcd(f(\XX),g(\XX))\}.$$ 
	Then we have, uniformly for $p\leqslant (\log B)^{A}$, 
	$$\#\CK_{\sigma}^{(A)}(B;\{p\})\ll \frac{B(\log B)^{r-1}}{p^2}+p^{d-1}B(\log B)^{r-1-A}.$$
\end{proposition}
\begin{remark}
	Theorem \ref{thm:effectiveEEUT} plus standard sieving techniques provide similar estimates, the second term being of the form $p^{\gamma_1}B(\log B)^{r-1-\iota_1}$, with some large $\gamma_1>0$ and small $\iota_1>0$. However, the arbitrarily large $\log B$-saving turns out to be crucial in our proof of Theorem \ref{thm:geomsieve1} below.
\end{remark}
\subsection{Proof of Theorem \ref{thm:geomsieve1} assuming Propositions \ref{prop:geomsieve1} and \ref{prop:geomsieve2}}
 We begin by assuming $N<(\log B)^2$. We consider the quantity 
  $$\Pi_1(B):=\sum_{\sigma\in\triangle_{\max}}\#\{\XX\in\CA_\sigma(B):h_\sigma(\XX)=0\}.$$ By Proposition \ref{prop:subvar}, we have
  $$\Pi_1(B) \ll B(\log B)^{r-2}\log \log B.$$
 		Applying Propositions \ref{prop:geomsieve1}  and \ref{prop:geomsieve2},  we obtain
 	\begin{align*}
 		&\#\{\XX\in\CA(B): \text{there exists } p\geqslant N,p\mid \gcd(f(\XX),g(\XX))\}\\ &\leqslant \Pi_1(B)+\sum_{\sigma\in\triangle_{\max}}\#\left(\CA_\sigma(B)\setminus\CA_{\sigma}^{(A)}(B)\right)+\sum_{\sigma\in\triangle_{\max}}\left(\#\CK_{\sigma}^{(A)}(B;(\log B)^2)+\sum_{N\leqslant p\leqslant (\log B)^2}\#\CK_{\sigma}^{(A)}(B;\{p\})\right)\\ &\ll_A B(\log B)^{r-2}\log \log B+\frac{B(\log B)^{r-1}}{N\log N}+B(\log B)^{r-2-A+2d},
 	\end{align*}
 	and then by fixing $A=2d$, we achieve the desired estimate.

Now assume $N\geqslant (\log B)^2$. We deduce directly from Proposition \ref{prop:geomsieve1} with $A=2$  that 
\begin{align*}
	&\#\{\XX\in\CA(B): \text{there exists } p\geqslant N,p\mid \gcd(f(\XX),g(\XX))\}\\ &\leqslant \sum_{\sigma\in\triangle_{\max}}\#(\CA_\sigma(B)\setminus\CA_\sigma^{(2)}(B))+\sum_{\sigma\in\triangle_{\max}}\#\CK_{\sigma}^{(2)}(B;N)\\ &\ll \frac{B(\log B)^r}{N\log N}+B(\log B)^{r-2}\log \log B=O(B(\log B)^{r-2}\log \log B),
\end{align*} since the second error term dominates the first one. 
This completes the proof of Theorem \ref{thm:geomsieve1}. \qed

\subsection{Proof of Proposition \ref{prop:geomsieve1}}

We need a quantitative and uniform version of Ekedahl's sieve \cite{Ekedahl} (cf. also \cite[Lemma 3.1 \& Theorem 3.3]{Bhargava}) for affine spaces,  allowing boxes with unequal side-lengths. 
We define the height of a polynomial to be the maximum of the absolute values of its coefficients.
\begin{lemma}[Browning--Heath-Brown \cite{Browning-HB}, Lemma 2.1]\label{le:EkedahlBrHB}
	Let $m\geqslant 2,d\geqslant 1$ be fixed. Then uniformly for any $B_1,\cdots,B_m\geqslant 1,Q,N\geqslant 2$, and for any polynomials $f_1,f_2\in\BZ[X_1,\cdots,X_m]$ which are coprime, of degree $\leqslant d$ and have height $\leqslant Q$, we have
	\begin{multline*}
		\#\{\XX\in\BZ^m:\text{ for every }1\leqslant  i\leqslant m,|X_i|\leqslant B_i \text{ and there exists } p>N,p\mid \gcd(f_1(\XX),f_2(\XX))\}\\
		\ll \FB\log (\FB Q)\left((N\log N)^{-1}+(B_{\min})^{-1}\right),
	\end{multline*}
 where $$\FB:=\prod_{i=1}^{m}B_i,\quad B_{\min}:=\min_{1\leqslant i\leqslant m}(B_i).$$ The implied constant only depends on $m,d$.
\end{lemma}

Now let us start the proof of Proposition \ref{prop:geomsieve1}.
For each fixed $\YY\in\BZ_{>0}^r$, Define the set
$$\CK_{\sigma,\YY}(B;N):=\{\ZZ\in\BZ_{>0}^d:Z_j\leqslant \YY^{E_\sigma(j)},1\leqslant j\leqslant d \text{ and there exists } p\geqslant N,p\mid\gcd( f(\YY,\ZZ),g(\YY,\ZZ))\}.$$
Then by the construction of $h_\sigma$,  for each $\YY\in\BZ_{>0}^r$ satisfying 
\begin{equation}\label{eq:condition1}
	\max_{1\leqslant i\leqslant r}Y_i\leqslant B,\quad  \YY^{D_0(\sigma)}\leqslant B, \quad h_\sigma(\YY)\neq 0,\quad  \min_{1\leqslant j\leqslant d}\YY^{E_\sigma(j)}\geqslant (\log B)^A,
\end{equation}  the polynomials  $f(\YY,X_{r+1},\cdots,X_{r+d}),g(\YY,X_{r+1},\cdots,X_{r+d})\in\BZ[X_{r+1},\cdots,X_{r+d}]$ are coprime and have height $\ll B^{\max(\deg f,\deg g)}$.
Applying Lemma \ref{le:EkedahlBrHB} we get
\begin{align*}
	\#\CK_{\sigma,\YY}(B;N)&\ll \left(\prod_{j=1}^{d}\YY^{E_\sigma(j)}\right)\log \left(B^{\max(\deg f,\deg g)}\prod_{j=1}^{d}\YY^{E_\sigma(j)}\right)\left((N \log N)^{-1}+\left(\min_{1\leqslant i\leqslant d}\YY^{E_\sigma(j)}\right)^{-1}\right) \\ &\ll \log B\left((N \log N)^{-1}+(\log B)^{-A}\right)\left(\prod_{i=1}^{r}Y_i^{a_i(\sigma)-1}\right).
\end{align*} 
The implied constant depends only on the degree of $f,g$.
Therefore we obtain
\begin{align*}
	\#\CK_{\sigma}^{(A)}(B;N)&=\sum_{\substack{\YY\in\BZ_{>0}^r\\\eqref{eq:condition1}\text{ holds}}}\#\CK_{\sigma,\YY}(B;N)\\ 	&\ll \left(\frac{\log B}{N \log N}+(\log B)^{1-A}\right)\sum_{\substack{\YY\in\BZ_{>0}^r\\\max_{1\leqslant i\leqslant r}Y_i\leqslant B, \YY^{D_0(\sigma)}\leqslant B }}\prod_{i=1}^{r}Y_i^{a_i(\sigma)-1}\\ &\ll\frac{B(\log B)^r}{N\log N}+B(\log B)^{r-A},
\end{align*} by Lemma \ref{le:sublemma}.
We get the desired estimate.
\qed

\subsection{Proof of Proposition \ref{prop:geomsieve2}}\label{se:proofpropgeomsieve2}
Let $p$ be a fixed prime. We define likewise in the proof of Proposition \ref{prop:geomsieve1}, for each $\YY\in\BZ_{>0}^r$ $$\CK_{\sigma,\YY}(B;\{p\}):=
\{\ZZ\in\BZ_{>0}^d:Z_j\leqslant \YY^{E_\sigma(j)},1\leqslant j\leqslant d,p\mid\gcd( f(\YY,\ZZ),g(\YY,\ZZ))\}.$$
We now estimate $\CK_{\sigma,\YY}(B;\{p\})$ for every fixed $\YY\in\BZ_{>0}^d$ satisfying  \eqref{eq:condition1}. Employing the Lang-Weil estimate \cite{Lang-Weil}, we have, uniformly any such $\YY$, \begin{equation}\label{eq:LWcodim2}
	\#\{\bxi\in (\BF_p)^d:f(\YY,\bxi)\equiv g(\YY,\bxi)\equiv 0\bmod p\}\ll p^{d-2}.
\end{equation} Here the implied constant depends only on the degree of $f,g$. 
Now we break the $d$-dimensional cube $\prod_{j=1}^{d}\left[1, \YY^{E_\sigma(j)}\right]$ into a union of closed integral squared-cubes of side-length $p$, the number of which being 
$$\ll\frac{\prod_{j=1}^d \YY^{E_\sigma(j)}}{p^d}+\frac{\prod_{j=1}^d \YY^{E_\sigma(j)}}{\min_{1\leqslant j\leqslant d}\YY^{E_\sigma(j)}}\leqslant \left(p^{-d}+(\log B)^{-A}\right)\prod_{j=1}^d \YY^{E_\sigma(j)}.$$ Applying \eqref{eq:LWcodim2} together with the Chinese remainder theorem, we have
$$\CK_{\sigma,\YY}(B;\{p\})\ll p^{d-2}\left(p^{-d}+(\log B)^{-A}\right)\prod_{j=1}^d \YY^{E_\sigma(j)}= \left(\frac{1}{p^2}+\frac{p^{d-2}}{(\log B)^A}\right)\prod_{i=1}^{r}Y_i^{a_i(\sigma)-1}.$$

We are now ready to estimate $\#\CK_{\sigma}^{(A)}(B;\{p\})$ as follows.
\begin{align*}
	\#\CK_{\sigma}^{(A)}(B;\{p\})&=\sum_{\substack{\YY\in\BZ_{>0}^r\\ \eqref{eq:condition1}\text{ holds}}}\CK_{\sigma,\YY}(B;\{p\})\\&\ll \left(\frac{1}{p^2}+\frac{p^{d-2}}{(\log B)^A}\right)\sum_{\substack{\YY\in\BZ_{>0}^r\\\max_{1\leqslant i\leqslant r}Y_i\leqslant B, \YY^{D_0(\sigma)}\leqslant B }}\prod_{i=1}^{r}Y_i^{a_i(\sigma)-1}\\ &\ll\frac{B(\log B)^{r-1}}{p^2}+p^{d-2}B(\log B)^{r-1-A},
\end{align*} by using Lemma \ref{le:sublemma}.
\qed

\section{Applications to counting rational points with local conditions}\label{se:application}
In this section we apply the sieving results (Theorems \ref{thm:Tamagawazero}, \ref{thm:keyOmega} and \ref{thm:Selbergsieve}) developed in \S\ref{se:equidistgeomsieve} to prove Theorems \ref{thm:ratptsfibration} and \ref{thm:primevalue}.
\subsection{Paucity of rational points arising from fibrations -- Tamagawa measure zero cases}\label{se:appTamagawazero}
We begin by recalling the following general theorem scattered in the literature \cite{Cohen, Serre, BBL,Browning-Loughran,Loughran-Smeets} measuring the local contribution of rational points arising from fibrations.
\begin{theorem}\label{thm:locsolfibre}
Let $f_0:Y_1\to Y_2$ be a proper dominant morphism between smooth proper geometrically integral varieties over $k$.  Let $\widetilde{f}_0:\CY_1\to\CY_2$ be an extension of $f_0$ between proper regular  integral models of $Y_1,Y_2$ over $\CO_S$ where $S$ is a finite set of places containing $\infty_k$ such that $\CY_2(\CO_{\nu})\neq\varnothing$ for every $\nu\notin S$. 
\begin{enumerate}[label=(\roman*)]
	\item  Assume that $f_0$ is generically finite of degree $>1$. Then there exist a set $\CP(f_0)$ of places disjoint from $S$ of positive density and $c(f_0)\in \mathopen]0,1[$ such that for every $\nu\in \CP(f_0)$,  \begin{equation}\label{eq:est0}
		\#\left(\widetilde{f}_0(\CY_1(\BF_\nu))\right)\leqslant c(f_0)\#\CY_2(\BF_\nu).
	\end{equation}
	\item (D. Loughran \emph{et al}) Assume that the generic fibre of $f_0$ is geometrically integral\footnote{This implies that the generic fibre is smooth by the generic smoothness criterion \cite[III. Corollary 10.7]{Hartshorne}.} and that there exists at least one non-pseudo-split fibre above $Y_2^{(1)}$.  Then for every $\nu\not\in S$, we have \begin{equation}\label{eq:esta}
	1-\frac{\#\left(\widetilde{f}_0(\CY_1(\CO_\nu))\bmod \mathfrak{m}_\nu^2\right)}{\#\CY_2(\CO_\nu/\mathfrak{m}_\nu^2)}\ll \frac{1}{\#\BF_\nu}.
\end{equation} And there exists $\varDelta(f_0)>0$ such that \begin{equation}\label{eq:est1}
			\sum_{\substack{\nu\not\in S\\ \#\BF_\nu\leqslant N}}\left(1-\frac{\#\left(\widetilde{f}_0(\CY_1(\CO_\nu))\bmod \mathfrak{m}_\nu^2\right)}{\#\CY_2(\CO_\nu/\mathfrak{m}_\nu^2)}\right)\log \#\BF_\nu =\varDelta(f_0) \log N+O(1),
		\end{equation}
		\begin{equation}\label{eq:est2}
			\prod_{\substack{\nu\not\in S\\ \#\BF_\nu\leqslant N}}\frac{\#\CY_2(\CO_\nu/\mathfrak{m}_\nu^2)}{\#\left(\widetilde{f}_0(\CY_1(\CO_\nu))\bmod \mathfrak{m}_\nu^2\right)}\asymp (\log B)^{\varDelta(f_0)}.
		\end{equation}
\end{enumerate}
\end{theorem}
\begin{proof}
	Part (i) follows from the effective Hilbert irreducibility theorem (cf. \cite[Theorem 5.1, Lemma 5.2]{Cohen} and \cite[\S3.4]{Serre}). 
	Part (ii) is \cite[Corollary 2.4, Proposition 3.10]{Browning-Loughran} and \cite[Theorem 2.8]{Loughran-Smeets}.
\end{proof}
\begin{corollary}\label{co:paucityfibra}
	Let $V/k$ be as in Setting \ref{hyp:almost-Fano} and satisfy Hypothesis \ref{hyp:deltaB}. Let $f:Y\to V$ be a proper dominant morphism from $Y$ a smooth proper geometrically integral variety over $k$. Assume that either 
	\begin{enumerate}[label=(\roman*)]
			\item  $f$ is generically finite of degree $>1$; or 
				\item $f$ has geometrically integral generic fibre, and there exists at least one non-pseudo-split fibre above $V^{(1)}$. 
	\end{enumerate}
Recall the counting function \eqref{eq:countingDelta}. Then  as $B\to\infty$, $$\CN_{V}(f(Y(\RA_k));B)=o(1).$$
\end{corollary}
\begin{proof}
		We may assume that $Y(\RA_k)\neq\varnothing$ otherwise the estimate is trivial. For each $\nu\in\mathfrak{M}(k)$, the set $f(Y(k_\nu))$ is measurable and $\omega_\nu^V$-continuous. This follows from the theorems of Tarski--Seidenberg (for the archimedean places) and Marcintyre (the non-archimedean places) (cf. the proof of \cite[Lemma 3.9]{BBL}, which is stated for $V=\BA^n$ or $\BP^n$, but the same argument works for the general case where the base is an arbitrary algebraic variety). Moreover, as every smooth point admits an open neighbourhood (since the non-smooth locus is proper and Zariski closed), this shows that the family $(f(Y(k_\nu))\subset V(k_\nu))_{\nu\in\mathfrak{M}(k)}$ satisfies \eqref{eq:posmeas}. 
		
		Now we take a proper model $\widetilde{f}:\CY\to\CV$ of $f$ over $\CO_S$ where $S\subset \mathfrak{M}(k)$ is finite and contains $\infty_k$, such that $\CY,\CV$ are smooth over $\CO_S$, and that $(\omega_\nu^V)_{\nu\notin S}$ is the family of model measures attached to $K_{\CV}^{-1}\to\CV$.
		Recall the map \eqref{eq:redmodpk}. For every $\nu\notin S$ we take 
\begin{equation}\label{eq:case1}
			\Omega_\nu:=\begin{cases}
				\operatorname{Mod}_{\nu,1}^{-1}(\widetilde{f}(\CY(\BF_\nu))) & \text{ case (i)};\\ 	\operatorname{Mod}_{\nu,2}^{-1}\left(\widetilde{f}(\CY(\CO_\nu))\bmod \mathfrak{m}_\nu^2\right) & \text{ case (ii)}.
			\end{cases}
\end{equation}
		By Hensel's lemma (e.g. \cite[Lemma 2.1]{Browning-Loughran}) and  by  Lemma \ref{le:modelmeasurecomp}, we have for $\nu\notin S$,
		\begin{equation}\label{eq:case2}
			\frac{\omega_\nu^V(\Omega_{\nu})}{\omega_\nu^V(V(k_\nu))}=\begin{cases}
			\frac{\#\left(\widetilde{f}(\CY(\BF_\nu))\right)}{\#\CV(\BF_\nu)} & \text{ case (i)};\\
			\frac{\#\left(\widetilde{f}(\CY(\CO_\nu))\bmod \mathfrak{m}_\nu^2\right)}{ \#\CV(\CO_\nu/\mathfrak{m}_\nu^2)}& \text{ case (ii)}.
		\end{cases}
		\end{equation}
		It follows from Theorem \ref{thm:locsolfibre} (i) for case (i) and from Theorem \ref{thm:locsolfibre} (ii) for case (ii) that the infinite product $\prod_{\nu\in\mathfrak{M}(k)}\frac{\omega_\nu^V(\Omega_{\nu})}{\omega_\nu^V(V(k_\nu))}$ diverges to zero in both cases. 
		Since in all cases we have, whenever $\nu\notin S$,
		$$f(Y(k_\nu))=\widetilde{f}(\CY(\CO_{\nu}))\subset \Omega_{\nu}.$$ 
		On taking $\Omega_{\nu}=V(k_{\nu})$ for $\nu\in S$, we conclude by Theorem \ref{thm:Tamagawazero} that \begin{equation*}
			\CN_{V}(f(Y(\RA_k));B)\leqslant\CN_V\left(\CE[(\Omega_{\nu})];B\right)=o(1).\qedhere
		\end{equation*}
\end{proof}

The goal in the rest of this section is to prove Theorem \ref{thm:ratptsfibration} (1) \& (2a),  as an effective version of Corollary \ref{co:paucityfibra} where the base variety is any smooth proper split toric variety $X$ over $\BQ$ with globally generated $K_X^{-1}$.  For this we fix  $\widetilde{f}:\CY\to \CX$ a proper model of the morphism $f$ over $\BZ_S$ where $S$ a finite set of places containing $\BR$.  
\begin{proof}[Proof of Theorem \ref{thm:ratptsfibration} (1)]	
 Let the set of primes $\CP(f)$ and the constant $c(f)\in\mathopen ]0,1[$ be as in Theorem \ref{thm:locsolfibre} (i).  For every prime $p\in\CP(f)$, let the open-closed subset $\Omega_p\subset\CX(\BZ_p)$ be defined by \eqref{eq:case1}.
Then by Theorem \ref{thm:nonarchTmeas} and \eqref{eq:case2}, we have
$$\frac{\omega_{\operatorname{tor},p}^X(\Omega_p^c)}{\omega_{\operatorname{tor},p}^X(\Omega_p)}=\frac{\#\left(\CX(\BF_p)\setminus \widetilde{f}(\CY(\BF_p))\right)}{\# \widetilde{f}(\CY(\BF_p))}\geqslant c(f)^{-1}-1>0.$$
Hence  for every $N\geqslant 1$,  function $G$ \eqref{eq:GN} satisfies
\begin{align*}
	G(N)&=\sum_{\substack{j<N\\p\mid j\Rightarrow p\in \CP(f)}}\prod_{p\mid j }\frac{\omega_{\operatorname{tor},p}^X(\Omega_p^c)}{\omega_{\operatorname{tor},p}^X(\Omega_p)}\geqslant \sum_{\substack{j<N\\p\mid j\Rightarrow p\in \CP(f)}}(c(f)^{-1}-1)^{\#\{p:p\mid j\}}\gg_{\varepsilon} N^{1-\varepsilon}.
\end{align*}
On letting $\Omega_p=X(\BQ_p)$ for all prime $p\notin S\cup \CP(f)$, and applying Corollaries \ref{thm:effectiveEE} and \ref{cor:Selberg} (with $\Delta(B)=\delta_{(V(k)\setminus M)_{\leqslant B}}$), we obtain 
\begin{align*}
	\CN_{\operatorname{loc}}(f;B)&\ll B(\log B)^{r-1}\CN_X(\CE_S[\Omega];B)\\ & \ll_\varepsilon \frac{B(\log B)^{r-1}}{N^{1-\varepsilon}}+N^{2(r+2\dim X+1)+\varepsilon}B(\log B)^{r-\frac{3}{2}+\varepsilon}.
\end{align*}
So it suffices to take $N=(\log B)^\lambda$ with $0<\lambda<(4(r+2\dim X+1))^{-1}$ and fix $\varepsilon$ small enough to conclude.
\end{proof}

\begin{proof}[Proof of Theorem \ref{thm:ratptsfibration} (2a)]
For every prime $p\notin S$, let the open-closed set $\Omega_p\subset\CX(\BZ_p)$ be defined by \eqref{eq:case1}. Then by Theorem \ref{thm:nonarchTmeas} and \eqref{eq:case2},
$$\frac{\omega_{\operatorname{tor},p}^X(\Omega_p^c)}{\omega_{\operatorname{tor},p}^X(\Omega_p)} =\frac{\#\left(\CX(\BZ/p^2\BZ)\setminus\left(\widetilde{f}(\CY(\BZ_p))\bmod p^2\right)\right)}{\#\left(\widetilde{f}(\CY(\BZ_p))\bmod p^2\right)}.$$
By \eqref{eq:esta} in Theorem \ref{thm:locsolfibre} (ii), we have  $$\#\CX(\BZ/p^2\BZ)=\#\left(\widetilde{f}(\CY(\BZ_p))\bmod p^2\right)\left(1+O\left(\frac{1}{p}\right)\right).$$
It follows from \eqref{eq:est1} that there exists a constant $\varDelta(f)>0$ such that  \begin{align*}
	&\sum_{\substack{p\leqslant N,p\not\in S}}\frac{\omega_{\operatorname{tor},p}^X(\Omega_p^c)}{\omega_{\operatorname{tor},p}^X(\Omega_p)}\log p\\ =&\sum_{\substack{p\leqslant N,p\not\in S}}\frac{\#\left(\CX(\BZ/p^2\BZ)\setminus\left(\widetilde{f}(\CY(\BZ_p))\bmod p^2\right)\right)}{\#\CX(\BZ/p^2\BZ)} \log p +O\left(\sum_{p<N}\frac{\log p}{p^2}\right)\\= &\varDelta(f)\log N+O(1).
\end{align*} 
Moreover, by \eqref{eq:est2}, we have $\Omega_{p}\neq X(\BQ_{p})$ for $p\notin S$ and  the function $G$ \eqref{eq:GN} satisfies
\begin{align*}
	G(N)=\sum_{\substack{j<N\\ p\mid j\Rightarrow p\not\in S}}\prod_{p\mid j}\frac{\omega_{\operatorname{tor},p}^X(\Omega_p^c)}{\omega_{\operatorname{tor},p}^X(\Omega_p)}&\leqslant\prod_{p\leqslant N,p\not\in S}\left(1+\frac{\omega_{\operatorname{tor},p}^X(\Omega_p^c)}{\omega_{\operatorname{tor},p}^X(\Omega_p)}\right)\\ &=\prod_{p\leqslant N,p\not\in S}\frac{\#\CX(\BZ/p^2\BZ)}{\#\left(\widetilde{f}(\CY(\BZ_p))\bmod p^2\right)}\asymp (\log N)^{\varDelta(f)}.
\end{align*}
Based on these, an application of \cite[Theorem 1.1]{I-K} shows that there exists $c>0$ such that
$$G(N)=c(\log N)^{\varDelta(f)}+O\left((\log N)^{\varDelta(f)-1}\right).$$ So it follows from Corollaries \ref{thm:effectiveEE} and \ref{cor:Selberg} that in this case 
\begin{align*}
	\CN_{\operatorname{loc}}(f;B)&\ll B(\log B)^{r-1}\CN_X(\CE_S[\Omega];B)\\ &\ll_\varepsilon \frac{B(\log B)^{r-1}}{(\log N)^{\varDelta(f)}}+N^{4(r+2\dim X)+2+\varepsilon}B(\log B)^{r-\frac{3}{2}+\varepsilon}.\end{align*}
We take $N=(\log B)^\lambda$ where $0<\lambda<(8(r+2\dim X+1))^{-1}$ and fix $\varepsilon$ small enough to finish the proof. 
\end{proof}

\subsection{Density of rational points which are $S$-integral or have prime sectional values}
We first present a proof of a generalised Lang-Weil estimate for arbitrary varieties (Proposition \ref{co:Lang-Weil}) which is of independent interest (cf. a consideration of Serre \cite[\S3.6]{Serre} in estimating rational points in thin sets using the large sieve).

Let $L/K$ be an extension of number fields. For any $\mathfrak{p}\in\operatorname{Spec}(\CO_K)$, let us define
\begin{equation}\label{eq:MLK}
	\CM_{L/K}(\mathfrak{p}):=\{\mathfrak{q}\in\operatorname{Spec}(\CO_L):\mathfrak{q}\mid\mathfrak{p},[\BF_{\mathfrak{q}}:\BF_{\mathfrak{p}}]=1\},
\end{equation}
where we write $\BF_{\mathfrak{p}}:=\CO_K/\mathfrak{p}$ for the residue field and similarly $\BF_{\mathfrak{q}}:=\CO_L/\mathfrak{q}$. For instance, if $L/K$ is Galois and is unramified at $\mathfrak{p}$, then $\CM_{L/K}(\mathfrak{p})\neq\varnothing$ if and only if $\mathfrak{p}$ splits completely in $L$, in which case $\#\CM_{L/K}(\mathfrak{p})=[L:K]$.
\begin{proposition}\label{co:Lang-Weil}
	Let $Y$ be an algebraic variety (i.e. a separated reduced scheme of finite type) over a number field $k$ and let $\CY$ be an integral model of $Y$ over $\CO_{k,S}$, where $S\subset \mathfrak{M}(k)$ is a finite set of places containing the archimedean ones. Then there exists a finite collection $(k_j)_{j\in J}$ of finite extensions of $k$ such that for any $\mathfrak{p}\in \operatorname{Spec}(\CO_{k,S})$,
	\begin{equation}\label{eq:Lang-Weil}
		\#\CY(\BF_{\mathfrak{p}})=\left(\sum_{\substack{j\in J}}\#\CM_{k_j/k}(\mathfrak{p})\right)(\#\BF_{\mathfrak{p}})^{\dim Y}+O((\#\BF_{\mathfrak{p}})^{\dim Y-\frac{1}{2}}),
	\end{equation}
	where the implied constant depends only on $\CY$.
\end{proposition}

\begin{proof}[Proof of Proposition \ref{co:Lang-Weil}]
	Let $(Y_j)_{j\in J}$ be the irreducible components of $Y$ over $k$ of the maximal dimension. Upon considering the normal locus (which is open dense), by \cite[Theorem~7.7.1 (i)]{Poonen}, we may assume that each $Y_j$ is normal. Let $k_j$ be the algebraic closure of $k$ in the function field of $Y_j$. (This is the so-called \emph{constant field} of $Y_j$, cf. \cite[\S2.2.4]{Poonen}). By \cite[Propositions~2.2.19, 2.2.22]{Poonen}, $Y_j$ is geometrically irreducible over $k_j$.
	Every $Y_j$ spreads out to a model $\CY_j$ over $\CO_{k_j,S_j^\prime}$ with geometrically irreducible fibres, where $S_j^\prime$ is a finite set of primes of $k_j$ containing all the primes lying over $S$.
	
	Now for every $\mathfrak{p}\notin S$ which does not lift to any prime in $\cup_{j\in J}S_j^\prime$, the existence of a $\BF_{\mathfrak{p}}$-point of $\CY_j$ (over $\CO_{k,S}$) is equivalent to the existence of a $\BF_{\mathfrak{q}}$-point of $\CY_j$ (over $\CO_{k_j,S_j^\prime}$) for a certain $\mathfrak{q}\in\CM_{k_j/k}(\mathfrak{p})$. Then Equation \eqref{eq:Lang-Weil} follows from applying the relative version of the Lang-Weil estimate \cite[Theorem~7.7.1 (iii)]{Poonen} to each $\CY_j$ and their intersections. Upon enlarging the implied constant, Equation \eqref{eq:Lang-Weil} still holds for the primes of $\CO_{k,S}$ lying below $\cup_{j\in J}S_j^\prime$.
\end{proof}
\begin{remark}
	In fact, for every $j\in J$, for every $\mathfrak{p}$ of $\CO_{k,S}$, $\#\CM_{k_j/k}(\mathfrak{p})$ is exactly the number of irreducible components of $\CY_j \times_{\operatorname{Spec}(\CO_{k,S})} \operatorname{Spec}(\BF_{\mathfrak{p}})$ which are geometrically irreducible. However, the extension $k_j/k$ need not be Galois. \end{remark}

\begin{proof}[Proof of Theorem \ref{thm:primevalue}]
We fix $S_0$ a finite set places such that $\mathbf{X}\times_{\BZ}\BZ_{S_0}\simeq\CX\times_{\BZ}\BZ_{S_0}$, where as before $\CX$ is the canonical toric integral model of $X$ over $\BZ$. 

\emph{Estimation of $\CN_{(1)}(Z;B)$.}
Upon enlarging $S$ if necessary, we may assume $S_0\subset S$, and that  for every $p\notin S$, we have \begin{equation}\label{eq:ZFpXFp}
	\#\mathbf{Z}(\BF_p)<\#\CX(\BF_p). 
\end{equation}
Recall the map \eqref{eq:redmodpk}. 
For every prime $p\notin S$, let $$\Omega_p:=\CX(\BZ_p)\setminus \operatorname{Mod}_{p,1}^{-1}(\mathbf{Z}(\BF_p)).$$ Then it is clear that
$$P\in\Omega_p\Longleftrightarrow \widetilde{P}\bmod p\notin\mathbf{Z}(\BF_p).$$
Recall the counting function \eqref{eq:countingDelta} and the notation in \S\ref{se:Selbersieve}.
We then have  \begin{equation}\label{eq:splitpN1} \CN_{(1)}(Z;B)\ll  B(\log B)^{r-1}\CN_X(\CE^{\CX}_{(\mathfrak{M}(\BQ)\setminus S)}[(\Omega_\nu)];B).\end{equation}

Now we seek to apply the Selberg's sieve (Theorem \ref{thm:Selbergsieve}). For this we need to derive a lower bound estimate for the function $G$ \eqref{eq:GN}. 
Let $(k_j/\BQ)_{j\in J}$ be the collection of finite extensions for $Z$ in  Proposition \ref{co:Lang-Weil}. Let $L$ be the normal closure of the composite field $\prod_{j\in J}k_j$, and let $\CP_{S}$ be the set of primes not in $S$ which split completely in $L$. Then by the Chebotarev density theorem \cite[Corollary 13.6]{Neukirch}, $\CP_{S}$ has natural density $[L:\BQ]^{-1}$, and for every $p\in \CP_{S}$, $p$ also splits completely in each $k_j$, and hence $\#\CM_{k_j/\BQ}(p)=[k_j:\BQ]$.

We define the multiplicative arithmetic function $f$ supported at square-free integers by $$f(p):=\begin{cases}
	\frac{\omega_{\operatorname{tor},p}^X(\Omega_p^c)}{\omega_{\operatorname{tor},p}^X(\Omega_p)}&\text{ if } p\in \CP_{S};\\ 0 &\text{ otherwise}.
\end{cases}$$
Note that $\Omega_p\neq X(\BQ_p)$ for every $p\in \CP_{S}$ and so $0<f(p)<1$ for $p\in\CP_{S}$. For every $N\geqslant 2$ such that $p\in S\Rightarrow p<N$, we clearly have
$$G(N)=\sum_{\substack{j<N,j\mid\Pi(\CP_{S}(N))}}\prod_{\substack{p\mid j}}\frac{\omega_{\operatorname{tor},p}^X(\Omega_p^c)}{\omega_{\operatorname{tor},p}^X(\Omega_p)}= \sum_{\substack{j<N}}f(j).$$
Now by Theorem \ref{thm:nonarchTmeas} and Proposition \ref{co:Lang-Weil}, we obtain
\begin{equation}\label{eq:wirsing1}
	\begin{split}
		\sum_{p<N}f(p)\log p&=\sum_{p\in\CP_S, p<N}\frac{\omega_{\operatorname{tor},p}^X(\Omega_p^c)}{\omega_{\operatorname{tor},p}^X(\Omega_p)}\log p\\ &=\sum_{p\in\CP_S,p<N}\frac{\#\mathbf{Z}(\BF_p)}{\#(\CX(\BF_p)\setminus \mathbf{Z}(\BF_p))}\log p\\ &=\sum_{p\in\CP_S,p<N}\left(\sum_{j\in J}\#\CM_{k_j/\BQ}(p)(1+O(p^{-\frac{1}{2}}))\right)\frac{\log p}{p}\\ &=\sum_{j\in J}[k_j:\BQ]\sum_{p\in\CP_S,p<N}\left(\frac{\log p}{p}+O\left(\frac{\log p}{p^{\frac{3}{2}}}\right)\right)\\ &=\theta\log N+O(1),
	\end{split}
\end{equation}
where we define \begin{equation}\label{eq:theta}
	\theta:=\frac{\sum_{j\in J}[k_j:\BQ]}{[L:\BQ]},
\end{equation} and for the last line we have invoked Mertens' theorem (for primes in Chebotarev sets, cf. \cite[Theorem A]{A-K-K} and \cite[Theorem 8.3]{B-F}).
On the other hand, we also have
\begin{equation}\label{eq:wirsing2}
	\begin{split}
		\sum_{j<N}f(j)\leqslant\prod_{p<N}\left(1+f(p)\right)&=\prod_{p\in\CP_S,p<N}\left(1-\frac{\#\mathbf{Z}(\BF_p)}{\#\CX(\BF_p)}\right)^{-1}\\&\ll \prod_{p\in\CP_S,p<N}\left(1-\frac{\sum_{j\in J}\#\CM_{k_j/\BQ}(p)}{p}\right)^{-1}\ll (\log N)^{\theta}. \end{split}
\end{equation}
Consequently, combining \eqref{eq:wirsing1} and \eqref{eq:wirsing2} and applying \cite[Theorem 1.1]{I-K}, we obtain that there exists $c>0$ such that
$$\sum_{\substack{j<N}}f(j)= c(\log N)^{\theta}+O\left((\log N)^{\theta-1}\right),$$ and hence $$G(N)\gg (\log N)^{\theta}.$$
On applying Corollary \ref{cor:Selberg} combined with Corollary \ref{thm:effectiveEE}, we obtain
\begin{equation}\label{eq:step1}
	\CN_X(\CE^{\CX}_{(\mathfrak{M}(\BQ)\setminus S)}[(\Omega_\nu)];B)\ll_\varepsilon \frac{1}{(\log N)^{\theta}}+\frac{N^{2(r+2\dim X+1)+\varepsilon}}{(\log B)^{\frac{1}{2}-\varepsilon}}.
\end{equation}

\emph{Estimation of $\CN_{(2)}(Z;B)$.} 
For each prime number $p$, we define
$$\CN_{(2)}^{\{p\}}(Z;B):=\#\{P\in\CT_{O}(\BQ):H_{\operatorname{tor}}(P)\leqslant B,|s(\widehat{P})|=p\}.$$ Then clearly 
\begin{equation}\label{eq:splitpN2}
	\CN_{(2)}(Z;B)\ll B(\log B)^{r-1}\CN_X(\CE^{\CX}_{(\mathfrak{M}(\BQ)\setminus S)}[(\Omega_\nu)];B)+\sum_{p<N}\CN_{(2)}^{\{p\}}(Z;B),
\end{equation} where we take $S\supset S_0$ such that \eqref{eq:ZFpXFp} holds.

Let $Z_0:=\pi^{-1}(Z)\subset X_0$, where $X_0$ is the principle universal torsor of $X$ with the canonical $\BZ$-model $\CX_0$. Then $Z_0$ is a reduced effective divisor of $X_0$. 
Note that $\Pic(\CX_0)$ is trivial, and the ring of regular functions on $\CX_0$ (i.e. the Cox ring) is isomorphic to the polynomial ring in $n$-variables with $\BZ$-coefficients. It follows that there exist an integral global section $\widetilde{s}$ of $\CO_{\CX_0}$ and $m\in\BN$, such that $Z_0=(\widetilde{s}=0)\times_{\BZ}\BQ$ and that for every $P\in X(\BQ)$, $$s(\widehat{P})=m\widetilde{s}(\XX(P)),$$ where we recall that $\XX(P)$ denotes a lift of $P$ to $\CX_0(\BZ)$. 

Now for every prime $p$, we define $$\phi_{p}^+:=m\widetilde{s}+p,\quad \phi_{p}^-:=m\widetilde{s}-p.$$ Note that they are of the same degree as $\widetilde{s}$.  Recall the set \eqref{eq:CAphi}.
Now by Proposition \ref{prop:subvar}, uniformly for any prime $p$, $$\max\left(\#\CA_{\phi_{p}^+}(B),\#\CA_{\phi_{p}^-}(B)\right)\ll (\deg \widetilde{s})B(\log B)^{r-2}\log\log B.$$
We hence obtain 
\begin{equation}\label{eq:step2}
	\begin{split}
		\sum_{p<N}\CN_{(2)}^{\{p\}}(Z;B)&\ll \sum_{p<N}\#\{\XX\in\CA(B):m|\widetilde{s}(\XX)|=p\}\\ &\ll\sum_{p<N}\left( \#\CA_{\phi_{p}^+}(B)+\#\CA_{\phi_{p}^-}(B)\right)\\ &\ll \frac{N}{\log N}B(\log B)^{r-2}\log\log B,
	\end{split}
\end{equation}
where the implied constant is allowed to depend on $s$. 

Now going back to \eqref{eq:splitpN1} \eqref{eq:splitpN2} and combining \eqref{eq:step1} \eqref{eq:step2},  we obtain that for $i=1,2$,
\begin{align*}
	\CN_{(i)}(Z;B) &\ll_\varepsilon  \frac{B(\log B)^{r-1}}{(\log N)^{\theta}}+N^{2(r+2\dim X+1)+\varepsilon}B(\log B)^{r-\frac{3}{2}+\varepsilon}+\frac{N}{\log N}B(\log B)^{r-2}\log\log B\\ &\ll \frac{B(\log B)^{r-1}}{(\log N)^{\theta}}+N^{2(r+2\dim X+1)+\varepsilon}B(\log B)^{r-\frac{3}{2}+\varepsilon}
\end{align*}
On taking $N=(\log B)^\lambda$ with $0<\lambda<(4(r+2\dim X+1))^{-1}$ and $\varepsilon$ small enough, we conclude that Theorem \ref{thm:primevalue} holds with \eqref{eq:theta}. \end{proof}

\subsection{Applications of the geometric sieve for toric varieties}\label{se:appgeomsieve}
The goal of this subsection is to prove Principle \ref{prin:purity} and Theorem \ref{thm:ratptsfibration} (2b) for any fixed smooth proper split toric variety $X$ over $\BQ$ with globally generated $K_X^{-1}$, by means of Theorem \ref{thm:keyOmega}.

\subsubsection{Proof of Principle \ref{prin:purity} -- Purity of equidistribution for $X$}
	By the purity of the Brauer group \cite[\S6]{Grothendieck}, we have $W(\RA_\BQ)^{\operatorname{Br}}=W(\RA_\BQ)$. Let $Z\subset X$ be Zariski closed of codimension at least two, and let $W$ be the open subset $X\setminus Z$. Let $\CZ=\overline{Z}\subset \CX$ and $\CW:=\CX\setminus \CZ$.
 In view of Theorem \ref{thm:portmanteau}, we want to deduce an asymptotic formula for the counting function $\CN_W(\CF;B)$ \eqref{eq:countingDelta} with $\Delta(B)=\delta_{(W(k)\setminus M)_{\leqslant B}}$ associated to any fixed adelic subset of the form $$\CF=\CF_\infty\times\prod_{p\mid l}\CF_p\times \prod_{p\nmid l} \CW(\BZ_p)\subset W(\RA_\BQ),$$ where $l$ is sufficiently large, $\CF_\infty\subset W(\BR)$ is $\omega_{\operatorname{tor},\infty}^X$-continuous with $\omega_{\operatorname{tor},\infty}^X(\CF_\infty)>0$ and each non-archimedean $\CF_p \subset W(\BQ_p)$ is non-empty and open-closed, because such sets form a topological basis of $W(\RA_\BQ)$.

Clearly $\omega_{\operatorname{tor},p}^X(\CF_p)>0$ for every $p\mid l$. By Lemma \ref{le:preparation}, Theorem \ref{thm:nonarchTmeas}, Lemma \ref{le:modelmeasurecomp} and the Lang--Weil estimate \cite{Lang-Weil}, we see that, uniformly for $p\nmid l$,  $$\omega_{\operatorname{tor},p}^X(\CW(\BZ_p))=\frac{\#\CW(\BF_p)}{p^{\dim X}}=1+O\left(p^{-\frac{1}{2}}\right).$$ So $\omega_{\operatorname{tor},p}^X(\CW(\BZ_p))>0$ provided $p$ is sufficiently large.  So upon enlarging $l$,  we may assume that the family $\{\CF_\infty,(\CF_p)_{p\mid l},(\CW(\BZ_{p}))_{p\nmid l}\}$ satisfies condition \eqref{eq:posmeas} of Theorem \ref{thm:keyOmega}.

We now verify condition \textbf{(GS)}.
Lemma \ref{le:preparation} shows that for each prime $p<\infty$, every $P\in X(\BQ_p)$ which lifts to $\widetilde{P}\in \CX(\BZ_p)$ with $\widetilde{P}\not\in \CW(\BZ_p)$ satisfies $\widetilde{P}\bmod p\in\CZ(\BF_p)$. So condition \textbf{(GS)} follows from Theorem \ref{thm:maingeomsieve}. 

The variety $X$ satisfies Principle \ref{prin:Manin-Peyre} thanks to Theorem \ref{thm:mainequidist}. It follows from Theorem \ref{thm:keyOmega} that 
\begin{equation}\label{eq:CNWB}
\begin{split}
	\CN_W(\CF;B)&=\CN_X\left(\CF_\infty\times\prod_{p\mid l}\CF_p\times \prod_{p\nmid l} \CW(\BZ_p);B\right)\\&\sim \omega^X_{\operatorname{tor}}\left(\CF_\infty\times\prod_{p\mid l}\CF_p\times \prod_{p\nmid l} \CW(\BZ_p)\right)=\omega_{\operatorname{tor}}^X|_W\left(\CF_\infty\times\prod_{p\mid l}\CF_p\times \prod_{p\nmid l} \CW(\BZ_p)\right).
\end{split}\end{equation}
Thus Principle \ref{prin:purity} for $X$ is proved.
\qed

\subsubsection{Proportion of everywhere locally soluble fibres}
We first recall the following result about the surjectivity of integral points ``outside of codimension two''.
\begin{proposition}[Loughran et al. \cite{BBL} Corollary 3.7, \cite{Loughran-Smeets} Proposition 4.1]\label{prop:surjectivity}
	Let $f_0:Y_1\to Y_2$ be a proper dominant morphism between smooth proper geometrically integral varieties over $k$.  Assume that the generic fibre is geometrically integral and the fibre over each point of $Y_2^{(1)}$ is pseudo-split. Then there exist a Zariski closed subset $Z\subset Y_2$ of codimension at least two satisfying the following. Let $S$ be a finite set of places containing $\infty_k$ such that $f_0$ extends to a proper model $\widetilde{f}_0:\CY_1\to\CY_2$ over $\CO_S$. Let $\CZ$ be the Zariski closure of $Z$ in $\CY_2$. Then there exists a finite set of places $S'$ containing $S$ such that  that for every $\nu\not\in S'$, we have 
	\begin{equation}\label{eq:surj}
		\left(\CY_2\setminus\CZ\right)(\CO_\nu)\subset f_0(Y_1(k_\nu)).
	\end{equation}
\end{proposition}

\begin{proof}[Proof of Theorem \ref{thm:ratptsfibration} (2b)]
Now Returning to $f:Y\to X$ as in the setting of Theorem \ref{thm:ratptsfibration} (2b). We choose a proper integral model $\widetilde{f}:\CY\to\CX$ over $\BZ_S$ extending $f$.  Applying Proposition \ref{prop:surjectivity}, there exists a Zariski closed subset $Z\subset X$ of codimension at least two and $S'$ a finite set of places containing $S$ such that $\CW(\BZ_p)\subset f(Y(\BQ_p))=\widetilde{f}(\CY(\BZ_p))$ for every prime $p\notin S'$, where $\CZ$ is the Zariski closure of $Z$ in $\CX$ and $\CW :=\CX\setminus\CZ$. Then for any such $p$ and $P\in X(\BQ_p)\setminus f(Y(\BQ_p))$,  by Lemma \ref{le:preparation}, the lift $\widetilde{P}\in\CX(\BZ)$ satisfies $\widetilde{P}\bmod p\in\CZ(\BF_p)$. By Theorem \ref{thm:maingeomsieve}, condition \textbf{(GS)} is satisfied for the family $(f(Y(\BQ_\nu))\subset X(\BQ_\nu))_{\nu\in\mathfrak{M}(\BQ)}$. It also satisfies \eqref{eq:posmeas} as we have seen in the proof of Corollary \ref{co:paucityfibra}.	So it follows directly from  Theorem \ref{thm:keyOmega} that $\omega_{\operatorname{tor}}^X\left(\prod_{\nu\in\mathfrak{M}(\BQ)}f(Y(\BQ_\nu))\right)>0$ and
\begin{equation}\label{eq:CNfB}
	\begin{split}
		\CN_{\operatorname{loc}}(f;B)&=\alpha(X)\CN_X\left(\prod_{\nu\in\mathfrak{M}(\BQ)}f(Y(\BQ_\nu));B\right)B(\log B)^{r-1}\\
		&\sim \alpha(X)\omega_{\operatorname{tor}}^X\left(\prod_{\nu\in\mathfrak{M}(\BQ)}f(Y(\BQ_\nu))\right)B(\log B)^{r-1}\\ &=\alpha(X)\omega_{\operatorname{tor}}^X\left(f(Y(\RA_\BQ))\right)B(\log B)^{r-1}.
	\end{split}
\end{equation}
	This finishes the proof of Theorem \ref{thm:ratptsfibration} (2b). 
\end{proof}

\begin{remark}\label{rmk:effectivity} 
	Adapting the strategy in \cite[\S3]{Cao-Huang2} based on Theorems \ref{thm:mainequidist} and  \ref{thm:maingeomsieve}, the error terms in the asymptotic formulas \eqref{eq:CNWB} and \eqref{eq:CNfB} can all be made explicit with some extra work. \end{remark}

\end{document}